\documentclass{amsart}
\usepackage{amsthm}
\usepackage{amsmath}
\usepackage{amsfonts}
\usepackage[utf8]{inputenc}
\usepackage{graphicx}
\usepackage{amsthm} 
\usepackage{amsmath}
\usepackage{amsfonts}
\usepackage{tikz}
\usepackage[capitalize, noabbrev]{cleveref}
\usepackage{tikz-cd}
\usepackage{amssymb}
\usepackage{filecontents}
\usetikzlibrary{3d,calc,intersections}
\usepackage{graphicx}
\usepackage{listings}
\usepackage{pgfplots}
\pgfplotsset{compat=1.15}
\usepackage{mathrsfs}
\usepackage{tikz-3dplot}
\usepackage{tikz,tikz-3dplot} 
\usetikzlibrary{arrows}

\usepackage{subcaption}

\newtheorem{theo}{Theorem}[section]

\newtheorem{cor}[theo]{Corollary}
\newtheorem{prop}[theo]{Proposition}
\newtheorem{lemma}[theo]{Lemma}
\theoremstyle{remark}{}
\newtheorem{rmk}[theo]{Remark}
\theoremstyle{definition}

\newtheorem{ex}[theo]{Example}
\newtheorem{defn}[theo]{Definition}
\newtheorem*{hurwitz}{Hurwitz conditions}
\newtheorem*{HauptMoi}{Refined Haupt's theorem}
\newtheorem*{GKM}{Gallo Kapovich Marden}
\newtheorem{obstruction}{Obstruction}

\newcommand{\Hol}{\mathrm{Hol}}
\newcommand{\C}{\mathbb C}
\newcommand{\R}{\mathbb R}
\newcommand{\id}{\mathrm{id}}
\newcommand{\Id}{\mathrm{Id}}
\newcommand{\homeo}[1]{\mathrm{Homeo}_+ (#1)}

\newcommand{\sw}{\mathrm{sw}}

\newcommand{\PSL}[1]{\mathrm{PSL}_2 (#1)}
\newcommand{\SL}[1]{\mathrm{SL}_2 (#1)}
\newcommand{\Mod }{\mathrm{Mod}}
\newcommand{\CP}{\mathbb{CP}^1}

\newcommand{\Hom}{\mathrm{Hom}}
\newcommand{\SO}{\mathrm{SO}_3(\R)}
\newcommand{\SU}{\mathrm{SU}(2)}
\newcommand{\Aut}{\mathrm{Aut}^+}
\newcommand{\K}{\mathrm{K}}
\newcommand{\A}[1]{\mathfrak A_{#1}}
\newcommand{\Sy}[1]{\mathfrak S_{#1}}
\newcommand{\Z}{\mathbb Z}
\newcommand{\D}{\mathrm D}
\newcommand{\Sp}{\mathrm{Sp}}

\title[Holonomy of complex projective structures]{Holonomy of complex projective structures on surfaces with prescribed branch data}
\author{Thomas Le Fils}
\date{}

\begin{document}
\begin{abstract}
We characterize the representations of the fundamental group of a closed surface to $\mathrm{PSL}_2(\mathbb C)$ that arise as the holonomy of a branched complex projective structure with fixed branch divisor. In particular, we compute the holonomies of the spherical metrics with prescribed integral conical angles and the holonomies of affine structures with fixed conical angles on closed surfaces.
\end{abstract}
\maketitle
%\begin{center}
%VERSION PR\' ELIMINAIRE
%\end{center}

\section{Introduction}\label{Introduction}
\subsection{Holonomy of branched projective structures}
For each $g\geqslant 0$ and $n \geqslant 0$, we fix $\Sigma_{g,n}$ a connected compact oriented surface of genus $g$ with $n$ boundary components and $\Gamma_{g,n}$ a fundamental group of $\Sigma_{g,n}$. We will be most interested in the case $n=0$ and we let $\Gamma= \Gamma_{g,0}$ and $\Sigma = \Sigma_{g,0}$.

A \textit{branched projective structure} on $\Sigma$ is the datum of a branched atlas with values in $\CP$ and transition maps in $\PSL \C$. By branched we mean that our charts have the local form $z\mapsto z^{n+1}$ for some integer $n\geqslant 0$, called the \emph{order} of the point corresponding to $0$ in this chart. A point of order at least $1$ is called a \textit{branch point}. The analytic continuation of one of the charts in the universal cover of $\Sigma$ yields a \emph{developing map} that is $\rho$-equivariant for a unique homomorphism $\rho : \Gamma\to \PSL \C$ called the \textit{holonomy} of the projective structure. The conjugacy class of this holonomy does not depend on the choice of the chart and thus gives a map defined on the set $\mathcal {BP}$ of branched projective structures on $\Sigma$
$$\mathrm{Hol} : \mathcal {BP}\to \Hom(\Gamma, \PSL\C )/\PSL\C. $$
The map $\Hol$ and in particular its restriction $\Hol_{|\mathcal P}$ to the set $\mathcal P$ of projective structures without any branch point has been extensively studied. For example, it is known that $\Hol_{|\mathcal P}$ is a local homeomorphism but not a covering map, see \cite{hejhal}. It is also known since Poincar\' e \cite{poincare}, see also \cite{Loray}, that the restriction of $\Hol$ to the subset of $\mathcal P$ of projective structures with a fixed underlying complex structure is injective.
 A celebrated theorem of Gallo, Kapovich and Marden \cite{GKM} characterizes the image of $\Hol_{|\mathcal P}$, see also \cite{Moi1}.
\begin{GKM}
Suppose $g\geqslant 2$.
The image of $\mathcal P$ by $\Hol$ is the set of nonelementary representations that lift to $\SL \C$. The image of the set of projective structures with a single branch point, of order $1$, is the set of nonelementary representations that do not lift to $\SL \C$.
\end{GKM}
In this article, we adress a problem raised in \cite[Problem 12.6.1]{GKM}: ``make precise and optimize the connection between branching divisor and monodromy''. 
The set $\mathcal {BP}$ is naturally stratified by the sets $\mathcal P(n_1, \ldots, n_k)$ of branched projective structure with $k$ branch points of orders $n_1, \ldots, n_k$.
The aim of this article is thus to characterize the image of the strata $\mathcal P(n_1, \ldots, n_k)$ by $\Hol$.

%
% The developing map is locally injective save from a $\Gamma$-invariant discrete set, that is mapped by the covering map $\tilde \Sigma\to \Sigma$ to a finite set $\{p_1, \ldots, p_k\}$. These points $p_i$ are called the {conical points}, or {branch points}. Around a conical point $p_i$, $f$ has the local form $z\mapsto z^{n_i+1}$ for some $n_i \geqslant 1$ called the {order} of the branch point $p_i$. We also say that the conical point $p_i$ has angle $2\pi(n_i+1)$. The datum of the integers $n_1, \ldots, n_k$ is called the \textit{branch data} of the projective structure. Our goal is to characterize the representations $\rho : \Gamma\to \PSL \C$ that arise as the holonomy of a branched projective structure with prescribed branch data. Let us denote by $\mathcal P$ the set of unbranched projective structures (without conical points) and by $\mathcal P(n_1, \ldots, n_k)$ the set of projective structures with branch points of order $n_1, \ldots, n_k$. We also denote by $\mathcal{BP}$ the set of branched projective structures on $\Sigma$.
%
%Our goal is thus to understand the image of $\mathcal P(n_1, \ldots, n_k)$ by the map that associates to a projective structure the conjugacy class of its holonomy: 
%
%The characterization of the image of $\mathcal P$ and $\mathcal P(1)$ is a celebrated theorem of Gallo, Kapovich and Marden proved in \cite{GKM}, see also \cite{Moi1}.

Recall that an elementary representation is a representation $\rho\in \Hom(\Gamma, \PSL \C)$ yielding an action with a finite orbit on $\mathbb H^3$ or on $ \CP$. In the first case we say that $\rho$ is \emph{spherical} and we will assume without loss of generality up to conjugating $\rho$, that it takes its values in $\mathrm{PSU}_2\simeq \SO$. In the second case, if $\rho(\Gamma)$ fixes a point in $\CP$ we say that $\rho$ is \emph{affine} and if it moreover fixes a horosphere we say that $\rho$ is \emph{Euclidean}. We will assume without loss of generality that this point is $\infty\in \CP$ and thus that $\rho$ takes its values in $\mathrm{Aff}(\C) = \{az + b \mid a\in \C^*, b\in \C\}$. If $\rho(\Gamma)$ fixes globally a pair of points in $\CP$, we say that $\rho$ is \emph{dihedral} and we will assume that this pair is $\{0, \infty\}$. We refer to \cite[Chapter 5]{Ratcliffe} for more information on elementary subgroups of $\PSL \C$.%The group $\mathrm{PSU}_2$ that can be identified as the subgroup 
%\begin{enumerate}
%\item The group $\mathrm{PSU}_2$ if $\rho(\Gamma)$ fixes a point in $\mathbb H^3$,
%% = \{ \pm \begin{pmatrix}F
%%\alpha & -\overline \beta\\
%%\beta & \overline \alpha
%%\end{pmatrix} \mid |\alpha|^2 + |\beta|^2 = 1\}$, 
%in which case we say that $\rho$ is spherical. This is also the group $\SO$ that acts on $\CP$ as the positive isometries of the round sphere.
%
%\item The group $\mathrm{Aff}(\C)$ of M\" obius transformations of the form $az + b$ if $\rho(\Gamma)$ fixes a point in $\CP$,
%% = \{ \pm \begin{pmatrix} Martin Deraux 
%%\lambda & \mu\\
%%0 & \lambda^{-1}
%%\end{pmatrix} \mid \lambda\in \C^\star, \mu\in \C\}$, 
%in which case we say that $\rho$ is affine. This is the group of affine transformations of $\C\subset \CP$. We also say that $\rho$ is Euclidean if its image preserves a horosphere. Its image is then contained in the group of positive isometries of the plane $\mathrm{Isom}^+(\mathbb E^2) = \{ a z + b \mid a\overline a  = 1,\ b\in \C\}$.
%%\begin{pmatrix}
%%\lambda & \mu\\
%%0 & \lambda
%%\end{pmatrix} \mid \lambda\overline \lambda = 1, \mu\in \C\}$.
%
%\item The group of M\"obius transformations that are either of the form $\lambda z$, $\lambda\in \C^*$ or of the form $\mu z^{-1}$, $\mu \in \C^*$
%%$\{\pm \begin{pmatrix}
%%\lambda & 0\\
%%0 & \lambda^{-1}
%%\end{pmatrix}\mid \lambda\in \C^\star\}\cup \{\pm \begin{pmatrix}
%%0 & \mu\\
%%-\mu^{-1} & 0
%%\end{pmatrix} \mid \mu \in \C^\star\}$
%if $\rho(\Gamma)$ fixes globally a pair of points in $\CP$. In this case we say that $\rho$ is dihedral.
%\end{enumerate}

We wish to \emph{geometrize} a given $\rho\in \Hom(\Gamma, \PSL \C)$, that is to find a projective structure $X\in \mathcal P(n_1, \ldots, n_k)$ such that $\rho =\Hol(X)$. There are some obstructions to being the holonomy of such a $X$ that we will now explicit. Indeed suppose that $\rho = \Hol(X)$.
It is proven in \cite{GKM} that  $\rho$ lifts to $\SL \C$ if and only if $\sum_i n_i$ is even.
\begin{obstruction}\label{obstrSW}
The sum $\sum_i n_i$ must be even if $\rho$ lifts to $\SL \C$ and odd otherwise.
\end{obstruction}
There exist some surgeries on projective structures, such as bubbling (see \cref{bubbling}), that allow us to add an even number of branch points to some projective structures. In particular we will see in \cref{SNE} that it is a consequence of the theorem of Gallo, Kapovich and Marden cited above that a nonelementary representation of $\Hom(\Gamma, \PSL \C)$ is in $\Hol(\mathcal P(n_1, \ldots, n_k))$ if $\sum_i n_i$ has the required parity.
Thus we now focus on elementary representations.
As a consequence of the Gauss-Bonnet formula, we have another obstruction on $\sum_i n_i$ when $\rho$ is elementary.
\begin{obstruction}\label{obstrMin}
If $\rho$ is elementary, then we must have $\sum_i n_i\geqslant 2g-2$.
If moreover $\rho$ is spherical, then $\sum_i n_i\geqslant 2g - 1$.
\end{obstruction}

The image of $\rho$ might be a finite group. In this case the developing map induces a branched cover $S\to \CP$ defined on the cover $S$ of $\Sigma$ associated with the subgroup $\ker \rho$ of $\Gamma$. The Riemann-Hurwitz formula then gives another obstruction.

\begin{obstruction}\label{obstrRH}
If the image of $\rho$ is a finite group of order $n$, then $$n(\chi(\Sigma) + \sum_{i=1}^k n_i) \geqslant 2 (\max_{1\leqslant i\leqslant k} n_i + 1).$$
\end{obstruction}
Let us define an \emph{affine structure} on $\Sigma$ as a branched projective structure on $\Sigma$ with affine holonomy, such that the image of its developing map does not contain $\infty\in \CP$.
Adapting an argument of Kapovich in \cite{Kapovich}, we will see in \cref{liredegre} that if $\rho$ is affine, then $\sum_i n_i = 2g-2$ if and only if $X$ is an affine structure.
If moreover $\rho$ is Euclidean, then we may define the volume $\mathrm{Vol}(\rho)$ of $\rho$, see \cref{AffineHol}. The volume of $\rho$ must be positive, as we shall see in \cref{Euclide}.
\begin{obstruction}\label{obstrVol}
If $\sum_i n_i = 2g-2$ and $\rho$ is Euclidean, then $\mathrm{Vol}(\rho) > 0$.
\end{obstruction}
Haupt \cite{Haupt} observed that there is another obstruction to being the holonomy of a \emph{translation surface}, that is an affine structure with holonomy in the group of translations $\C$. Indeed if $g\geqslant 2$ and $\chi\in \Hom(\Gamma, \C)$ is the holonomy of a translation surface such that $\Lambda = \chi(\Gamma)$ is a lattice, then $d=\mathrm{Vol}(\chi) /\mathrm{Area}(\Lambda)$ must be at least $2$ because the developing map induces a branched cover $\Sigma\to \C/\Lambda$ of degree $d$. This obstruction was generalized by \cite{Deroin} for translation surfaces with prescribed singularities, where the authors observed that $d$ must satisfy $d\geqslant \max_i n_i +1$ and asked if these were the only obstructions to being the holonomy of a translation surface with prescribed singularities. This latter question was answered in the positive in \cite{Moi, Judge}. 
The \emph{linear part} of an Euclidean representation $\rho$ is the representation $\mathrm{Li}\circ \rho$, where $\mathrm{Li} : \mathrm{Aff}(\C)\to \C^*$ is defined by $az + b\mapsto a$. If the linear part of $\rho$ takes its values in a finite group, then the cover $S$ of $\Sigma$ associated with the subgroup $\ker \mathrm{Li}\circ \rho$ of $\Gamma$ inherits a translation surface structure. Therefore the obstruction of Haupt gives the following.
%Observe that if $\rho$ is Euclidean with the image of its linear part in a finite group, then $X$ gives a projective structure with trivial linear holonomy on a finite cover of $\Sigma$. There is another obstruction on a representation in $\Hom(\Gamma, \C)$ to being the holonomy of a translation surface, \textit{i.e.} of an affine structure with holonomy in the subgroup $\C\subset \mathrm{Aff}(\C)$ of translations. This obstruction was first observed by Haupt in \cite{Haupt}. It was generalized in \cite{Deroin} to the case of translation surfaces with prescribed singularities, where the authors asked if these were the only obstructions to being the holonomy of a translation surface with prescribed singularities. This latter question was answered in the positive in \cite{Moi, Judge}. This obstruction can be stated as follows in our setting.
\begin{obstruction}\label{obstrHaupt}
If $\sum_i n_i = 2g-2$ and $\rho$ is Euclidean such that the image of its linear part is a finite group of order $n$ and the group $\Lambda = \{z_0\in \C \mid 
z + z_0\in \rho(\Gamma)\}$ is a lattice in $\C$, then we must have $$n\mathrm{Vol}(\rho) \geqslant (\max_i n_i +1)\mathrm{Area}(\C / \Lambda).$$
\end{obstruction}
As we will see, $\Lambda$ can be a lattice only if $n\in \{1,2,3,4,6\}$.

If $\rho$ is dihedral there is an obstruction that only occurs in genus $2$, when $\sum_i n_i$ is minimal: that is when $\sum_i n_i = 2$.
This new obstruction comes from the observation that every such structure is the exponential of a half-translation structure. We observe thanks to \cref{obstrHaupt} that a half-translation structure in genus $2$ with a single branch point is in fact a translation surface structure, see \cref{exemple}. Therefore we show  that there is no branched projective structure with holonomy $\rho$ and a single branch point in genus $2$.
\begin{obstruction}\label{obstrGenre2}
Suppose $g=2$.
If $\rho$ is a dihedral representation that is not affine  and that lifts to $\SL \C$, then $X$ cannot have a single branch point of order $2$.
\end{obstruction}

Our main result is that these are the only obstructions to being in $\Hol(\mathcal P(n_1, \ldots, n_k))$.
\begin{theo}\label{mainth}
Let $\rho\in \Hom(\Gamma, \PSL \C)$. We have $\rho\in \Hol(\mathcal P(n_1, \ldots, n_k))$ if and only if $\rho$ satisfies the conditions of the 6 obstructions above. 
\end{theo}
In particular for a given $\rho\in \Hom(\Gamma, \PSL \C)$ we can compute the number $d(\rho)$ defined as the minimal value of $\sum_i n_i$ such that there exists a branched projective structure in $\mathcal P(n_1, \ldots, n_k)$ with holonomy $\rho$.
\begin{cor}\label{dintro}
The image of the function $d : \Hom(\Gamma, \PSL \C) \to \mathbb Z$ is the set $\{0, 1, 2g-2, 2g-1, 2g, 2g+2\}$ and we explicit $d$, see \cref{fonctiond} at the end of this introduction.
\end{cor}

Our proof of \cref{mainth} consists in the case-by-case study of the possible $n_1, \ldots, n_k$ such that $\rho\in \Hol(\mathcal P(n_1, \ldots, n_k))$, for each form that $\rho(\Gamma)$ can take. In some of these cases we get back to notions that were already studied in the literature. 
For example if $\rho(\Gamma)\subset \mathrm{PSU}_2$, then we are characterizing the holonomies of conical spherical metrics on $\Sigma$.
These results in some special cases that we now explicit have their own interest and form the pieces of our proof of \cref{mainth}.

\subsubsection{Spherical structures.}
Let us first focus on spherical representations. \cref{mainth} consists in this case in computing the holonomies of spherical metrics with conical singularities of integral angles. For more information on spherical metrics with conical singularities we refer to \cite{Eremenko, MondelloPanov}.
\begin{prop}
Let $\rho\in \Hom(\Gamma, \SO)$. The representation $\rho$ is the holonomy of a spherical metric with conical singularities of angles $2\pi(n_1+1), \ldots, 2\pi(n_k+1)$ if and only if the following conditions are met.
\begin{enumerate}
\item $\sum_i n_i$ is even if $\rho$ lifts to $\mathrm{SU}_2$ and odd otherwise.
\item $\sum_i n_i \geqslant 2g-1$.
\item If the image of $\rho$ is finite of order $n\geqslant 1$, then $$n(\chi(\Sigma) + \sum_{i=1}^k n_i) \geqslant 2 (\max_{1\leqslant i\leqslant k} n_i + 1).$$
\end{enumerate}
\end{prop}
Observe that a branched projective structure with trivial holonomy is just a branched cover of the sphere.
\begin{prop}\label{branchedcover}
There exists a branched cover $\Sigma\to \mathbb S^2$ with branch points of order $1\leqslant n_1\leqslant \ldots\leqslant n_k$ if and only if $\sum_i n_i = 2(g + d-1)$ for an integer $d \geqslant 1$ such that $1\leqslant n_i \leqslant d-1$ for every $1\leqslant i\leqslant k$.
\end{prop}
Note that our convention differs in this setting from the usual one: by a branch point of order $n$ we mean that the branched cover is locally $z\mapsto z^{n+1}$ while this is often called a branch point of order $n+1$.
While the present article was being written, an independent proof of
Proposition 1.4 appeared in the work of Chenakkod, Faraco and Gupta
\cite[Corollary D]{CFG}. Their proof, as well as ours, consists in showing directly
that some branched data are realizable. We do it by following the
techniques of \cite{EdmondsKulkarni}, while their proof is more geometric.
The case $g=0$ was also proven, with different methods, independently
by Kapovich and Tomasini in \cite{KapovichSphere, Tomasini}.

\subsubsection{Affine holonomy.}
If we consider affine representations and $\sum_i n_i = 2g-2$ then we compute the holonomies of {affine structures} with prescribed singularities. 
\begin{prop}\label{corIntro}
Let $1\leqslant n_1\leqslant \ldots \leqslant n_k$ be such that $\sum_i n_i = 2g-2$. Let $\rho\in \Hom(\Gamma,\mathrm{Aff}(\C))$.
\begin{enumerate}
\item If $\rho$ is not Euclidean then $\rho\in \Hol(\mathcal P(n_1, \ldots, n_k))$.
\item If $\rho$ is Euclidean, then $\rho\in \Hol(\mathcal P(n_1, \ldots, n_k))$ if and only if both $\mathrm{Vol}(\rho) >0$
 and $\rho$ satisfies the condition of \cref{obstrHaupt} if the image of its linear part is finite and has order in $\{1, 2, 3, 4, 6\}$.
\end{enumerate}
\end{prop}
This answers a question raised in \cite[Question 1.3]{CFG}. In fact \cref{corIntro} is due to a large extent to Ghazouani \cite{Ghazouani}, where he studies the action of the mapping class group on $\Hom(\Gamma, \mathrm{Aff}(\C))/\mathrm{Aff}(\C)$ and characterizes the representations of $\Hom(\Gamma, \mathrm{Aff}(\C))$ that are the holonomy of an affine structure.
Ghazouani does not address the problem of characterizing the representations that are holonomies of affine structures with prescribed branch points.
Let us observe that \cref{corIntro} gives a new proof of the fact that the strata of quadratic differentials $\mathcal Q(2n_1, \ldots, 2n_k)$ are nonempty except if $g = 1$ and $k = 0$, or $g = 2$ and $k = 1$, see \cref{Q(vide)}. This was proven by Masur and Smilie in \cite{MasurSmillie}, where they study the nonemptiness of all the strata of quadratic differentials. See also \cite{Lanneau} for another proof and the study of the connectedness of these strata.

In the case where $\sum_i n_i = 2d$ for $d\geqslant g$,  we will see that every $\rho\in \Hom(\Gamma, \mathrm{Aff}(\C))$ that is not spherical is in $\Hol(\mathcal P(n_1, \ldots, n_k))$.
\subsection{Action of the mapping class group on some representation spaces}
We will prove a Ehresmann-Thurston principle, following an argument of Kapovich \cite[Theorem 2.7]{Kapovich}.
\begin{prop}\label{ET} 
The sets $\Hol(\mathcal P(n_1, \ldots, n_k)) $ are open.
\end{prop}

The mapping class group $\Mod (\Sigma) = \pi_0(\homeo \Sigma)$ acts on the space of conjugacy classes of representations $\Hom(\Gamma, \PSL \C)/\PSL \C$ and this action preserves the sets $\Hol(\mathcal P(n_1, \ldots, n_k))$. This group identifies with the positive outer automorphisms of $\Gamma$,  $\mathrm{Out}^+(\Gamma) = \Aut(\Gamma)/ \mathrm{Inn}(\Gamma)$ and acts on $\Hom(\Gamma, \PSL \C) / \PSL \C$ by precomposition (see \cref{rappel}) where $\Aut(\Gamma)$ is the index $2$ subgroup of $\mathrm{Aut}(\Gamma)$ of automorphisms induced by orientation preserving homeomorphisms.
In order to geometrize a representation with prescribed branch data, it suffices by \cref{ET} to geometrize any element of the closure of its orbit under the action of $\Mod(\Sigma)$. Therefore our proof of \cref{mainth} is based on the description of the closure of the orbits of this action. Indeed this description will leave us with a small number of explicit representations to geometrize.

The study of the mapping class group action on character varieties $\Hom(\Gamma, G)/G$, where $G$ is a group, was pioneered by Goldman and is a wide and active research area. We refer to \cite{Goldman} for an introduction to this subject and \cite{Canary} and references therein for more information on this domain of research.
In particular Goldman proved in \cite{GoldmanErgodic} that the action of the mapping class group on $\Hom(\Gamma, \mathrm{SU}_2)/\mathrm{SU}_2$ is ergodic. While this result provides a description of the closure of almost every orbit, it does not allow us to understand the closure of a given orbit.
The description of the closure of the orbit of every representation in $\Hom(\Gamma, \mathrm{SU}_2)/\mathrm{SU}_2$ with dense image was achieved by Previte and Xia in \cite{PreviteXia1, PreviteXia2}.  A consequence of their work is that the orbit of a representation in $\Hom(\Gamma, \SO)$ with image a dense subgroup of $\SO$ is dense in its connected component of $\Hom(\Gamma, \SO)/\SO$, see \cref{previtxia}. We study the other orbits of representations in $\Hom(\Gamma, \SO)$ and give a complete description of their closures. This somehow complements the main result of \cite{PreviteXia2} in the case of closed surfaces.
Let $\epsilon : \mathrm{O}_2(\R)\to \mathrm{O_2}(\R)/\mathrm{SO}_2(\R) = \Z_2$ be the quotient homomorphism. Goldman showed in \cite{GoldmanCC} that $\Hom(\Gamma, \PSL \C)$ endowed with the pointwise convergence topology has two connected components. Representations in one of these connected components lift to $\SL \C$ while the others do not. Let $\sw : \Hom(\Gamma, \PSL \C)\to \{0,1\}$ be the map that assigns $0$ to the representations that lift to $\SL \C$ and $1$ to the others. The proper closed subgroups of $\SO$ are isomorphic to $\mathrm{SO}_2(\R)$, $\mathrm O(2)$, the cyclic groups $\Z_n = \Z/n\Z$, the dihedral groups $\D_n$ of order $2n$, the alternating groups $\A 4$ and $\A 5$ and the symmetric group $\Sy 4$. Let us start with $\rho\in \Hom(\Gamma, \SO)$ with infinite image. After conjugating $\rho$, we may assume that $\rho\in \Hom(\Gamma, \mathrm{O}_2(\R))$.
%We denote by $\Aut(\Gamma)$ the group of positive automorphisms of $\Gamma$.

\begin{prop}\label{infini}
%Let $G$ be a proper closed subgroup of $\SO$ and $\rho_0\in \Hom(\Gamma, G)$ with image dense in $G$. The closure of $\Aut(\Gamma)\cdot \rho_0$ is the set of $\rho\in \Hom(\Gamma, G)$ such that $\sw(\rho) = \sw(\rho_0)$ and $p(\rho) = p(\rho_0)$ in the case where $G$ is a included in a conjugate of $\mathrm{O}_2$.
Let $\rho\in \Hom(\Gamma, \mathrm{O}_2(\R)))$ with infinite image. The closure of $\Aut(\Gamma)\cdot\rho$ is the set of representations $\rho_\infty\in \Hom(\Gamma, \mathrm{O}_2(\R))$ such that $\epsilon\circ \rho_\infty$ has the same image as $\epsilon\circ \rho$ and $\sw(\rho_\infty) = \sw(\rho)$.
\end{prop}
We also describe the orbits of representations with finite image in $\SO$.

\begin{prop}\label{GpeFin}
Suppose $g\geqslant 2$.
The action $\Aut(\Gamma)$ on the set of surjective homomorphisms of $\Hom(\Gamma, G)$ is transitive for $G = \Z_n$ or $G = \D_n$ with $n$ odd. There are two orbits for $G = \A 4, \A 5, \Sy 4, \D_n$ with $n$ even.
%There are two orbits for $G = \D_n$ with $n$ even, unless if $g = 1$ and $n = 2$, where the action is transitive.
\end{prop}

Edmonds showed in \cite{Edmonds1, Edmonds2} that if $G$ is abelian or if $G$ is a semi-direct product of cyclic groups then two surjective representations $\rho_1$ and $\rho_2$ are in the same orbit under the action of $\Aut(\Gamma)$ if and only if $\rho_1([\Sigma]) = \rho_2([\Sigma])\in \mathrm{H}_2(G, \Z)$, see \cref{SGP} for more details on these notations. In particular this proves \cref{GpeFin} in the case $G=\Z_n$ and $G = \D_n$. Dunfield and Thurston also showed in \cite[Section 6]{DunfieldThurston} that given a finite group $G$, the action of $\Aut(\Gamma)$ on the set of surjective homomorphisms of $\Hom(\Gamma, G)$ has $|\mathrm H_2(G, \Z)|$ orbits for large enough $g$. A consequence of \cref{GpeFin} is that in these particular cases, it happens for every $g\geqslant 2$.

Let us now describe the organization of the paper.
In \cref{rappel} we recall some well-known facts about fundamental groups of surfaces, their automorphism groups and their representations.
In \cref{SGP} we study the action of the mapping class group on the space of conjugacy classes of some representations with finite image. More precisely we prove \cref{GpeFin}. In \cref{SINF} we provide a description of some orbits of elementary representations with infinite image. In particular we prove \cref{infini} and exhibit some particularly simple representations in the closure of the orbits of affine representations. In \cref{BPS} we recall the definition of branched projective structures and prove \cref{ET}. We also explain some surgeries on these structures that allow us to construct new projectives structures from old ones. In \cref{SGEO} we prove \cref{mainth}: we explain the origins of the obstructions and show that a representation satisfying the conditions of these obstructions is \textit{geometric}: it is in $\Hol(\mathcal P(n_1, \ldots, n_k))$. Finally in \cref{last_section} we explicit the function $d$ and prove \cref{dintro}.

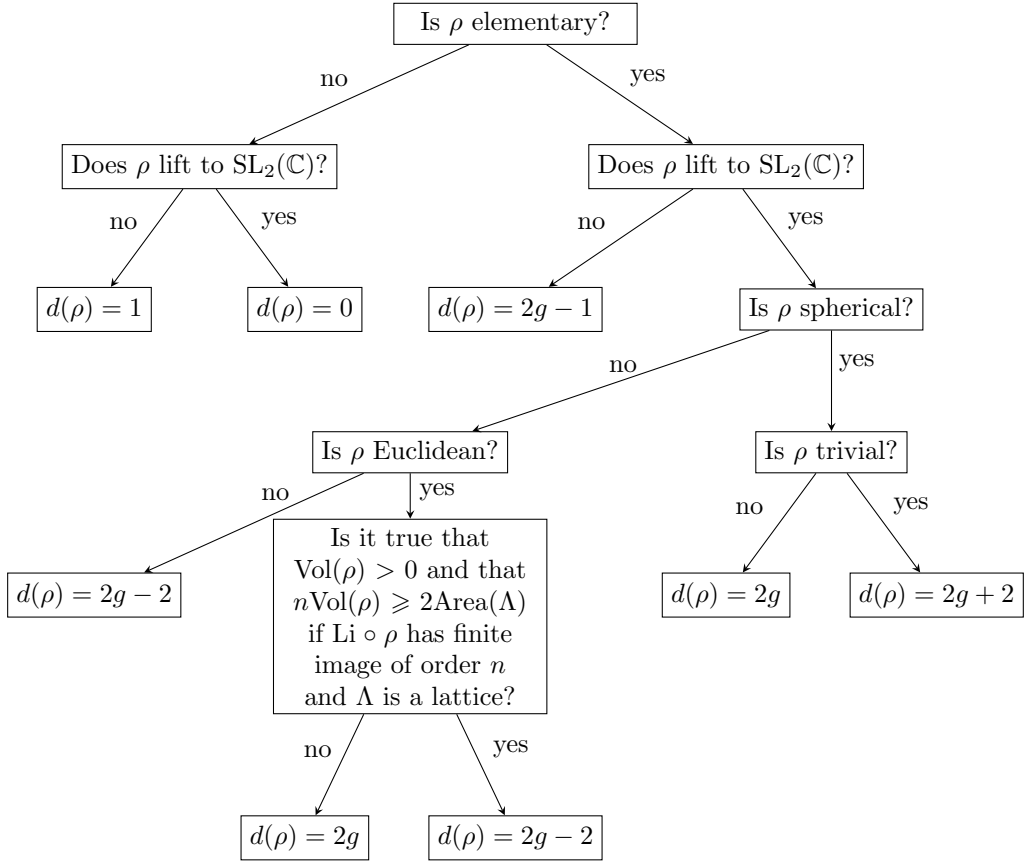
\begin{figure}[h!]
\begin{tikzpicture}[>=stealth,yscale=1.9,xscale=1.4]  
% xscale et yscale permettent d'ajuster largeur et hauteur

% positionner les nœuds
\node(comp) at (4,5)[rectangle,draw,text width=3cm,text centered] {Is $\rho$ elementary?};
\node(desk) at (1,4)[rectangle,draw] {Does $\rho$ lift to $\SL \C$?};
\node(lap)  at (6,4)[rectangle,draw] {Does $\rho$ lift to $\SL \C$?}; 
\node(HPl)  at (0,3)[rectangle,draw] {$d(\rho) = 1$};     
\node(IBM)  at (2,3)[rectangle,draw] {$d(\rho) = 0$};    
\node(DELL) at (4,3)[rectangle,draw] {$d(\rho) = 2g-1$};   
\node(HPr)  at (7,3)[rectangle,draw] {Is $\rho$ spherical?};
\node(Eucl) at (3, 2)[rectangle,draw] {Is $\rho$ Euclidean?};  
\node(Tri) at (7, 2)[rectangle,draw] {Is $\rho$ trivial?};  
\node(fintri) at (8, 1)[rectangle,draw]{$d(\rho) = 2g+2$};
\node(finG) at (6, 1)[rectangle,draw]{$d(\rho) = 2g$};
\node(Aff) at (0, 1)[rectangle,draw]{$d(\rho) = 2g-2$};
\node(Vol) at (3, 0.85)[rectangle,draw, text width=3.4cm,text centered]{Is it true that $\mathrm{Vol}(\rho) > 0$ and that $n\mathrm{Vol}(\rho)  \geqslant 2 \mathrm{Area}(\Lambda)$ if $\mathrm{Li}\circ \rho$ has finite image of order $n$ and $\Lambda$ is a lattice?};
\node(finYes) at (4, -0.7)[rectangle, draw]{$d(\rho) = 2g-2$};
\node(finNo) at (2, -0.7)[rectangle, draw]{$d(\rho) = 2g$};

% tirer les liens
\draw[->] (comp) -- (desk); 
\draw[->] (comp) -- (lap);
\draw[->] (desk) -- (HPl); 
\draw[->] (desk) -- (IBM); 
\draw[->] (lap)  -- (DELL); 
\draw[->] (lap)  -- (HPr);
\draw[->] (HPr) -- (Eucl);
\draw[->] (Tri) -- (fintri);
\draw[->] (Tri) -- (finG);
\draw[->] (HPr) -- (Tri);
\draw[->] (Eucl) -- (Aff);
\draw[->] (Eucl) -- (Vol);
\draw[->] (Vol) -- (finYes);
\draw[->] (Vol) -- (finNo);
\draw (2.5, 4.5) node [above left] {no} ;
\draw (5, 4.5) node [above right] {yes} ;
\draw (0.5, 3.5) node [above left] {no} ;
\draw (1.5, 3.5) node [above right] {yes} ;
\draw (4.5, 3.5) node [above right] {no} ;
\draw (6.5, 3.5) node [above right] {yes} ;
\draw (4.8, 2.5) node [above right] {no} ;
\draw (7.5, 2.5) node [above left] {yes} ;
\draw (7.5, 1.5) node [above right] {yes} ;
\draw (6, 1.5) node [above right] {no} ;
\draw (3, 1.6) node [above right] {yes} ;
\draw (1.5, 1.6) node [above right] {no} ;
\draw (1.9, -0.2) node [above right] {no} ;
\draw (3.7, -0.2) node [above right] {yes} ;
\end{tikzpicture}
\caption{The function $d : \Hom(\Gamma, \PSL \C)\to \Z$.}
\label{fonctiond}
\end{figure}
\subsection*{Ackowledgements}
I wish to thank my advisor Maxime Wolff for his help and support. I also thank Gabriel Calsamiglia, Bertrand Deroin and Julien March\' e  for their kind interest.
\section{Reminder on the mapping class group}\label{rappel}
\subsection{Fundamental groups of surfaces}
For each $g,n\geqslant 0$, we fix $\Sigma_{g,n}$ a compact oriented surface of genus $g$ with $n$ boundary components. We also choose a basepoint in each of these surfaces and thus a fundamental group $\Gamma_{g,n}$. Let us fix standard generators $a_1, \ldots, b_g, c_1, \ldots, c_n$ for $\Gamma_{g,n}$, as in \cref{generator}.
\begin{figure}[h]
    \centering    
    \def\svgwidth{\columnwidth}
%	\hspace*{2cm}
	\def\svgwidth{0.7\textwidth}
	%% Creator: Inkscape 1.0.1 (0767f8302a, 2020-10-17), www.inkscape.org
%% Accompanies image file 'generateurs.pdf' (pdf, eps, ps)
\begingroup%
  \makeatletter%
  \providecommand\color[2][]{%
    \errmessage{(Inkscape) Color is used for the text in Inkscape, but the package 'color.sty' is not loaded}%
    \renewcommand\color[2][]{}%
  }%
  \providecommand\transparent[1]{%
    \errmessage{(Inkscape) Transparency is used (non-zero) for the text in Inkscape, but the package 'transparent.sty' is not loaded}%
    \renewcommand\transparent[1]{}%
  }%
  \providecommand\rotatebox[2]{#2}%
  \newcommand*\fsize{\dimexpr\f@size pt\relax}%
  \newcommand*\lineheight[1]{\fontsize{\fsize}{#1\fsize}\selectfont}%
  \ifx\svgwidth\undefined%
    \setlength{\unitlength}{447.04482446bp}%
    \ifx\svgscale\undefined%
      \relax%
    \else%
      \setlength{\unitlength}{\unitlength * \real{\svgscale}}%
    \fi%
  \else%
    \setlength{\unitlength}{\svgwidth}%
  \fi%
  \global\let\svgwidth\undefined%
  \global\let\svgscale\undefined%
  \makeatother%
  \begin{picture}(1,0.38788195)%
    \lineheight{1}%
    \setlength\tabcolsep{0pt}%
    \put(0,0){\includegraphics[width=\unitlength,page=1]{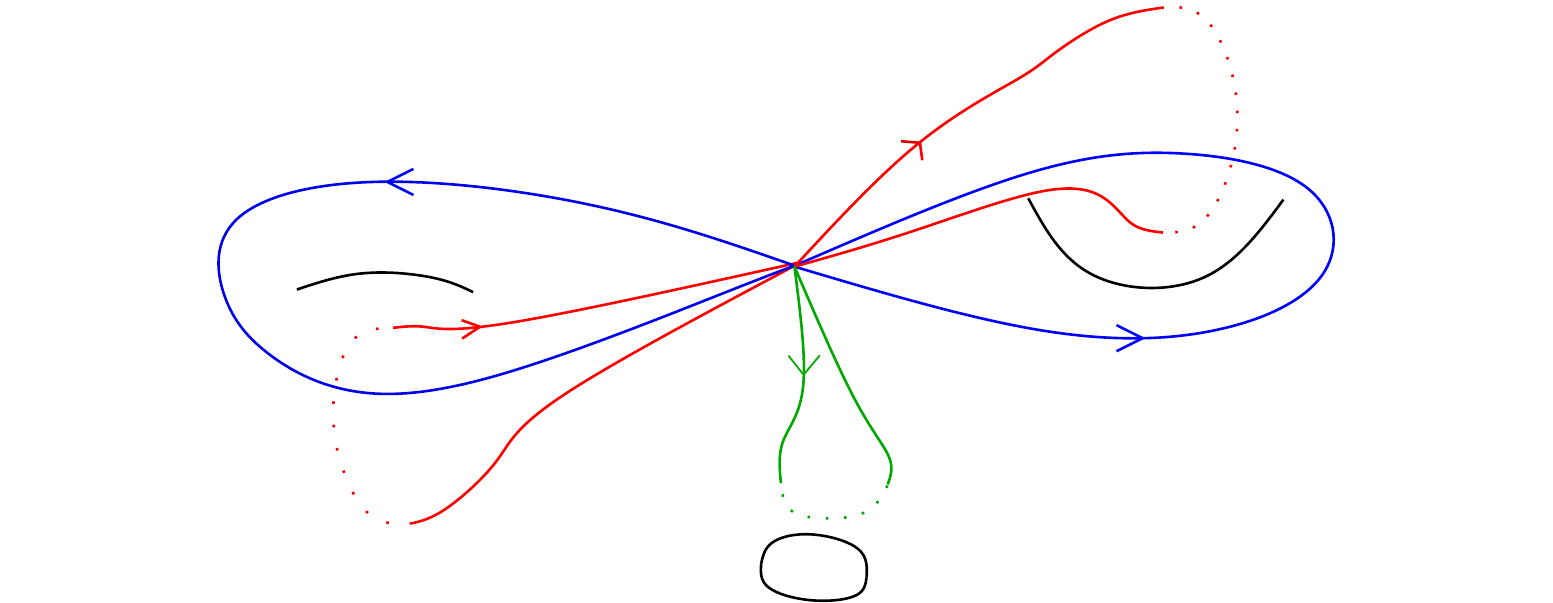}}%
    \put(0.27830602,0.2814433){\color[rgb]{0,0,0}\makebox(0,0)[lt]{\lineheight{1.25}\smash{\begin{tabular}[t]{l}$b_1$\\\end{tabular}}}}%
    \put(0.31776404,0.19764224){\color[rgb]{0,0,0}\makebox(0,0)[lt]{\lineheight{1.25}\smash{\begin{tabular}[t]{l}$a_1$\end{tabular}}}}%
    \put(0.47183891,0.13015215){\color[rgb]{0,0,0}\makebox(0,0)[lt]{\lineheight{1.25}\smash{\begin{tabular}[t]{l}$c_1$\end{tabular}}}}%
    \put(0.68163049,0.33972261){\color[rgb]{0,0,0}\makebox(0,0)[lt]{\lineheight{1.25}\smash{\begin{tabular}[t]{l}$a_2$\end{tabular}}}}%
    \put(1.99128631,0.46804483){\color[rgb]{0,0,0}\makebox(0,0)[lt]{\begin{minipage}{0.88552296\unitlength}\raggedright \end{minipage}}}%
    \put(0.73600682,0.13601806){\color[rgb]{0,0,0}\makebox(0,0)[lt]{\lineheight{1.25}\smash{\begin{tabular}[t]{l}$b_2$\end{tabular}}}}%
    \put(0,0){\includegraphics[width=\unitlength,page=2]{generateurs.pdf}}%
  \end{picture}%
\endgroup%
\
	\caption{Standard generators of  $\Gamma_{2,1}$.}
    \label{generator}
\end{figure}
These standard generators yield the following presentation of $\Gamma_{g,n}$
$$\Gamma_{g,n} = \langle a_1, b_1, \ldots, a_g, b_g, c_1, \ldots, c_n \mid \prod_{i=1}^g [a_i, b_i] = c_1 \ldots c_n\rangle.$$
In particular for $n=0$, we have $$\Gamma = \langle a_1, b_1, \ldots, a_g, b_g\mid \prod_{i=1}^g [a_i, b_i] = \id\rangle.$$
Recall that every $\sigma\in \mathrm{Aut}(\Gamma)$ is induced by a continuous map $f : \Sigma\to \Sigma$ preserving the basepoint of $\Gamma$. We denote by $\Aut(\Gamma)$ the index $2$ subgroup of $\mathrm{Aut}(\Gamma)$ induced by the $f$ preserving the orientation of $\Sigma$.
Let us denote by $\mathrm{Homeo}_0(\Sigma)$ 
the connected component of the identity of $\mathrm{Homeo}_+(\Sigma)$, the set of positive homeomorphisms of $\Sigma$, equipped with the compact-open topology. It is the set of homeomorphisms of $\Sigma$ isotopic to the identity, see \cite[Chapter 2]{FarbMargalit}.
 The mapping class group is the group of $\Mod(\Sigma) = \mathrm{Homeo}_+(\Sigma)/\mathrm{Homeo}_0(\Sigma)$ of isotopy class of positive homeomorphisms of $\Sigma$.
Every $f\in \Mod(\Sigma)$ induces an outer automorphism of $\Gamma$ by the action on $\Gamma$ of one of its representatives that fixes the basepoint of $\Gamma$. The theorem of Dehn, Nielsen and Baer, see \cite[Chapter 6]{FarbMargalit}, states that this homomorphism $\Mod(\Sigma)\to \mathrm{Out}^+(\Sigma)$ is an isomorphism, where $\mathrm{Out}^+(\Gamma) = \Aut(\Gamma)/\mathrm{Inn}(\Gamma)$ and $\mathrm{Inn}(\Gamma)$ denotes the set of inner automorphisms of $\Gamma$.

\subsection{Action on representation spaces}

Let $G$ be a group.
The group $\Aut(\Gamma)$ acts on the set of representations $\Hom(\Gamma, G)$ by precomposition. The group $G$ also acts by conjugation on $\Hom(\Gamma, G)$ and these two actions commute. Therefore $\Aut(\Gamma)\times G$ acts on $\Hom(\Gamma, G)$. The group $\mathrm{Inn}(\Gamma)$ of inner automorphisms of $\Gamma$ acts trivially on $\Hom(\Gamma, G)/G$. Hence we have an action of $\mathrm{Out}^+(\Gamma)$ on $\Hom(\Gamma, G)/G$.
Therefore the group $\Mod(\Sigma)$ acts on $\Hom(\Gamma, G) /G$. Note that we may restrict both the action of $\Aut(\Gamma)$ and the action of $\mathrm{Out}^+(\Gamma)$ on the set of (conjugacy classes) of surjective homomorphisms of $\Hom(\Gamma, G)$.
\begin{rmk}\label{conjugaison}
Two surjective homomorphisms of $\Hom(\Gamma, G)$ are in the same orbit under the action of $\Aut(\Gamma)\times G$ if and only if they are in the same orbit under the action of $\Aut(\Gamma)$. 
\end{rmk}

In our study of the $\Aut(\Gamma)$ action on $\Hom(\Gamma, G)$, we will often consider the restriction of a homomorphism $\rho\in \Hom(\Gamma, G)$ such that $\rho$ kills $c = \prod_{i=1}^k [a_i, b_i]$ to $\Hom(\Gamma_{g-k}, G)$. The following lemma ensures that we may then study the action of $\Aut(\Gamma_{g-k})$ on this restriction to understand the orbit of $\rho$ under the action of $\Aut(\Gamma)$.

\begin{lemma}\label{decoupe}
Let $\rho\in \Hom(\Gamma, G)$ and $c$ be a simple separating curve in $\Sigma$ such that $\rho(c) = \id$.
The curve $c$ induces a splitting $\Sigma = \Sigma_h \sharp \Sigma_l$, where $h + l = g$. For every $\sigma\in \Aut(\Gamma_h)$, the representation that we denote by $\sigma\cdot \rho$ satisfying:
$$\sigma\cdot\rho(\gamma) = \left\{
\begin{array}{ll}
	\rho(\sigma\cdot\gamma)& \mbox{if } \gamma\in \pi_1(\Sigma_h)\\
	\rho(\gamma) & \mbox{if } \gamma\in \pi_1(\Sigma_l)
\end{array}
\right.$$
lies in $\Aut(\Gamma)\cdot\rho$.
%be such that $c = \prod_{i=k}^l [a_i, b_i]$ satisfies $\rho(c) = \id$, where $1\leqslant k \leqslant l \leqslant g$. For every $\sigma\in \Aut(\Gamma_{g - (k-l+1)})$, the representation that we denote by $\sigma\cdot \rho$,
%defined by $\gamma\mapsto \rho(\gamma)$ for $\gamma\in \{a_1, \ldots, b_g\}\setminus \{a_k, \ldots, b_l\}$ and by $\gamma\mapsto \sigma\cdot\rho'(\gamma)$ for $\gamma\in \{a_k, \ldots, b_l\}$ is in $\Aut(\Gamma)\cdot \rho$. 
\end{lemma}

\begin{proof}
Let us fix such a curve $c$.
The set $\mathcal A$ of $\sigma\in \Aut(\Gamma_h)$ satisfying this property for every $\rho$ such that $\rho(c) = \id$ is stable by multiplication. It contains the automorphisms induced by Dehn twists along nonseparating simple closed curves contained in $\Sigma_h$. The set $\mathcal A$ also contains the inner automorphisms of $\Gamma_h$: if $\delta$ is a nonseparating simple closed curve in $\Sigma_h$, then the composition of a Dehn twist along a curve homotopic to $\delta$ and of a Dehn twist along a curve homotopic to $c^{\pm 1}\delta$ conjugates $\rho(\gamma)$ with $\rho(\delta)$ for $\gamma\in \pi_1(\Sigma_h)$ and leaves $\rho(\gamma)$ unchanged for $\gamma\in \pi_1(\Sigma_l)$.
Since the Dehn twists along nonseparating simple closed curves generate $\mathrm{Out}^+(\Gamma_h)$, see \cite[Chapter 4]{FarbMargalit}, we have $\mathcal A = \Aut(\Gamma_h)$. 
\end{proof}

Observe that if $G$ is an abelian group, then any homomorphism $\rho\in \Hom(\Gamma, G)$ factors through a homomorphim in $\Hom(\mathrm H_1(\Sigma, \Z), G)$ since the abelianization of $\Gamma$ is $\mathrm H_1(\Sigma, \Z)$. The group $\mathrm H_1(\Sigma, \Z)$ is the free abelian group on the generators $[a_i], [b_i]$. Therefore, we may identify $\Hom(\Gamma, G)$ with $G^{2g}$, by the map $\rho\mapsto (\rho(a_1), \ldots, \rho(b_g))$.
The group $\Mod(\Sigma)$ acts on $\mathrm H_1(\Sigma, \Z)\simeq \Z^{2g}$ and it is well known that this action is made by the full symplectic group $\mathrm{Sp}_{2g}(\Z)$, see \cite[Chapter 6]{FarbMargalit}. The action of $\Mod(\Sigma)\simeq \mathrm{Out^+(\Gamma)}$ on $\Hom(\Gamma, G)\simeq G^{2g}$ is thus by $\mathrm{Sp}_{2g}(\Z)$ that acts on $G^{2g}$ seen as a $\Z$-module.

\section{Mapping class group orbits on representation spaces into some finite groups}\label{SGP}

In this section we study the orbits of the $\Aut(\Gamma)$ action on the set of surjective homomorphisms of $\Hom(\Gamma, G)$ for some finite groups $G$. 
Observe that the projection $\Hom(\Gamma, G)\to \Hom(\Gamma, G)/G$ induces a bijection between the $\Aut(\Gamma)$ orbits of surjective homomorphisms and the $\Mod(\Sigma)$ orbits of conjugacy class of surjective homomorphisms by \cref{conjugaison}.
Let us recall the classification of these orbits when the genus is large enough, from \cite[Theorem 6.23]{DunfieldThurston}.
There is a natural map $\Phi : \Hom(\Gamma, G)\to \mathrm H_2(G, \Z)$ defined as follows. Every $\rho\in \Hom(\Gamma, G)$ induces a map in homology $\rho_\star : \mathrm H_2(\Gamma, \Z)\to \mathrm H_2(G, \Z)$ and we define $\Phi(\rho)$ to be $\rho_\star([\Sigma])\in \mathrm H_2(G, \Z)$, where $[\Sigma]\in \mathrm{H}_2(\Sigma, \Z) = \mathrm{H_2}(\Gamma, \Z)$ is the fundamental class of $\Sigma$. This map is invariant under the action of $\Aut(\Gamma)$, since every positive automorphism of $\Gamma$ maps $[\Sigma]$ to itself.
\begin{theo}[Dunfield-Thurston]
Let $G$ be a finite group.
There exists $g_0\geqslant 0$ such that the natural map from the orbits of surjective homomorphisms of $\Hom(\Gamma_g, G)$ under the action of $\Aut(\Gamma)$ to $\mathrm{H}_2(G, \Z)$ is a bijection for every $g \geqslant g_0$.
\end{theo}
The proof of Dunfield and Thurston is not constructive and does not give any bound on the minimal $g_0$ required. We will see that $g_0 = 1$ in the case $G = \Z_n$ and that $g_0=2$ in the case $G \in \{\D_n, \A 4, \Sy 4, \A 5\}$.
There does not exist any surjective homomorphism in $\Hom(\Gamma_1, G)$ for $G\in \{\A 4, \Sy 4, \A 5\}$ and for $G = \D_n$, $n\geqslant 3$, since these groups are not abelian. Moreover, the image of the set of surjective homomorphisms of $\Hom(\Gamma_1, \D_2)$ by $\Phi$ is not the whole group $\mathrm{H}_2(\D_2, \Z)$. Hence $g_0\geqslant 2$ for these groups. However for $g\geqslant 2$ and these groups $G$, the image of the set of surjective homomorphisms of $\Hom(\Gamma, G)$ by $\Phi$ is surjective. We may check that $|\mathrm{H}_2(G, \Z)|\leqslant 2$. We are going to show that the action of $\Aut(\Gamma)$ on the set of surjective homomorphisms with fixed image by $\Phi$ is transitive.
\subsection{Cyclic groups}
Let us start with the case $G = \Z_n$, the cyclic group of order~$n\geqslant 2$.
\begin{prop}\label{ModCycl}
The action of $\Aut(\Gamma)$ on the set of surjective homomorphisms of $\Hom(\Gamma, \Z_n)$ is transitive.
\end{prop}

This follows from the work of Edmonds in \cite{Edmonds1}, since $\mathrm{H}_2(\Z_n, \Z)$ is trivial. Indeed the main result of \cite{Edmonds1} is the following.
\begin{theo}[Edmonds]
If $G$ is a finite abelian group, then the natural map from the orbits of the set of surjective homomorphisms of  $\Hom(\Gamma, G)$ under the action of $\Aut(\Gamma)$ to $\mathrm{H}_2(G, \Z)$ is injective.
\end{theo}
This result for the special case $G = \Z_n$ even goes back to Nielsen \cite{Nielsen}.
 Let us nevertheless give a short proof of \cref{ModCycl}.

\begin{proof}
Recall from \cref{rappel} that $\Aut(\Gamma)$ acts as $\Sp_{2g}(\Z)$ on $\Hom(\Gamma, \Z_n) \simeq (\Z_n)^{2g}$. Let us lift $\rho\in \Hom(\Gamma, \Z_n)$ to $\tilde{\rho}\in \Hom(\Gamma, \Z^{2g})\simeq \Z^{2g}$. It is well known that the orbits of the $\Sp_{2g}(\Z)$ action on $\Z^{2g}$ are classified by the greatest common divisor of each of its elements, see for example \cite[Proposition 6.2]{FarbMargalit}. Hence we may assume that $\tilde{\rho} = (k, 0, \ldots, 0)$ with $k\in \Z$. Note that $\gcd(k, n) = 1$, since $\rho$ is surjective. The homomorphism $\rho'=(k, n, 0,\ldots, 0)\in \Z^{2g}$ is another lift of $\rho$ and with the action of $\mathrm{Sp}_{2g}(\Z)$, one can change it to $\rho' = (1, 0,\ldots, 0)$. We now have $\rho(a_1) = 1$, $\rho(b_1) = 0$ and $\rho(a_i) = \rho(b_i) = 0$ for all $i\geqslant 2$.
\end{proof}

\subsection{Dihedral groups}\label{diedrMCG}
In this section we consider the case where $G = \D_n$ is the dihedral group of order $2n$, where $n\geqslant 2$, generated by a rotation of the plane of order $n$ and a symmetry $s$ along an axis of the Euclidean plane. Recall that $\epsilon \in \Hom(\mathrm{O}_2(\R), \Z_2)$ denotes the quotient homomorphism $\mathrm{O}_2(\R)\to \mathrm{O}_2(\R)/\mathrm{SO}_2(\R)$. Observe that there exists a natural isomorphism $\D_n \simeq \Z_2\rtimes \Z_n$ and that $\epsilon$ restricted to $\D_n$ is the composition of this isomorphism with the projection $\Z_2\rtimes \Z_n\to \Z_2$.
Note that for every $x\in \ker \epsilon$ and every $y\in \D_n\setminus \ker \epsilon$, we have $y^{-1}xy = x^{-1}$ and $y = y^{-1}$.

Let us first deal with the case $g = 1$.
\begin{prop}\label{DG1}
If $g = 1$, then there is no surjective homomorphism $\Hom(\Gamma, \D_n)$ unless $n =2$, in which case the action of $\Aut(\Gamma)$ on the set of surjective homomorphisms of $\Hom(\Gamma, \D_n)$ is transitive.
\end{prop}
\begin{proof}
If there exists a surjective homomorphism in $\Hom(\Gamma, \D_n)$ then $\D_n$ must be abelian and $n = 2$. 
The set $\Hom(\Gamma, \D_2)$ has cardinality $16$. Among those $16$ representations, $6$ are surjective. Indeed this set is in bijection with the bases of the $\mathbb F_2$ vector space $\mathbb F_2\oplus \mathbb F_2$ since $\D_2 \simeq \mathbb \Z_2^2$. We leave it to the reader to check that the action of $\Aut(\Gamma)$ on this set is transitive.
\end{proof}

We now suppose until the end of this section that $g\geqslant 2$.

\begin{prop}\label{diedr}
The action of the group $\Aut(\Gamma)$ on the set of surjective homomorphisms of $\Hom(\Gamma, \D_n)$ has two orbits if $n$ is even and one if $n$ is odd.
\end{prop} 
This follows from this theorem of Edmonds which is the main result of \cite{Edmonds2}.
\begin{theo}[Edmonds]
Let $G$ be a split metacyclic group.
The natural map from the $\Aut(\Gamma)$ orbits of surjective homomorphisms of $\Hom(\Gamma, G)$ to $\mathrm H_2(G, \Z)$ is injective.
\end{theo}
The group $\D_n$ is indeed a semi-direct product of cyclic groups: it inserts in the split exact sequence
$$1\to \Z_n \to \D_n \xrightarrow{\epsilon} \Z_2\to 1.$$ However for the sake of completeness we will give a direct proof of \cref{diedr}.
Let us describe these orbits. We define $\rho\in \Hom(\Gamma, \D_n)$ as follows. 
Let $\rho(a_1)\in \ker \epsilon$ be an order $n$ rotation and $\rho(b_1) = \id$. 
We let $\rho(a_i) = \rho(b_i) = \id$ for $1 < i < g$ and $\rho(a_g) = \id$ and  $\rho(b_g) = s$. If $n$ is even, we define $\rho'\in \Hom(\Gamma, \D_n)$ by  $\rho'(\gamma) = \rho(\gamma)$ for $\gamma\in \{a_1, \ldots, b_g\}\setminus \{a_g\}$ and by $\rho'(a_g) = \rho(a_1)^{n/2}$. Observe that $s$ and $\rho(a_1)^{n/2}$ commute since $s \rho(a_1)^{n/2} = \rho(a_1)^{-n/2} s = \rho(a_1)^{n/2} s$. Therefore $\rho'$ is well-defined.
The representations $\rho$ and $\rho'$ are not in the same orbit under the action of $\Aut(\Gamma)$ since $\rho$ lifts to $\SL \C$ while $\rho'$ does not because of the following lemma.

\begin{lemma}\label{sw1}
Let $\alpha, \beta\in \PSL \C$ be order $2$ elements that commute. Either $\alpha = \beta$ or $[\tilde \alpha, \tilde \beta] = - \Id$, where $\tilde \alpha$ and $\tilde \beta$ are lifts of $\alpha$ and $\beta$ in $\SL \C$.
\end{lemma}

\begin{proof}
Suppose $\alpha\neq \beta$. We may conjugate $\alpha$ and $\beta$ so that $\alpha = \pm\begin{pmatrix}
i & 0\\
0 & -i\\
\end{pmatrix}$. The map $\beta$ must interchange the fixed points of $\alpha$ in $\CP$, thus $\beta = \pm\begin{pmatrix}
0 & 1\\
-1 & 0 
\end{pmatrix}$. We have $[\tilde \alpha, \tilde \beta] = -\Id$.
\end{proof}
We turn to the proof of \cref{diedr}.
\begin{proof}
Let $\rho\in \Hom(\Gamma, \D_n)$ with image $\D_n$.
The homomorphism $\epsilon\circ \rho\in \Hom(\Gamma, \Z_2)$ is onto, hence by \cref{ModCycl}, we may assume that $\epsilon\circ\rho(b_g) = 1$, $\epsilon\circ\rho(a_g) = 0$ and  $\epsilon\circ\rho(a_i) = \epsilon\circ\rho(b_i) = 0$ for $i< g$. For $i < g$, both $\rho(a_i)$ and $\rho(b_i)$ are in $\mathrm{SO}_2(\R)$, hence $\rho([a_i, b_i]) = \id$. Therefore $\rho$ restricts to a homomorphism $\Hom(\Gamma_{g-1}, \mathrm{SO}_2(\R))$ with finite image. By \cref{ModCycl} again, we may assume that $\rho(a_1)$ generates the group $\rho(\Gamma_{g-1})$, that $\rho(b_1) = \id$ and that $\rho(a_i) = \rho(b_i) = \id$ for $1 < i < g$. We claim that $\rho(a_1)$ and $\rho(a_g)$ generate $\ker \epsilon$. Indeed every $x\in \ker \epsilon$ is a word in $\rho(a_1)$, $\rho(a_g)$ and $\rho(b_g)$ since $\rho$ is onto. Such a word contains an even number of $\rho(b_g)$ because $x\in \ker\epsilon$. Since $\rho(b_g)^2=\id$ and $\rho(b_g)\rho(a_i) = \rho(a_i)^{-1} \rho(b_g)$ for all $i$, we can write $x$ without any $\rho(b_g)$. We may apply a Dehn twist along a curve homotopic to $a_1a_g$ to replace $\rho(b_1) = \id$ with $\rho(a_1)\rho(a_g)$. Applying again an automorphism of $\Gamma_{g-1}$, we may assume that $\rho(a_1)$ is a chosen generator of $\ker \epsilon$, while $\rho(b_1) = \id$ and  $\rho(a_i) = \rho(b_i)=\id$ for $1 < i < g$.
Let us write $\rho(b_g) = \rho(a_1)^k s$. Let us show that we may assume that $k$ is even. If $k$ is odd and $\rho(a_g) = \rho(a_1)^l$ with $\ell$ odd, then we apply a Dehn twist along a curve homotopic to $a_g$ so that $\rho(b_g)$ is replaced with $\rho(a_g)\rho(b_g) = \rho(a_1)^{k+l}s $. If $\rho(a_g) = \rho(a_1)^l$ with $\ell$ even, then apply a Dehn twist along a curve homotopic to $a_ga_1$. It replaces $\rho(b_g)$ with $\rho(b_g)\rho(a_g)\rho(a_1) = \rho(a_1)^{k+l+1}s$. We may apply an automorphism of $\Gamma_{g-1}$ so that $\rho(b_1) = \id$ again.
We now assume that $\rho(b_g) = \rho(a_1)^{k} s$, with $k = 2k'$.
We conjugate $\rho$ with $\rho(a_1)^{k'}$ so that $\rho(b_g) = s$.
%Let us now apply a point-pushing map: it is the product of Dehn twists along two curves homotopic to $a_1$ bounding an annulus that contains the basepoint of $\Gamma$. It acts on $\Gamma$ as an inner automorphism: a conjugation by  $a_1$. Taking this automorphism to the power $k'$, we replace $\rho(b_g)$ with $s$.
Since $\rho([a_g, b_g]) = \id$, we have $\rho(b_g)^{-1}\rho(a_g)^{-1}\rho(b_g)\rho(a_g) = \rho(a_g)^2 = \id$. Since $\rho(a_g)$ is in the cyclic group $\ker\epsilon$, either $\rho(a_g) = \id$, or $n$ is even and $\rho(a_g) = \rho(a_1)^{n/2}$.
\end{proof}
%
%\begin{rmk}\label{KZ}
%For every $\rho\in \Hom(\Gamma, \D_2)$ there exists $\sigma\in \Aut(\Gamma)$ such that $\sigma\cdot\rho(a_i) = \sigma\cdot \rho(b_i) = \id$.
%Indeed one can exhibit such a representation in each orbit of the action of $\Aut(\Gamma)$ on $\Hom(\Gamma, \D_2)$. 
%In the orbit of surjective homomorphisms that lifts to $\SL \C$, the representation that send $\rho(a_i)$
%If $\rho$ is surjective and lifts to $\SL \C$, then there exists $\rho'$
%\end{rmk}
%
\subsection{The group $\A 4$}
Let us now classify the orbits of the action of $\Aut(\Gamma)$ on the set of surjective homomorphisms of $\Hom(\Gamma, G)$ where $G = \A 4$.
\begin{prop}\label{AutA4}
The action of $\Aut(\Gamma)$ on the set of surjective $\rho\in \Hom(\Gamma, \mathfrak A_4)$ has two orbits.
\end{prop}
%
%Let us describe these two orbits. A first one is the orbit of a representation that send every generator to $\id$ except $a_1$ and $a_2$ that are send to noncommutating $3$-cycles. Another representation is defined by sending every generator to $\id$ except $a_1$ that is send to a $3$-cycle and $a_2$ and $b_2$ to two different double-transpositions. While the former representation obviously lifts to $\SL \C$, the latter does not because of \cref{sw1}. Therefore they cannot lie in the same orbit under the action of $\Aut(\Gamma)$.
%
Recall that the group $\K = \{\sigma\in \mathfrak{A_4} \mid \sigma^2 = \id\}$ is an order $4$ normal subgroup of $\A 4$ isomorphic to $\Z_2\times \Z_2$. The quotient $\A 4/\K$ has order $3$; hence there exists a surjective homomorphism $\varphi : \A 4\to \Z_3$.

Let us describe a representative of one of these two orbits. Let $\rho\in \Hom(\Gamma, \A 4)$ be such that $\rho(a_1)$ is $3$-cycle, $\rho(b_1) = \rho(a_2) = \id$ and $\rho(b_2)\in \K\setminus \{\id\}$ and $\rho(a_i) = \rho(b_i) = \id$ for $3\leqslant i \leqslant g$.
As a first step to prove \cref{AutA4}, we show that all the representations of this form lie in the same orbit under the action of $\Aut(\Gamma)$.
%ere is only one surjective $\rho\in \Hom(\Gamma, \A 4)$ such that $\rho(a_1)$ is a $3$-cycle, $\rho(b_1) = \rho(a_2) = \id$ and $\rho(b_2)\in \K$, and $\rho(a_i) = \rho(b_i) = \id$ for $3\leqslant i \leqslant g$.
\begin{lemma}\label{lemA4}
For every $\rho_0$ and $\rho_1$ of this form, there exists $\sigma\in \Aut(\Gamma)$ such that $\sigma\cdot \rho_0 = \rho_1$.
\end{lemma}
\begin{proof}
%Doing a Dehn twist as in [fig], we change $\rho(a_2)$ by $c\rho(a_2)c^{-1}$ or $c^{-1}\rho(a_2)c$, where $c = \rho(a_1)$. Observe that $c$, $c\rho(a_2)c^{-1}$ and $c^{-1}\rho(a_2)c$ are $3$ pairwise distinct elements in $\K\setminus\{\id\}$. Hence we may choose $\rho(a_2)$ as desired. Note that we may have changed the first handle, but applying a Dehn twist on this handle, we may take back $\rho(b_1)$ to $\id$. Similarly, we can replace $\rho(a_1)$ with $\rho(a_2)\rho(a_1)\rho(a_2)^{-1}$. We check that the set $\sigma \rho(a_1) \sigma^{-1}$ with $\sigma\in \K$ intersect every order $3$ cyclic subgroup of $\A 4$. In the handle $(a_1, b_1)$ we can apply the homeomorphism that replace $(a_1, b_1)$ with $(b_1^{-1} a_1^2, a_1b_1)$, to change $\rho(a_1)$ to its inverse. Therefore we may also choose $\rho(a_1)$ as desired.
By \cref{conjugaison} it suffices to show that they are in the same orbit under the action of $\Aut(\Gamma)\times \mathfrak A_4$.
We may conjugate $\rho_0$ so that $\rho_0(a_1)\in \{\rho_1(a_1), \rho_1(a_1)^{-1}\}$. 
%The subgroup $\Gamma_{1}$ of $\Gamma$ is embedded as the group of automorphisms that fix each $a_i$ and $b_i$ for $i > 1$. 
We can even suppose $\rho_0(a_1) = \rho_1(a_1)$, replacing $\rho_0$ with $\sigma\cdot\rho_0$ by \cref{ModCycl}, where $\sigma\in \Gamma_1$ following the notation of \cref{decoupe}. Conjugating $\rho_0$ with a power of $\rho_0(a_1)$ replaces $\rho_0$ with $\rho_2$. Indeed we have $$\{\rho_0(a_1)^k \rho_0(b_2) \rho_0(a_1)^{-k}\mid 0\leqslant k \leqslant 2\} = \K\setminus \{\id\}.\qedhere$$
\end{proof}

Let us now describe the other orbit.
Let $\rho\in \Hom(\Gamma, \A 4)$ be such that $\rho(a_1)$ is a $3$-cycle, $\rho(b_1) = \id$ and $\rho(a_2)$ and $\rho(b_2)$ are distinct double transpositions and $\rho(a_i) = \rho(b_i) = \id$ for $3\leqslant i\leqslant g$. Observe that such a representation cannot be in the same orbit as the previous one, since it does not lift to $\SL \C$ by \cref{sw1}.
\begin{lemma}\label{lemA4bis}
If $\rho_0$ and $\rho_1$ have this form, then there exists $\sigma\in \Aut(\Gamma)$ such that $\sigma\cdot \rho_0 = \rho_1$.
\end{lemma}

\begin{proof}
As before we can assume that $\rho_0(a_1) = \rho_1(a_1)$.
The action of $\Gamma_1$ on the set of surjective homomorphisms of $\Hom(\Gamma_1, \K)$ is transitive, since $\K\simeq \D_2$. Therefore, we may replace $\rho_0$ with $\sigma\cdot \rho_0$ so that $\rho_0 = \rho_1$, where $\sigma\in \Gamma_1$, following the notations of \cref{decoupe}.
\end{proof}

We now show that any representation may be taken to the form of \cref{lemA4} or \cref{lemA4bis}. This completes the proof of \cref{AutA4}.
\begin{proof}
Let $\rho \in \Hom(\Gamma, \A 4)$ with image $\A 4$. The homomorphism $\varphi \circ \rho$ in $\Hom(\Gamma, \Z_3)$ is surjective. We have seen in \cref{ModCycl} that the action of $\Aut(\Gamma)$ on the set of surjective homomorphisms of $\Hom(\Gamma, \Z_3)$ is transitive. Hence we may assume that $\varphi\circ \rho(a_i) = \varphi\circ \rho(b_i) = 0$ for all $2\leqslant i\leqslant g$ and that $\varphi\circ \rho(a_1) = 1$ and $\varphi\circ \rho(b_1) = 0$. For all $2\leqslant i\leqslant g$, $\rho([a_i, b_i])=\id$ since $\rho(a_i)$ and $\rho(b_i)$ are in the abelian group $\K$. Since $\varphi\circ\rho(a_1)$ has order $3$, $\rho(a_1)$ itself has order $3$ and $\rho(a_1)$ is a $3$-cycle. Moreover $\rho(b_1)\in \K$ and $\rho(b_1)$ commutes with $\rho(a_1)$. The centralizer of $\rho(b_1)$ is $\A 4$ thus $\rho(b_1) = \id$ since $\A 4$ is centerless. We can restrict $\rho$ to the subsurface obtained by collapsing the first handle and define $\rho'\in \Hom(\Gamma_{g-1}, \K)$. The image $\rho'$ cannot be the trivial group, since otherwise $\rho(a_1)$ would generate $\A 4$. Let us first assume that the image of $\rho'$ is cyclic of order $2$. The action of $\Aut(\Gamma_{g-1})$ on the set of surjective $\rho'\in \Hom(\Gamma_{g-1}, \Z_2)$ has only one orbit, hence we may assume that $\rho(a_2)$ has order $2$, while $\rho(a_i) = \rho(b_i) = \id$ otherwise for $i\geqslant 2$. We now turn to the case where the image of $\rho'$ is $\K$. The action of $\Aut(\Gamma_{g-1})$ on the set of surjective homomorphisms of $\Hom(\Gamma_{g-1}, \K)$ has at most two orbits by \cref{diedr} since $\K$ is isomorphic to $\D_2$. If $\rho'$ does not lift to $\SL \C$, then we may assume that $\rho'(a_2)$ and $\rho'(b_2)$ are different order $2$ elements and that $\rho'(a_i) = \rho'(b_i) = \id$ for $3 \leqslant i\leqslant g$. If $\rho'$ lifts to $\SL \C$, then $g\geqslant 3$ and we may assume that $\rho'(a_2)$ and $\rho'(a_3)$ are different elements of $\K$ and that $\rho(a_i) = \rho(b_i) = \id$ otherwise for $3\leqslant i \leqslant g$. After applying \cref{lemA4} in the first two handles, we may assume that $\rho(a_2) = \rho(a_3)$ and get back to the case where the image of $\rho'$ is cyclic. In each case, we can thus assume that $\rho(\gamma) = \id$ for $\gamma\in \{a_3, \ldots, b_g\}$ and that the restriction of $\rho$ on the first two handles has the form of \cref{lemA4} or \cref{lemA4bis}.
\end{proof}

\subsection{The group $\Sy 4$}
We turn to the case $G = \Sy 4$. As before there are two orbits of class of surjective homomorphisms in $\Hom(\Gamma, G)/G$ under the action of $\Mod(\Sigma)$.
\begin{prop}
There are two orbits of surjective homomorphisms of $\Hom(\Gamma, \Sy 4)$ under the action of $\Aut(\Gamma)$.
\end{prop}

The group $\K = \{\sigma\in \A 4\mid \sigma^2 = 1\}$ is also normal in $\Sy 4$ and the quotient group $\Sy 4/\K$ is isomorphic to $\Sy 3\simeq \D_3$. Let us denote by $\varphi$ a surjective homomorphism of $\Hom(\Sy 4, \Sy 3)$.
%
%Let $\rho\in \Hom(\Gamma, \Sy 4)$ with image $\Sy 4$. The homomorphism $\varphi\circ \rho\in \Hom(\Gamma, \Sy 3)$ is surjective, hence by \cref{diedr}, we may assume that $\varphi\circ\rho(a_1)$ is a $3$-cycle, $\varphi\circ\rho(a_2)$ is a transposition, and that the other generators are send to $\id$ by $\varphi\circ \rho$. Since $\rho(a_i)$ and $\rho(b_i)$ are in $\K$ for $i \geqslant 3$, collapsing the first two handles gives a representation in $\Hom(\Gamma_{g-2}, \K)$, and we may assume that $\rho(a_i) = \rho(b_i) = \id$ if $i \geqslant 5$ by \cref{diedr}. Let us show that we can also suppose that the restriction of $\rho$ to the first two handles is surjective. It suffices to show that we can assume that the transposition $\rho(a_2)$ does not fix the fixed point of the $3$-cycle $\rho(a_1)$ in $\{1, 2, 3, 4\}$, because their product has then order $4$, and group they generate has order at least $12$. Since it is not contained in $\A 4$, it is $\Sy 4$. If $\rho(b_2)$ does not fix the fixed point of $\rho(a_1)$, then we may assume that $\rho(a_2)$ does not either, changing $a_2$ to $b_2a_2$. If they both fix its fixed point, applying a Dehn twist along a curve homotopic to $b_3[b_1, a_1]b_2$ or $a_3^{-1}[b_1,a_1]b_2$, replace $\rho(a_2)$ with another transposition that does not fix the fixed point of $\rho(a_1)$.
Let us first suppose that $g=2$.
\begin{lemma}\label{nonsep}
There exists a nonseparating simple closed curve $\gamma\in \Gamma$ such that $\rho(\gamma) = \id$.
\end{lemma}

\begin{proof}
We suppose that $\rho(b_1) \neq \id$ and $\rho(b_2) \neq \id$ since otherwise we are done. Since $\rho(a_1)$ is a $3$-cycle, it does not commute with $\rho(b_1)=\rho(b_1)^{-1}$, thus $\rho(b_1)$, $\rho(a_1^{-1}b_1^{-1}a_1)$ and $\rho(a_1 b_1^{-1} a_1^{-1})$ are $3$ distinct elements of $\K\setminus \{\id\}$. One of them is therefore equal to $\rho(b_2) = \rho(b_2)^{-1}$. 
The loop $b_2b_1$ in the first case, $b_2a_1^{-1}b_1^{-1}a_1$ in the second case and $b_2a_1 b_1^{-1} a_1^{-1}$ in the last one is sent to $\id$ and is represented by a simple nonseparating curve.
\end{proof}

We suppose thanks to \cref{nonsep} that $\rho(b_1) = \id$. We also assume that $\rho(a_1)\in \A 4$. Indeed, if $\rho(a_1)\notin \A 4$, then suppose $\rho(a_2)$ or $\rho(b_2)$ is not in $\A 4$. We may assume that $\rho(b_2)\notin \A 4$ and replace $\rho(a_1)$ with $\rho(b_2)\rho(a_1)$ with a Dehn twist along a curve homotopic to $b_1b_2$. Suppose now that both $\rho(a_2)$ and $\rho(b_2)$ are in $\A 4$. Since $[\rho(a_2), \rho(b_2)] = \id$, the group generated by $\rho(a_2)$ and $\rho(b_2)$ is either cyclic or $\K$. It cannot be $\K$ since one can check that $\Sy 4$ is not generated by $\K$ together with a single $\sigma\in \Sy 4$. 
Since the group generated by $\rho(a_2)$ and $\rho(b_2)$ is cyclic, we may assume that $\rho(b_2) = \id$ and, interchanging the handle, that $\rho(a_1)\in \A 4$.

Since $[\rho(a_2), \rho(b_2)] = \id$, and  $\rho(a_2)$ and $\rho(b_2)$ are not both in $\A 4$, we may assume that one of the following holds:
\begin{enumerate}
\item $\rho(a_2)$ is a $4$-cycle and $\rho(b_2) = \id$
\item $\rho(a_2)$ is a transposition and $\rho(b_2) = \id$
\item $\rho(a_2)$ and $\rho(b_2)$ are two distinct transpositions that commute.
\end{enumerate} 

We leave it to the reader to check that such $\rho(a_2)$ and $\rho(b_2)$ together with a double transposition do not generate $\Sy 4$. Therefore $\rho(a_1)$ is a $3$-cycle. If we are in the second case above, then a Dehn twist replaces $\rho(a_2)$ with $\rho(a_2a_1)$, which is a $4$-cycle. Indeed, it cannot be a transposition: since $\rho(a_1)$ and $\rho(a_2)$ generate $\Sy 4$, $\rho(a_2)\rho(a_1)$ does not fix the fixed point of $\rho(a_1)$ nor the fixed points of $\rho(a_2)$. 
We may then apply a Dehn twist in the first handle so that $\rho(b_1) = \id$ again.

Let us show that two representations $\rho_0$ and $\rho_1$ that are in the first case above are in the same orbit under the action of $\Aut(\Gamma)$. 
By \cref{conjugaison} we may conjugate $\rho_0$ so that $\rho_0(a_1) = \rho_1(a_1)$. Conjugating $\rho_0$ by some $\rho(a_1)^{k}$, $0\leqslant k \leqslant 2$, we may also assume that $\rho_0(a_2) = \rho_1(a_2)^{\pm 1}$, thanks to the following observation whose proof is left to the reader. 
\begin{lemma}
%\begin{enumerate}
%\item If $\sigma$ is a $3$-cycle and $\tau$ is a $4$-cycle, then $\{\tau^{-k} \sigma \tau^k\}$ intersect every cyclic subgroup of order $3$ of $\mathfrak S_4$.
%\item 
If $\sigma$ is a $3$-cycle and $\tau$ is a $4$-cycle, then $\{\sigma^{-k} \tau \sigma^k\}$ intersect every cyclic subgroup of order $4$ of $\mathfrak S_4$.
%\end{enumerate}
\end{lemma}
Applying an automorphism of the second handle, thanks to \cref{ModCycl}, we may replace $\rho_0(a_2)$ with $\rho_0(a_2)^{-1}$ if necessary.

Let us now suppose that $\rho_0$ and $\rho_1$ both have the form of the third case above: $\rho_0(a_2)$ and $\rho_0(b_2)$ are different transpositions that commute. We may conjugate $\rho_0$ so that $\rho_0(a_1) = \rho_1(a_1)$. There are three possibilities for $\{\rho_0(a_2), \rho_0(b_2)\}$ and we can go from one to the other by conjugating by $\rho_0(a_1)^k$, $0\leqslant k \leqslant 2$. Therefore we may assume that $\{\rho_0(a_2), \rho_0(b_2)\} = \{\rho_1(a_2), \rho_1(b_2)\}$. Since the action of $\Aut(\Gamma_1)$ on the set of surjective homomorphisms of $\Hom(\Gamma_1, \K)$ is surjective, we can assume that $\rho_0 = \rho_1$. Representations in the first case above cannot be in the same orbit as a representation in the third case since the former lifts to $\SL \C$ while the latter does not by \cref{sw1}.
As a corollary of this classification in genus $2$, we show that if $\rho : \Gamma_2\to \Sy 4$ is surjective, then every $g\in \Sy 4$ is the image of a nonseparating simple closed curve by $\rho$.
\begin{lemma}\label{rpz}
For every surjective $\rho\in \Hom(\Gamma_2, \Sy 4)$ and $g\in \Sy 4$, there exists $\sigma\in \Aut(\Gamma_2)$ such that $\sigma\cdot \rho (b_1) = g$.
\end{lemma}
\begin{proof}
For every conjugacy class of $g\in \Sy 4$, it suffices to exhibit a representation $\rho$ in each of the two $\Aut(\Gamma)$-orbits such that $\rho(b_1) = g$. We leave it to the reader to check this for each of the $5$ conjugacy classes of $\Sy 4$.
\end{proof}

Let us now turn to the case $g\geqslant 3$. The homomorphism $\varphi\circ\rho$ in $\Hom(\Gamma, \Sy 3)$ is surjective. Hence we may assume that $\varphi\circ\rho(a_1)$ is a $3$-cycle, $\varphi\circ\rho(a_2)$ is a transposition and $\varphi\circ\rho(\gamma) = \id$ for the others $\gamma\in \{a_1, \ldots, b_g\}$. Let us show that the number of $\gamma\in \{a_3, \ldots, b_g\}$ such that  $\rho(\gamma) = \id$ can be lowered. Iterating this process, we get back to the genus $2$ case. 

The group generated by $\rho(\gamma)$ for $\gamma\in \{a_1, b_1, a_2, b_2\}$ has order divisible by $6$ and is not included in $\A 4$. Therefore it is $\Sy 4$ or it is isomorphic to $\Sy 3$. In the case where it is isomorphic to $\Sy 3$, $\rho(a_2)$ must be a transposition and after conjugating we may assume that $\rho(a_2) = (12)$. Let $\gamma\in \{a_3, \ldots, b_g\}$ be such that $\rho(\gamma)\neq \id$. We may assume that $\gamma \in \{ a_3, b_3\}$, applying an automorphism of $\Gamma_{g-2}$ that interchanges the handles. We can even suppose that $\gamma = a_3$: we may apply the transformation $(x,y)\mapsto (y, -x)$ of $\mathrm{Sp}_2(\Z)\subset \mathrm{Sp}_{2g-2}(\Z)$.
We can apply a Dehn twist along a curve homotopic to $[a_g, b_g]\ldots[a_3, b_3]a_1$ if necessary so that $\rho(\gamma)$ does not preserve the set $\{1, 2\}$. It does not change $\rho(\gamma)$ for $\gamma\in \{a_1, a_2, b_2\}$ and we can turn back $\rho(b_1)$ to its previous value with a Dehn twist in the first handle. A Dehn twist along a curve homotopic to $a_3^{-1} b_2$ replaces $\rho(a_2)$ with a $4$-cycle. Hence we may assume that the group generated by $\rho(\gamma)$ for $\gamma\in \{a_1, b_1, a_2, b_2\}$ is $\Sy 4$.

% We assume that $\gamma = a_3$; the reader may check that one can do the same reasoning when $\gamma = b_3$.
%We have seen that the action of $\Aut(\Gamma_2)$ on the surjective homomorphisms of $\Hom(\Gamma_2, \Sy 4)$ has two orbits, and that in each one of these orbits, there exists a representative that send $b_1$ to any given $g\in \Sy 4$.
Therefore, applying an automorphsim of $\Gamma_2$, we may assume that $\rho(b_1)=\rho(a_3)\rho(b_3)$ by \cref{rpz}.  We then apply a Dehn twist along a curve homotopic to $b_3^{-1}b_1$, that replaces $\rho(a_3)$ with $\rho(a_3)\rho(b_1)^{-1}\rho(b_3)^{-1}=\id$. This does not change $\rho(b_3)$ and conjugate the $\rho(a_i)$ and $\rho(b_i)$ for $i\geqslant 4$. Since $\varphi\circ\rho(\gamma)= \id$ for $\gamma\in \{a_3, \ldots, b_g\}$, we may apply an automorphism of $\Gamma_2$ to turn back to the case where $\varphi\circ\rho(a_1)$ is a $3$-cycle, $\varphi\circ \rho(a_2)$ is a transposition, and $\varphi\circ\rho(b_1) = \varphi\circ\rho(b_2) = \id$.

\subsection{The group $\A 5$}
Let us classify the orbits of the $\Aut(\Gamma)$ action on the set of surjective homomorphisms of $\Hom(\Gamma, \A 5)$.
\begin{prop}\label{modA5}
There are two orbits of surjective homomorphisms of $\Hom(\Gamma, \A 5)$ under the action of $\Aut(\Gamma)$.
%The action of $\Aut(\Gamma)$ on the surjective homomorphisms of $\Hom(\Gamma, \A 5)$ has two orbits.
\end{prop}
We first find a simple closed curve that is sent to a $5$-cycle.
\begin{lemma}\label{5cycle}
Let $g'\geqslant 2$ and $\rho \in \Hom(\Gamma_{g', 1}\, \A 5)$ be such that $\rho(\Gamma_{g', 1})$ acts transitively on $\{1, \ldots, 5\}$. There exists $\sigma\in \Aut(\Gamma_{g', 1})$ such that $\sigma\cdot\rho(a_1)$ is a $5$-cycle.
\end{lemma}

\begin{proof}
It suffices to find a nonseparating simple closed curve $\gamma\in \Gamma_{g',1}$ such that $\rho(\gamma)$ is a $5$-cycle, since for every nonseparating simple closed curve $\gamma\in \Gamma_{g',1}$ there exists $\sigma\in \Aut(\Gamma_{g',1})$ such that $\sigma\cdot \gamma\in \{a_1, a_1^{-1}\}$.
The product of two double transpositions $\tau$ and $\tau'$ that do not share the same fixed point is not a double transposition. Indeed if $\tau\tau'$ has order $2$, then $\tau\tau' = (\tau\tau')^{-1} = \tau'\tau$ thus $\tau \tau' \tau^{-1} = \tau'$ and $\tau$ and $\tau'$ have the same fixed point. Let us first suppose that every $\rho(\gamma)$ for $\gamma\in \{a_1, \ldots, b_{g'}\}$ has order $2$ or $1$. Two of them do not share the same fixed point, hence their product has order $3$ or $5$. Therefore we may assume that $\rho(a_1)$ has order $3$. Let us denote by $p$ and $q$ the fixed points of $\rho(a_1)$.

\begin{lemma}\label{tech}
Let $p\neq q\in \{1, \ldots, 5\}$ and $\alpha$, $\beta\in \A 5$. If $\alpha(p) = p$ and $\alpha(q) \neq q$, and $\beta(q) = q$ and $\beta(p)\neq p$, then $\beta\alpha$ does not interchange $p$ and $q$, nor does it fix $p$ or $q$. If $\tau$ interchanges $p$ and $q$, then $\tau \alpha$ does not fix nor interchanges $p$ and $q$. These two statements also hold for $\alpha\beta$ and $\alpha\tau$.
\end{lemma}

\begin{proof}
For the first part, $\beta(\alpha(p)) = \beta(p) \neq p$ and $\beta(\alpha(q)) \neq q$ because $\alpha(q)\neq q$. Moreover if $\beta\alpha(p) = \beta(p) = q$, then $\beta(q) = q = \beta(p)$ and $q=p$.
Let us prove the second part. We have $\tau(\alpha(p)) = \tau(p) = q$ and $\tau(\alpha(q)) \neq p$ since otherwise $\alpha(q) = q$. Moreover $\tau(\alpha(q)) = q$ if and only if $\alpha(q) = p$. But $\alpha(q) = p = \alpha(p)$ implies that $p = q$.
\end{proof}
Let us make some observations on the order of the product of elements of $\A 5$ whose proof are left to the reader.
\begin{lemma}\label{produit}
If $\sigma$ is a $3$-cycle, then the product of $\sigma$ with a double transposition that does not fix its fixed points, nor interchanges them is a $5$-cycle. The product of $\sigma$ with a $3$-cycle that does not share a fixed point with $\sigma$ is also a $5$-cycle. The product of a double transposition with another double transposition that does not share the same fixed point, nor a transposition, is a $5$-cycle.\qed
\end{lemma}

If one of $\rho(\gamma)$, $\gamma\in \{b_1, a_2, \ldots, b_g\}$ does not fix $p$ nor $q$ nor interchanges them, then the product of $\rho(a_1)$ with this permutation has order $5$. If it is not the case, then we can find one of them that does not fix $p$ but fixes $q$ , interchanging $p$ and $q$ if necessary.
Indeed since $\rho(\Gamma_{g',1})$ acts transitively on $\{1, \ldots, 5\}$, one of them does not preserve the set $\{p, q\}$.
 Another one of these permutations does not fix $q$ and either interchanges $p$ and $q$, or fixes $p$.
Let us say for clarity that these permutations are $\rho(a_2)$ and $\rho(b_2)$. Then $\rho(b_2^{-1}a_2)$ does not fix $p$ nor $q$ nor interchanges them by \cref{tech}, hence $\rho(b_2^{-1}a_2a_1)$ is a $5$-cycle.
\end{proof}
We now show that we may assume that both $\rho(a_i)$ and $\rho(b_i)$ fix $5$ for every $i\geqslant 2$.
\begin{prop}\label{reduction}
There exists $\sigma\in \Aut(\Gamma)$ such that for all $2\leqslant i \leqslant g$, both $\sigma\cdot\rho(a_i)$ and $\sigma\cdot\rho(b_i)$ fix $5$.
\end{prop}

We may ensure that $\rho(a_i)$ fixes $5$ by applying an automorphism of $\Aut(\Gamma_{1,1})$, unless both $\rho(a_i)$ and $\rho(b_i)$ are double transpositions that commute.
\begin{lemma}\label{boom}
There exists $\sigma\in \Aut(\Gamma_{1,1})$ such that $\sigma\cdot \rho(a_1)$ fixes $5$, unless both $\rho(a_1)$ and $\rho(b_1)$ are double transpositions that fix some $p\in \{1, \ldots, 4\}$, and $\alpha\neq \beta$.
\end{lemma}

\begin{proof}
Let  $\alpha = \rho(a_1)$ and $\beta = \rho(b_2)$.
If $\beta$ fixes $5$, then one can apply the automorphisms that acts on $\Hom(\Gamma_{1,1}, \A 5)$ as $(\alpha, \beta)\to (\alpha, \alpha\beta)\to (\beta^{-1}, \alpha\beta)$. If $\alpha(5)$ is fixed by $\beta$, then one can apply the automorphisms $(\alpha, \beta)\to (\alpha, \alpha^{-1}\beta)\to (\alpha^{-1}\beta\alpha, \alpha^{-1}, \beta)$ to replace $\alpha$ with $\alpha^{-1}\beta \alpha$. This new $\alpha$ fixes $5$. Similarly, if $\beta(5)$ is fixed by $\alpha$ then one can replace $\beta$ with $\beta^{-1}\alpha \beta$. If $\alpha(5)$ is in the orbit of $5$ by the action of the group generated by $\beta$, then one can apply the Dehn twist that replaces $\alpha$ with $\beta^k\alpha$ so that $\beta^k \alpha$ fixes $5$. Similarly we can suppose that $\beta(5)$ is not in the orbit of 5 under the action of the group generated by $\alpha$. Therefore, if $\alpha$ or $\beta$ is a cycle, then there exists such a $\sigma$. 
Indeed if $\alpha$ is a cycle that does not fix $5$, either $\beta(5)$ is fixed by $\alpha$ or $\beta(5)$ is in the orbit of $5$ by the action of $\langle\alpha\rangle$.
%Indeed if one of them is a $5$-cycle, say $\alpha$, then $\beta(5)$ is in the orbit of $5$ under the action of the group $\langle \alpha \rangle$. If one of them is a $3$-cycle, say $\alpha$, then either $\beta(5)$ is in the orbit of $5$ under the action of $\langle \alpha \rangle$, or $\beta\alpha$ is a $5$-cycle and we get back the previous case with $(\alpha, \beta)\to (\beta \alpha, \beta)$. 
Let us now suppose that both $\alpha$ and $\beta$ are double transpositions. We may replace $\alpha$ with $\beta\alpha$. Either $\beta\alpha$ is not a double transposition, or $\alpha$ and $\beta$ are different double transpositions that share a common fixed point by \cref{produit}. 
\end{proof}
%
%\begin{lemma}\label{tuer5}
%If $\rho(a_1)$ is a $5$-cycle, then there exists $\sigma\in \Aut(\Gamma)$ such that both $\sigma\cdot \rho(a_2)$ and $\sigma\cdot\rho(b_2)$ are in $\A 4$.
%\end{lemma}
%
%\begin{proof}
%We may assume that $\rho(b_2)$ is not a $5$-cycle. Indeed $\rho(a_2)$ is not a $5$-cycle, we apply a homeomorphism of the handle $(a_2, b_2)$ that replace $b_2$ by $a_2^{-1}$. If however both $\rho(b_2)$ and $\rho(a_2)$ are $5$-cycle, then there exists $k\geqslant 0$ such that $\rho(a_2^k b_2)$ fixes $5$ because the cyclic group generated by $\rho(a_2)$ acts transitively on $\{1, \ldots, 5\}$, and we may change $\rho(b_2)$ to $\rho(a_2^kb_2)$. We can moreover assume that $\rho(b_2)$ fixes $5$. Indeed a Dehn twist as in [fig] conjugate $\rho(a_2)$ and $\rho(b_2)$ by $\rho(a_1)^k$. Now a Dehn twist as in [fig] replaces $\rho(a_2) $ with $\sigma_k = \rho(a_2)\rho(b_1^{-1})\rho(a_1)^k\rho(b_2^{-1})$. We may choose $k\geqslant 0$ such that $\sigma_k$ fixes $5$.
%\end{proof}

We now prove \cref{reduction}.
%By \cref{tuer5}, it suffices to show that there exists a nonseparating simple curve in the subsurface generated by the $k_\rho$ first handles that is mapped by $\rho$ to a $5$-cycle. 
\begin{proof}
Let us assume that for all $k_\rho+1\leqslant i \leqslant g$, both $\rho(a_i)$ and $\rho(b_i)$ fix $5$, and that $k_\rho$ is minimal for that property. We may also assume that for all $\sigma\in \Aut(\Gamma)$ we have $k_{\rho} \leqslant k_{\sigma\cdot \rho}$. Assume by contradiction that $k_\rho \geqslant 2$. Let us denote by $G$ the group generated by $\rho(\gamma)$ for $\gamma\in \{a_1, \ldots, b_{k_\rho}\}$.
The group $G$ is included in $\A 5$, hence it is either $\A 5$, cyclic or isomorphic to $\A 4$, $\Sy 3$, or $\D_2$. 

\begin{itemize}
\item If $G$ is cyclic, then by \cref{ModCycl}, one can assume applying an automorphism of $\Gamma_{k_\rho}$ that $\rho(a_{k_\rho}) = \rho(b_{k_\rho}) = \id$, contradicting the minimality of $k_\rho$.

\item If $G$ is isomorphic to $\A 4$, then it is embedded in $\A 5$ as the subgroup that fixes a point $p\in \{1, 2, 3, 4\}$. Since  
\[\rho([a_1, b_1] \ldots [a_{k_\rho}, b_{k_\rho}]) = \rho([a_{k_\rho+1}, b_{k_\rho+1}]\ldots [a_g, b_g])^{-1}\]
the permutation $\rho([a_1, b_1] \ldots [a_{k_\rho}, b_{k_\rho}])$ has two fixed points and is in a group isomorphic to $\D_2$, thus it is $\id$. Therefore, we may restrict the representation $\rho$ to  $\Hom(\Gamma_{k_\rho}, G)$. By \cref{AutA4}, there are two orbits of surjective $\rho' : \Gamma_{k_\rho}\to G$, under the action of $\Aut(\Gamma_{k_\rho})$. In each of these orbits, there exists a representation $\rho'$ such that $\rho'(a_{k_\rho})$ is any desired $3$-cycle  and $\rho'(b_{k_\rho}) = \id$. Hence we may assume that both $\rho(a_{k_\rho})$ and $\rho(b_{k_\rho})$ fix $5$, a contradiction.

\item Let us suppose that $G$ is isomorphic to $\Sy 3$. Since $\rho([a_1, b_1] \ldots [a_{k_\rho}, b_{k_\rho}])$ is in the commutator subgroup of $G$, it has order $3$ or $1$. It is also in the commutator subgroup of $\A 4$ and thus must be trivial. Hence we may restrict $\rho$ to $\Hom(\Gamma_{k_\rho}, \Sy 3)$. The action of $\Aut(\Gamma_{k_\rho})$ on the set of surjective homomorphisms of $\Hom(\Gamma_{k_\rho}, G)$ is transitive, hence we may assume that $\rho(a_{k_\rho})$ is any fixed element of $G\setminus \{\id\}$ and $\rho(b_{k_\rho})= \id$ since there exists a surjective representation $\Gamma_{k_\rho}\to G$ with this property. Therefore we can assume that both $\rho(a_{k_\rho})$ and $\rho(b_{k_\rho})$ fix $5$, since in each subgroup $G$ of $\A 5$ isomorphic to $\Sy 3$, there exists $g\in G\setminus \{\id\}$ that fixes $5$: up to conjugacy we have
$$G = \{\id, (1\ 2\ 3), (1\ 3\ 2), (1\ 2)(4\ 5), (2\ 3)(4\ 5), (1\ 3)(4\ 5)\}.$$

\item Let us now assume that $G$ is isomorphic to $\D_2$. We may restrict $\rho$ to $\Hom(\Gamma_{k_\rho}, \D_2)$. If this restriction does not lift to $\SL \C$, then one can assume that $\rho(a_1)$ and $\rho(b_1)$ are different double transpositions and that $\rho(a_{k_\rho}) = \rho(b_{k_\rho}) = \id$, a contradiction. Similarly, if $k_\rho > 2$, one can assume that $\rho$ kills both $a_{k_\rho}$ and $b_{k_\rho}$. Let us thus suppose that $k_\rho = 2$ and that the restriction of $\rho$ to the first two handles lifts to $\SL \C$. We may restrict $\rho$ to the last $g-2$ handles and get a representation $\rho'\in \Hom(\Gamma_{g-2}, \A 4)$. If $\rho'(\Gamma_{g-2})$ is isomorphic to $\Z_3$, $\Z_2$ or $\D_2$ and does not lift to $\SL \C$, then one can assume that $\rho(a_i) = \rho(b_i) = \id$ for all $i > 3$. We may then conjugate $\rho$ so that $\rho(\gamma)$ fix $5$ for $\gamma\in \{a_1, b_1, a_2, b_2\}$ and thus get $k_\rho = 1$. If $\rho'(\Gamma_{g-2})= \A 4$, we may assume that $\rho(a_3)$ is a $3$-cycle that fixes the fixed point of $\rho(\gamma)$ for $\gamma\in \{a_1, b_1, a_2, b_2\}$ and $\rho(b_3) = \id$. We also assume that $\rho(a_i) = \rho(b_i) = \id$ for $i > 4$. Again, conjugating $\rho$ we get $k_\rho = 1$. Let us turn to the case where $\rho'(\Gamma_{g-2})$ is isomorphic to $\D_2$ and lifts to $\SL \C$.
After conjugating $\rho$ and applying automorphisms of $\Gamma_{g-2}$ and of $\Gamma_2$, we may assume that 
$\rho(a_1) = (1\ 3)(2\ 5)$, $\rho(a_2) = (1\ 2)(3\ 5)$, $\rho(a_3) = (1\ 2)(3\ 4)$, $\rho(a_4) = (1\ 3)(2\ 4)$, and $\rho(b_i) = \id$ for all $i$ and $\rho(a_i) = \id$ for all $i > 4$. We may restrict $\rho$ to the second and third handles $\Hom(\Gamma_2, \A 5)$. The image of $\rho(\Gamma_2)$ is isomorphic to $\D_3 \simeq \Sy 3$. Since the action of $\Aut(\Gamma_2)$ on the set of surjective homomorphisms of $\Hom(\Gamma_2, \Sy 3)$ is transitive, we may assume that $\rho(a_2) = (3\ 4\ 5)$, $\rho(a_3) = (1\ 2)(3\ 4)$ and $\rho(b_2) = \rho(b_3) = \id$.
We now apply a Dehn twist along a curve homotopic to $[a_g, b_g]\ldots [a_3, b_3]a_2$. This conjugate the first handle so that $\rho(a_1)$ fixes $5$. Hence for all $i\neq 2$, both $\rho(a_i)$ and $\rho(b_i)$ fix $5$, thus $k_\rho = 1$.

\item Let us suppose that $G = \A 5$. By \cref{5cycle}, we may assume that $\rho(a_1)$ is a $5$-cycle. By \cref{boom} we may assume that $\rho(b_2)$ fixes $5$ or that both $\rho(a_2)$ and $\rho(b_2)$ fix some $p\in \{1, 2, 3, 4\}$. In this latter case we may apply a power of a Dehn twist along a curve homotopic to $[a_2, b_2]a_1$ so that both $\rho(a_i)$ and $\rho(b_i)$ fix $p$ for $k_\rho <i \leqslant g$. We then conjugate $\rho$ to interchange $5$ with $p$. We have added a hanlde that fixes $5$, which is a contradiction. We thus turn to the case where $\rho(b_2)$ fixes $5$. A Dehn twist as in \cref{Dehn twist} along a curve homotopic to $b_2^{-1} a_1^k b_1^{-1}$ replaces $\rho(a_2) $ with $\sigma_k = \rho(a_2)\rho(b_1^{-1})\rho(a_1)^k\rho(b_2^{-1})$. We may take $k$ so that it fixes $5$. \qedhere
\end{itemize}
\end{proof}
\begin{figure}
    \centering    
    \def\svgwidth{\columnwidth}
%  	\def\svgwidth{1\textwidth}
%	\hspace*{-1cm}
  	\def\svgwidth{0.7\textwidth}

	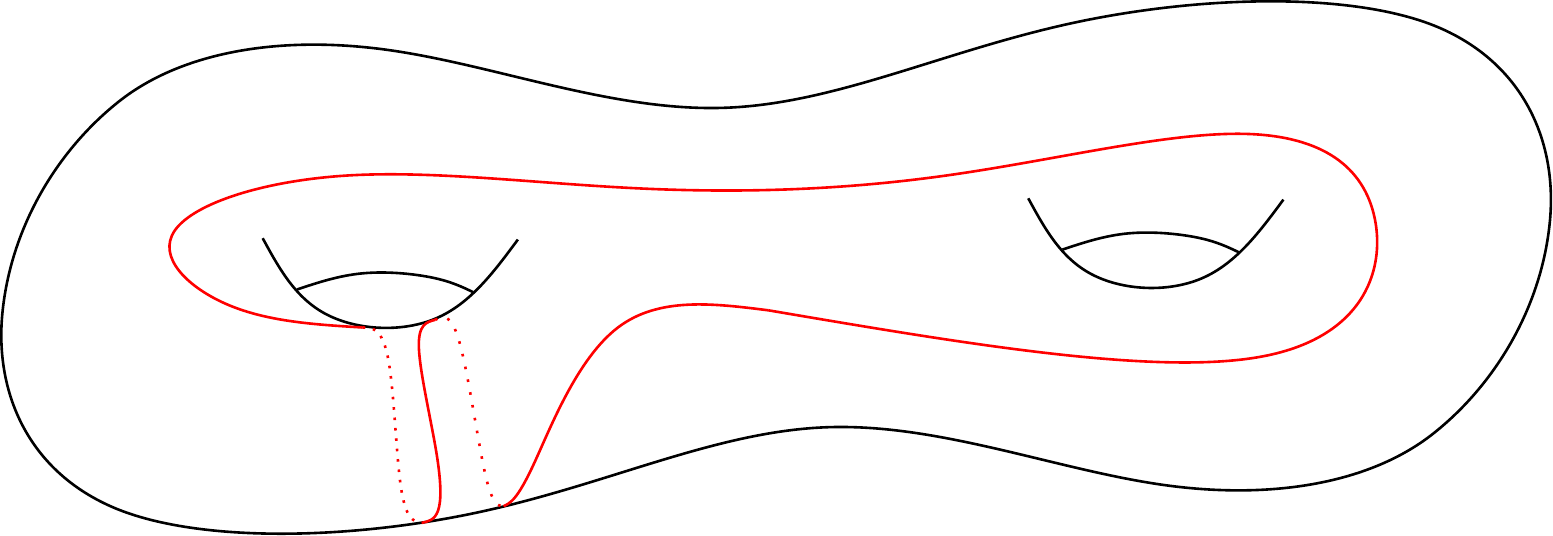
	\caption{Dehn twist along a simple closed curve}
    \label{Dehn twist}
\end{figure}

Let us now observe that if $\rho(\gamma)$ fixes $5$ for each $\gamma\in \{a_2, \ldots, b_g\}$, then $\rho(a_1)$ commutes with $\rho(b_1)$.
\begin{lemma}
Let $\alpha, \beta\in \A5$ be such that $[\alpha, \beta]$ has order $2$, and denote by $p\in \{1, \ldots, 5\}$ the fixed point of $[\alpha, \beta]$. Both $\alpha$ and $\beta$ fix $p$.
\end{lemma}
\begin{proof}
Let us suppose that $[\alpha, \beta]$ has order $2$ and fixes $5$. We assume for contradiction that $\alpha$ does not fix $5$.
If $\alpha$ is a double transposition, then we may assume that $\alpha = (5\ k)(i\ j)$ for $i,j,k\in \{2, 3, 4, 5\}$. The double transposition $\beta \alpha^{-1} \beta$ has the form $(5\ k)(x,y)$, since $[\alpha, \beta]$ must fix $k$. Therefore $[\alpha, \beta]$ fixes both $5$ and $k$, this is not possible since it is a double tranposition.
If $\alpha = (5\ i\ j)$ is a $3$-cycle, then $\beta^{-1}\alpha^{-1}\beta$ is a $3$-cycle of the form $(i\ 5\ k)$, since $[\alpha, \beta]$ fixes $5$. But $|\{i,j,k,5\}|\leqslant 4$, hence $[\alpha, \beta]$ fixes some $p\in \{1, 2, 3, 4\}$ and cannot have order $2$.
If $\alpha$ is a $5$-cycle, say $\alpha = (1\ 2\ 3\ 4\ 5)$. We must have $\beta^{-1} \alpha^{-1} \beta = (1 \ 5\ i\ j\ k)$ with $\{i,j,k\} = \{2, 3, 4\}$. We cannot have $k = 2$, otherwise $[\alpha, \beta]$ would fix $1$. We cannot have $i = 4$, otherwise $[\alpha, \beta]$ would fix $4$.
If $i = 3$, then $j = 2$ and $[\alpha, \beta]$ fixes $2$. If $i = 2$, then $\beta^{-1}\alpha^{-1}\beta = (1\ 5\ 2\ 3\ 4)$ or $\beta^{-1}\alpha^{-1}\beta = (1\ 5\ 2\ 4\ 3)$. In this latter case, $[\alpha, \beta]$ fixes $3$. We thus suppose that $\beta^{-1}\alpha^{-1}\beta = (1\ 5\ 2\ 3\ 4)$. Since $\beta' = (2\ 4)$ satisfies $\beta'^{-1} \alpha^{-1} \beta = \beta^{-1}\alpha^{-1} \beta$, the permutation $\beta\beta'^{-1}$ must be in the centralizer of $\alpha^{-1}$, and thus be a power of $\alpha$. Therefore $\beta\in \Sy 5\setminus \A 5$.
\end{proof}

We now suppose thanks to \cref{reduction} that $\rho(\gamma)$ fix $5$ for $\gamma\in \{a_2, \ldots, b_g\}$.
Note that since $\rho(a_1)$ and $\rho(b_1)$ do not both fix $5$, the commutator $[\rho(a_1), \rho(b_1)] = \id$ since it is a product of commutators of permutations fixing $5$. Therefore the representation restricted on the last $g-1$ handles gives $\rho : \Gamma_{g-1}\to \A 4$. The image of this representation is a subgroup of $\A 4$: it is isomorphic to $\A 4$, $\D_2\simeq \Z_2\times \Z_2$,  $\Z_3$ or $\Z_2$ or the trivial group. In each of these cases, we have seen that we may assume that $\rho(a_i)=\rho(b_i)=\id$ for all $i\geqslant 4$, and that $\rho([a_2,b_2]) = \id$.
We are thus reduced to the genus $g = 3$ case. In order to deal with this case, we first prove \cref{modA5} for $g=2$.
\subsubsection{The genus 2 case}
\begin{prop}\label{gen2}
The action of $\Aut(\Gamma_2)$ on the set of surjective homomorphisms of $\Hom(\Gamma_2, \A 5)$ has two orbits.
\end{prop}

Let us describe a representation in each of these orbits. Let $\rho_0$ be the representation defined by $\rho_0(a_1) = (1\ 2\ 3\ 4\ 5)$, $\rho_0(b_1) = \rho_0(b_2) = \id$ and $\rho_0(a_2) = (1\ 5\ 4\ 3\ 2)$. Let $\rho_1$ be the representation defined by $\rho_1(a_1) = (1\ 2\ 3\ 4\ 5)$ and $\rho_1(b_1) = \id$, $\rho_1(a_2) = (1\ 2) (3\ 4)$ and $\rho_1(b_2) = (1\ 3)(2\ 4)$. While $\rho_0$ lifts to $\SL \C$, the representation $\rho_1$ does not by \cref{sw1}.

Let us fix a surjective homomorphism $\rho\in \Hom(\Gamma, \A 5)$.
By \cref{reduction}, we can assume that $\rho(a_2)$ and  $\rho(b_2)$ both fix $5$, and that $\rho([a_1,b_1]) = \rho([a_2, b_2]) = \id$. Since the centralizer of a cycle in $\A 5$ is the group it generates, if $\rho(a_1)$ or $\rho(b_1)$ has order $3$, we may assume that $\rho(b_1) = \id$. The same holds in the second handle.

Let us first show that if $\rho$ does not lift to $\SL \C$, then there exists $\sigma\in \Aut(\Gamma)$ such that $\sigma\cdot\rho = \rho_1$.
\begin{lemma}
The action of $\Aut(\Gamma)$ on the set of surjective homomorphisms of $\Hom(\Gamma, \A 5)$ that do not lift to $\SL \C$ is transitive.
\end{lemma}

\begin{proof}
We suppose that $\rho$ does not lift to $\SL \C$.
By \cref{sw1} and the remark above, $(a_2, b_2)$ must be sent to two different double transpositions. 
We may assume that $\rho(b_1) = \id$: if $\rho(a_1)$ is a cycle, since $[\rho(a_1), \rho(b_1)] = \id$, $\rho(b_1)$ is a power of $\rho(a_1)$. If $\rho(a_1)$ is a double transposition, then since $\rho$ lifts to $\SL \C$ and $\rho(b_1)$ commutes with $\rho(a_1)$, $\rho(b_1)$ must again be a power of $\rho(a_1)$. The group generated by $\rho(a_1)$ and $\rho(b_1)$ is thus cyclic and we may assume with an automorphism of $\Gamma_1$ that $\rho(b_1) = \id$. Since $\A 5$ is not abelian, $\rho(a_1)\neq \id$. We may assume that that $\rho(a_1)$ is a $5$-cycle. If it is a double transposition or a $3$ cycle, then we may apply an automorphism of $\Gamma_1$ in the handle $(a_2, b_2)$ so that $\rho(a_2)\rho(a_1)$ is a $5$-cycle, by \cref{produit}. A Dehn twist along a curve homotopic to $a_2^{-1}b_1$ replaces $\rho(a_1)$ with $\rho(a_1)\rho(a_2)^{-1}$ which is a $5$-cycle. Note that $\rho(a_1)$ and $\rho(b_1)$ are still distinct double transpositions that commute.
We can conjugate $\rho$ so that $\rho(a_1) = (1\ 2\ 3\ 4\ 5)$. We can then conjugate $\rho$ by $\rho(a_1)^k$ so that $\rho(a_1)$ and $\rho(b_1)$ fix $5$. Since the action of $\Aut(\Gamma_1)$ on the set of surjective homomorphisms $\Hom(\Gamma_1, \D_2)$ is transitive, we may assume that $\rho(a_1) = (1\ 2)(3\ 4)$ and $\rho(b_1) = (1\ 3)(2\ 4)$, and $\rho = \rho_1$.
\end{proof}

We now suppose that $\rho$ lifts to $\SL \C$. We still assume that $\rho([a_1,b_1]) = \rho([a_2, b_2]) =\id$.
\begin{lemma}
There exists a nonseparating simple closed curve $\gamma$ such that~$\rho(\gamma)=\id$.
\end{lemma}

\begin{proof}
We have already explained why this holds when one of the handles $(a_1, b_1)$ and $(a_2, b_2)$ does not have its generators mapped by $\rho$ to distinct double transpositions. We now turn to this case.
Conjugating $\rho$ if necessary, and then applying a homeomorphism in each handle, we may assume that $\rho(a_1) = (1\ 2)(3\ 4)$, $\rho(b_1) = (1\ 3)(2\ 4)$, $\rho(a_2) = (1\ 3)(2\ 5)$ and $\rho(b_2) = (1\ 5)(2\ 3)$.
The loop $\gamma = b_2^{-1}a_2b_1^{-1} a_2b_1^{-1}a_1^{-1}a_2b_1^{-1}$ is simple, nonseparating, and $\rho(\gamma) = \id$.
\end{proof}
 
We can thus assume that $\rho(b_1) = \id$. Note that we may also assume that $\rho(b_2) = \id$. Indeed if $\rho(a_2)$ and $\rho(b_2)$ are distinct double transpositions, then $\rho$ does not lift do $\SL \C$. Since $[\rho(a_2), \rho(b_2)] = \id$, and that the commutator of a cycle is the group it generates, we may assume that $\rho(b_2) = \id$. One of $\rho(a_1)$ and $\rho(a_2)$ has order at least $3$ because  the group $\A 5$ is not generated by two double transpositions.
Therefore we suppose $\rho(a_1)$ has order $3$ or $5$. The following lemma enables us to assume that it has order $3$.

\begin{lemma}\label{lm}
If $\rho(a_1)$ is a $5$-cycle, then we may replace $\rho(a_1)$ with a double transposition or a $3$-cycle without changing $\rho(b_1)$, $\rho(a_2)$ and $\rho(b_2)$.
\end{lemma}

\begin{proof}
Let $p\in \{1, \ldots, 5\}$ be such that $\rho(a_2)$ does not fix $p$. There exists $1\leqslant k\leqslant 4$ such that $\rho(a_1)^k\rho(a_2)$ fixes $p$. We apply an automorphism of $\Gamma_1$ in the handle $(a_1, b_1)$, that replaces $\rho(a_1)$ with $\rho(a_1)^k$ and leaves $\rho(b_1) = \id$ unchanged; it exists by \cref{ModCycl}. We then apply a Dehn twist along a curve homotopic to $a_2b_1^{-1}$ to replace it with $\rho(a_1)^k \rho(a_2)$ that fixes $p$. We still have $\rho(b_1) = \id$, and we may apply an automorphism of $\Gamma_1$ in the second handle $(a_2, b_2)$ that does not change $\rho(a_2)$ and turns back $\rho(b_2)$ to $\id$.
%
%If $\rho(a_2)$ is a $3$ cycle, suppose that $\rho(a_1)^k\rho(a_2)$ is a double transposition, and fix $q\neq p$ that is not fixed by $\rho(a_2)$. There exists $\ell$ such that $\rho(a_1)^{l+k}\rho(a_2)$ fixes $q$. Let us show that it is not a double transposition. If $\rho(a_1)^{l+k}\rho(a_2) = \rho(a_1)^l \rho(a_1)^k \rho(a_2)$ where a double transposition, it would be true for all $\ell$ by \cref{bonnetransposition} with $\tau = \rho(a_1)^k \rho(a_2)$, hence $\rho(a_2)$ would have order $2$ itself.
\end{proof}
We can now assume that $\rho(a_1)$ is a $3$-cycle: we may assume that both $\rho(a_1)$ and $\rho(b_1)$ have order 2 or 3 by \cref{lm}, and they cannot both have order $2$. We may interchange the handles if $\rho(a_1)$ is not a $3$-cycle.
If $\rho(a_2)$ is a double transposition, then we apply a Dehn twist along a curve homotopic to $a_1b_2^{-1}$ that replaces it with $\rho(a_2)\rho(a_1)$ which is a $5$-cycle by \cref{produit}. 
By \cref{lm} we may assume that $\rho(a_2)$ is a $3$-cycle without modifying $\rho(\gamma)$ for $\gamma\in \{a_1, b_1, b_2\}$.
%We can choose $1\leqslant k\leqslant 4$ so that $\rho(a_1)\rho(a_2)^k$ is a $3$-cycle according to the following lemma.
%
%\begin{lemma}
%For every pair $x,y\in \A 5$ of generators of $\A 5$ such that $x$ is a $3$-cycle, and $y$ is a $5$-cycle, there exists $1\leqslant k\leqslant 4$ such that $xy^k$ is a $3$-cycle.
%\end{lemma}
%We can apply an automorphism of the first handle to replace $\rho(a_2)$ with $\rho(a_2)^k$ for any $1\leqslant k\leqslant 4$ by \cref{ModCycl}.
%A Dehn twist along a curve homotopic to $a_2b_2^{-1}$ then replaces $\rho(a_2)$ with $\rho(a_1)\rho(a_2)$, and an automorphism of the first handle may turn back $\rho(b_1)$ to $\id$. Therefore we may assume that $\rho(a_1)$ and $\rho(a_2)$ are both $3$-cycle, and that $\rho(b_1) = \rho(b_2) = \id$.

We now show that all the homomorphisms of this form are in the same orbit under the action of $\Aut(\Gamma_2)$, thus proving \cref{gen2}.
\begin{proof}
We may conjugate $\rho$ so that $\rho(a_1) = (1\ 2\ 3)$. Let us write $\rho(b_2) = (\alpha\ \beta\ \gamma)$, with $\alpha\in \{1, 2, 3\}$ and $\{\beta, \gamma\} = \{4, 5\}$. Conjugating $\rho$ with a power of $\rho(a_1)$, we can suppose that $\alpha  = 1$, that is $\rho(a_2) = (1\ 4\ 5)$ or $\rho(a_2) = (1\ 5\ 4)$. By \cref{ModCycl}, there exists an automorphism of $\Gamma_1$ that replaces $\rho(a_2)$ with $\rho(a_2)^{-1}$, and we can thus assume that $\rho(a_2) = (1\ 4\ 5)$. 
\end{proof}

Let us state a corollary of this classification that will be useful to finish our classification of the orbits in genus at least $3$.
\begin{cor}\label{anseSp}
Let $g\in \A 5\setminus\{\id\}$.
For every surjective $\rho\in \Hom(\Gamma_2, \A 5)$, there exists $\sigma\in \Aut(\Gamma_2)$ such that $\sigma \cdot \rho(a_1) = g$ and $\sigma \cdot \rho(b_1) = \id$.
\end{cor}

\begin{proof}
There exists $h\in \A 5$ such that $g$ and $h$ generate $\A 5$. The representation $\rho'$ defined by $\rho'(a_1) = g$, $\rho'(b_1) = \rho'(b_2) = \id$ and $\rho'(a_2) = h$ lifts to $\SL \C$. Therefore, if $\rho$ lifts to $\SL \C$, then there exists $\sigma\in \Aut(\Gamma_2)$ such that $\sigma\cdot \rho = \rho'$.
There exist two distinct commuting double transpositions $h_1, h_2\in \A 5$ such that $h_1$, $h_2$ and $g$ generate $\A 5$: it suffices to take two distinct transpositions with a fixed point $p\in \{1, \ldots, 5\}$ that is not fixed by $g$, since the proper subgroups of $\A 5$ of order divisible by $4$ have a fixed point in $\{1, \ldots, 5\}$. Let $\rho''$ be the representation defined by $\rho''(a_1) = g$, $\rho''(b_1) = \id$ and $\rho''(a_2) = h_1$ and $\rho''(b_2) = h_2$. The representation $\rho''$ does not lift to $\SL \C$ by \cref{sw1}. Therefore, if $\rho$ does not lift to $\SL \C$, then there exists $\sigma\in \Aut(\Gamma_2)$ such that $\sigma\cdot \rho = \rho''$.
\end{proof}

\subsubsection{Genus $g\geqslant 3$}

The proof of \cref{modA5} is now reduced to the case $g = 3$. Therefore, the following lemma completes its proof.
\begin{lemma}
Let us assume that $\rho \in \Hom(\Gamma_3, \A 5)$ is surjective. There exists $\sigma\in \Aut(\Gamma_3)$ such that $\sigma\cdot\rho(a_3) =\sigma\cdot\rho(b_3) = \id$.
\end{lemma}

\begin{proof}
Recall that we may assume that $\rho([a_1, b_1])= \rho([a_2, b_2]) = \rho([a_3, b_3]) = \id$, and that the group $H$ generated $\rho(\gamma)$ for $\gamma\in \{a_2, b_2, a_3, b_3\}$ is a subgroup of $\A 4$, that is embedded as the subgroup whose elements fix $5$. Let us act by $\Aut(\Gamma_2)$ on the last two handles. The subgroups of $\A 4$ are cyclic, isomorphic to $\D_2\simeq \Z_2\times \Z_2$ and $\A 4$. We have classified the orbits of the $\Aut(\Gamma_2)$ action on $\Hom(\Gamma_2, H)$ in the previous sections. We saw that we may assume that $\rho(a_3) = \rho(b_3) =\id$ unless $H$ is isomorphic to $\A 4$, or $\D_2$, in which case we can assume that the group $G$ generated by the $\rho(\gamma)$, $\gamma\in \{a_1, b_1, a_2, b_2\}$, is isomorphic to $\A 5$ and that $\rho(a_3)\neq \id$. Indeed if $H$ is isomorphic to $\D_2$, we may assume that $\rho(a_2)$ and $\rho(b_2)$ are two distinct double transpositions. Since $\rho(a_1)$ and $\rho(b_1)$ do not both fix $5$, the group $G$ is $\A 5$ because one can check that the proper subgroups of $\A 5$ of order divisible by $4$ have a fixed point $p\in \{1, \ldots, 5\}$. If $H$ is isomorphic to $\A 4$, then we may suppose that $\rho(a_2)$ is a $3$-cycle such that $\rho(a_1), \rho(b_1)$ and $\rho(a_2)$ acts transitively on $\{1, \ldots 5\}$. We then have $G = \A 5$ because the only subgroup of order divisible by $3$ that acts transitively on $\{1, \ldots, 5\}$ is $\A 5$ itself.
We saw that the action of $\Aut(\Gamma_2)$ on the set of surjective homomorphism of $\Hom(\Gamma_2, \A 5)$ has two orbits. Therefore, by \cref{anseSp}, we may apply an automorphism of the first two handles so that $\rho(a_2) = \rho(a_3)^{-1}$ and that $\rho(b_2) = \id$.  
A Dehn twist along a curve homotopic to $a_3b_2^{-1}$ replaces $\rho(a_2)$ with $\id$, and we still have $\rho(b_2) = \id$. We can permute the second and the third handles so that $\rho(a_3) = \rho(b_3) = \id$.
\end{proof}

\section{Closure of infinite orbits of the mapping class group action on representation spaces into elementary subgroups of $\PSL \C$}\label{SINF}
We now describe the closure of the orbits of some representations $\rho\in \Hom(\Gamma, G)$ with infinite image under the action of $\Aut(\Gamma)$, where $G$ is an elementary subgroup of $\PSL \C$. We will also show some other results describing these orbits.
\subsection{Holonomy in $\SO$}
\subsubsection{The group $\mathrm{SO}_2(\R)$}
Let $\rho\in \Hom(\Gamma, \mathrm{SO}_2(\R))$ with infinite image. The orbit of $\rho$ under the action of $\Aut(\Gamma)$  is dense in $\Hom(\Gamma, \mathrm{SO}_2(\R))$.
\begin{prop}\label{ptitlemme}
The orbit $\Aut(\Gamma)\cdot \rho$ is dense in $\Hom(\Gamma, \mathrm{SO}_2(\R))$.
\end{prop}

We identify $\mathrm{SO}_2(\R)$ with $\R/\Z$ and $\Hom(\Gamma, \mathrm{SO}_2(\R))$ with $(\R/\Z)^{2g}$ on which $\Aut(\Gamma)$ acts as $\Sp_{2g}(\Z)$, see \cref{rappel}.

\begin{proof}
Let $\rho_n = (\frac{1}{n}, 0, \ldots, 0)\in \Hom(\Gamma, \mathrm {SO}_2(\R))$. Let us consider a lift $\tilde \rho$ of $\rho$ to $\Hom(\Gamma, \R) \simeq \R^{2g}$. Let $x_i = \tilde \rho(a_i)$ and $y_i = \tilde \rho(b_i)$ for $1\leqslant i \leqslant g$. One of $x_1, \ldots, y_g$ is not rational and we can assume without loss of generality that $x_1$ is as such: we can replace $\rho$ with some $\sigma\cdot \rho$, exchanging the handles so that $x_1$ or $y_1$ is not rational and we may then apply the symplectic transformation $(x_1, y_1)\mapsto (y_1, -x_1)$ if necessary. We can apply the transformation that replaces $(x_1, y_1)$ with $(x_1, y_1 - y_i)$ and $(x_i, y_i)$ with $(x_1 + x_i, y_i)$ for each $i$ if necessary so that $x_i\notin \mathbb Q$. For every $i$ such that $x_i\Z + y_i \Z$ is not dense in $\R$, replace $y_i$ with $y_i + 1$. The new $\tilde \rho$ is still a lift of $\rho$ and now $ x_i\Z + y_i \Z$ is dense in $\R$ because $\frac{y_i}{x_i}$ and $\frac{y_i +1}{x_i}$ cannot both be rationals. We may assume that each $x_i$ and $y_i$ are nonnegative, applying $(x_i, y_i)\mapsto (x_i + k y_i, y_i)$ and $(x_i, y_i)\mapsto (x_i, y_i + lx_i)$, for some $k,l\in \Z$. We now apply the Euclidean algorithm in each handle: if $x_i > y_i$, replace $x_i$ with $x_i -y_i$ and otherwise replace $y_i$ with $y_i - x_i$. Iterating this process, we can assume that $0 < x_i, y_i < \epsilon$ for every $i$, where $\epsilon > 0$ is fixed. Finally replacing $x_1$ with $x_1 + m y_1$, $m\in \Z$, we can assume that $|x_1 - \frac{1}{n}| < \epsilon$ and that the other $x_i$ and $y_i$ are in $[0, \epsilon]$.

We have shown that each $\rho_n$ is in the closure of $\Aut(\Gamma)\cdot \rho$. Since this closure is invariant under $\Aut(\Gamma)$, it contains every representation in $\Hom(\Gamma, \mathrm{SO}_2(\R))$ with finite image by \cref{ModCycl}. Therefore it contains the closure of $(\mathbb Q/\Z)^{2g}$, hence is equal to $\Hom(\Gamma, \mathrm{SO}_2(\R))$.
\end{proof}

\subsubsection{The group $\mathrm O_2(\R)$}
Let $\rho\in \Hom(\Gamma, \mathrm O_2(\R))$ be such that $\overline {\rho(\Gamma)} = \mathrm O_2(\R)$. Note that this amounts to requiring that $\rho$ has infinite image in $\mathrm{O}_2$ and is not contained in $\mathrm{SO}_2(\R)$. The orbit of such a $\rho$ under the action of $\Aut(\Gamma)$ is included in the closed set $\Delta$ of representations $\rho'\in \Hom(\Gamma, \mathrm O_2(\R))$ such that $\epsilon\circ \rho'$ is surjective and such that $\sw(\rho') = \sw(\rho)$.
\begin{prop}\label{denseDiedr}
The orbit $\Aut(\Gamma)\cdot \rho$ is dense in $\Delta$.
\end{prop}

\begin{proof}
It suffices to show that every representation $\rho'\in \Hom(\Gamma, \D_n)$ with finite image such that $\epsilon\circ \rho'$ is surjective and $\sw(\rho) = \sw(\rho')$ is in $\Aut(\Gamma)\cdot \rho$.
The representation $\epsilon\circ \rho\in \Hom(\Gamma, \Z_2)$ is surjective, hence we may assume by \cref{ModCycl} that $\epsilon\circ\rho(\gamma) = 0$ for $\gamma\in \{b_1, a_2\ldots, b_g\}$ and that $\epsilon(\rho(a_1)) = 1$. We thus have $[\rho(a_i), \rho(b_i)] = \id$ for every $1\leqslant i \leqslant g$. Therefore $\rho(b_1)$ is either $\id$ or the rotation of order $2$ in $\mathrm{SO}_2(\R)$. 
Let us consider a representation $\rho'$ with finite image in $\mathrm{O_2}$, such that $\epsilon\circ \rho$ is surjective and $\sw(\rho) = \sw(\rho')$.
By \cref{diedr}, we may assume that $\rho'(\gamma) = \rho(\gamma)$ for $\gamma\in \{a_1, b_1\}$. Since the image of $\rho$ is dense in $\mathrm O_2(\R)$, one of the $\rho(\gamma)$ has infinite order, where $\gamma\in \{a_2, \ldots, b_g\}$. 
By \cref{ptitlemme}, there exists a sequence $\sigma_n$ in $\Aut(\Gamma_{g-1})$ such that $\sigma_n\cdot \rho \to \rho'$, following the notations of \cref{decoupe}. Therefore $\rho'\in \overline{\Aut(\Gamma)\cdot \rho}$.
\end{proof}
\subsubsection{The group $\SO$}

In this subsection, we consider representations $\rho\in \Hom(\Gamma, \SO)$ such that the image of $\rho$ is dense in $\SO$.
Let us recall the main result of \cite{PreviteXia2}.
\begin{theo}\label{PrePreXia}[Previte-Xia]
Let $g' \geqslant 1$ and $b \geqslant 0$.
Let $\rho\in \Hom(\Gamma_{g',b}, \mathrm{SU}_2)$, and let us denote by $[\rho]$ its conjugacy class in $\Hom(\Gamma_{g',b}, \mathrm{SU}_2)/\mathrm{SU}_2$. If $\rho(\Gamma_{g',b})$ is dense in $\mathrm{SU}_2$ then the closure of the orbit $\Mod(\Sigma)\cdot [\rho]$ is the set $$\{\rho' : \Gamma_{g', b}\to \mathrm{SU}_2\mid \forall i\leqslant b\ \exists h\in \mathrm{SU}_2,  \rho'(c_i) = h\rho(c_i) h^{-1}\}/\mathrm{SU}_2.$$
\end{theo}

Let us recall that the representation space $\Hom(\Gamma, \SO)/\SO$ has two connected components: one of which is the set of conjugacy classes of representations that lift to $\mathrm{SU}_2$, and the other the set of conjugacy classes of representations that do not lift to $\mathrm{SU}_2$.
It follows from \cref{PrePreXia} that the orbit of a representation in $\Hom(\Gamma, \SO)$ with dense image is dense in its connected component of $\Hom(\Gamma, \SO)/\SO$.
\begin{lemma}\label{previtxia}
The $\Mod(\Sigma)$ orbit of $[\rho]\in \Hom(\Gamma, \SO)/\SO$ with dense image is dense in its connected component.
\end{lemma}
\begin{proof}
Let us see $\rho$ in $\Hom(\Gamma_{g,1}, \mathrm P\SU)$ where $\rho(c_1) = \id$. There exists a lift $\tilde \rho\in \Hom(\Gamma_{g,1}, \SU)$ such that $\tilde \rho(\partial \Sigma_{g,1}) = \epsilon\Id$ where $\epsilon = 1$ if $\rho$ lifts to $\SL \C$ and $\epsilon = -1$ otherwise. \cref{PrePreXia} ensures us that the orbit of  $[\tilde \rho]$ is dense in its relative character variety, hence its projection in $\Hom(\Gamma, \SO)/\SO$ is dense in its connected component.
\end{proof}

\subsection{Affine holonomy}\label{AffineHol}
%\color{red} Discours de politique générale / Ne pas oublier de dire qu'on suppose ici que les rep ne sont pas dans $\SO$\color{black}

In this subsection we consider representations $\rho\in \Hom(\Gamma, \mathrm{Aff}(\C))$, where the group $\mathrm{Aff}(\C)$ is the subgroup of $\PSL \C$ that fix $\infty\in \CP$. In other words, $\mathrm{Aff}(\C) = \{\pm \begin{pmatrix}
\lambda & t\\
0 & \lambda^{-1}
\end{pmatrix}\mid \lambda\in \C^*, t\in \C\}$. The map $\mathrm{Aff}(\C)\to \C^*\rtimes \C$ defined by $\pm \begin{pmatrix}
\lambda & t\\
0 & \lambda^{-1}
\end{pmatrix}\mapsto (\lambda^2, \lambda^{-1}t)$ is an isomorphism. We denote by $az + b$ the element of $\mathrm{Aff}(\C)$ corresponding to $(a, b)\in \C^*\rtimes \C$.
We identify a representation $\rho$ with the $2g$-tuple $(\rho(\gamma))_{\gamma\in \{a_1, \ldots, b_g\}}$.
Let us first observe that a representation in $\Hom(\Gamma, \mathrm{Aff}(\C))$ lifts to $\SL \C$.
\begin{lemma}
Every representation $\rho\in \Hom(\Gamma, \mathrm{Aff}(\C))$ lifts to $\SL \C$.
\end{lemma}
\begin{proof}
Let us pick lifts $\tilde \rho(\gamma) = \begin{pmatrix}
\tilde \lambda_\gamma & \tilde t_\gamma\\
0 & \tilde \lambda_\gamma^{-1}
\end{pmatrix}$ in $\SL \C$ of $\rho(\gamma)$ for each $\gamma\in\{a_1, \ldots, b_g\}$. The product of the commutators $[\tilde \rho(a_i), \tilde \rho(b_i)]$ is $\Id$.
\end{proof}
We denote by $\mathrm{Li} : \C^*\rtimes \C\to \C^*$ the projection on the first coordinate and by $\mathrm{Tr} : \C^*\rtimes \C\to \C$ the projection on the second coordinate. Note that $\mathrm{Li}$ is a homomorphism while $\mathrm{Tr}$ satisfies $\mathrm{Tr}(h_1h_2) = \mathrm{Tr}(h_1) + \mathrm{Li}(h_1)\mathrm{Tr}(h_2)$ for every $h_1,h_2\in \mathrm{Aff}(\C)$.
Recall that an affine representation $\rho$ induces a representation in $\C^*$, \textit{its linear part} $\mathrm{Li}\circ \rho$. We also define the \textit{translation part} of $\rho$ to be the map $\mathrm{Tr}\circ \rho$.

We identify $\Hom(\Gamma, \C^*)$ with the set of affine representations $\rho$ such that $\mathrm{Tr}(\rho) = 0$.
The linear part of a representation is in the closure of its orbit under the action of $\mathrm{Aff}(\C)$ by conjugation.
\begin{lemma}\label{lineaire}
Let $\rho\in \Hom(\Gamma, \mathrm{Aff}(\C))$. The representation $\mathrm{Li}\circ\rho$ is in the closure of $\mathrm{Aff}(\C)\cdot \rho$.
\end{lemma}

\begin{proof}
The sequence $(g_n^{-1}\rho g_n)_n$ tends to $\mathrm{Li}\circ \rho$, where for every $n\geqslant 1$, $$g_n = \begin{pmatrix}
n & 0\\
0 & 1/n
\end{pmatrix}.$$
\end{proof}

Let us denote by $\mathbb S^1$ the unit circle $\{z\in \C \mid z\overline z = 1\}$. The group $\mathrm{Isom}^+(\mathbb E^2)$ is the subgroup $\mathbb S^1\rtimes \C$ of $\C^*\rtimes \C = \mathrm{Aff}(\C)$. It is the group of positive isometries of the Euclidean plane.

Let us suppose for now that $\rho$ is Euclidean: $\rho\in \Hom(\Gamma, \mathrm{Isom}^+(\mathbb E^2))$.
%We also suppose that there exists $\gamma\in \Gamma$ such that $t(\gamma)\neq 0$, where $t = \mathrm{Tr}(\rho)$, since otherwise we are reduced to the case $\rho\in \Hom(\Gamma,\SO)$.
The representation $\rho$ acquires a notion of \textit{volume}, see also \cite[Section 6]{Selim1}.

\begin{lemma}\label{forme_volume}
There exists a smooth $\rho$-equivariant function $f : \tilde \Sigma\to \C$. Moreover, the pullback of $dx\wedge dy$ by $f$ induces a form $\omega$ on $\Sigma$ whose volume $\int_\Sigma \omega$ does not depend on $f$.
\end{lemma}

\begin{proof}
Since $\Hom(\Gamma, \mathrm{Isom}^+(\mathbb E^2))$ is path-connected, the bundle $\tilde \Sigma\times \C / \Gamma$ is topologically trivial, where $\Gamma$ acts by $\gamma\cdot (\tilde x, z) = (\gamma\cdot\tilde x, \rho(\gamma)\cdot z)$. Hence it admits smooth sections. We can lift such a section to a function $F : \tilde \Sigma\to \tilde \Sigma\times \C$ of the form $F(\tilde x) = (\tilde x, f(z))$. The function $f : \tilde \Sigma\to \C$ is smooth and $\rho$-equivariant.

Given two smooth $\rho$-equivariant functions $f_0$ and $f_1$, the smooth function $H : (x,t)\in \tilde \Sigma\times [0,1]\mapsto tf_1(x) + (1-t)f_0(x)$ is also $\rho$-equivariant. The pullback $H^*(dx\wedge dy)$ is $\Gamma$-invariant and thus induces a form $\omega$ on $\Sigma\times [0,1]$. Observe that this form is closed since $dx\wedge dy$ is closed. Therefore Stokes' theorem gives:
\[0 = \int_{\Sigma\times [0,1]} d\omega = \int_{\Sigma\times \{1\}} \omega - \int_{\Sigma\times \{0\}} \omega.\]
The restriction of $\omega$ to $\Sigma\times \{i\}$ is the pullback of $dx\wedge dy$ by $f_i$ for $i=0,1$.
\end{proof}

\begin{defn}
The volume of $\rho$ is the number $\mathrm{Vol}(\rho) = \int_\Sigma \omega$, where $\omega$ is the induced volume form on $\Sigma$ by the pullback of $dx\wedge dy$ by a smooth $\rho$-equivariant function $f : \tilde \Sigma\to \C$.
\end{defn}

\begin{rmk}
The volume is invariant under the action of $\mathrm{Aut}^+(\Gamma)$ and under the action of $\mathrm{Isom}^+(\mathbb E^2)$ by conjugation.
\end{rmk}

\subsubsection{Finite linear part}\label{finite_lin}
Let $\rho\in \Hom(\Gamma, \mathrm{Isom}^+(\mathbb E^2))$ be such that $\alpha = \mathrm{Li}\circ \rho$ has finite image of order $n\geqslant 1$. We also suppose that $\rho$ is not spherical.

We exhibit some representations in the closure of the orbit of $\rho$ under the action of $\Aut(\Gamma)\times \mathrm{Aff}(\C)$. Using our classification of the orbits of $\Hom(\Gamma, \Z_n)$ under the action of $\Aut(\Gamma)$, we provide a simple form of $\rho$ if $n\geqslant 2$.

\begin{lemma}\label{OOOD}
If $n\geqslant 2$, then there exists $\rho'\in \Aut(\Gamma)\times \C\cdot \rho$ such that $$\rho' = (\zeta z, z, z + z_3, z + z_4, \ldots, z + z_{2g}).$$
Where $\C$ is the subgroup of $\mathrm{Isom}^+(\mathbb E^2)$ of translations, and $\zeta = e^{2\pi i / n}$.
\end{lemma}

\begin{proof}
There exists $\sigma\in \Aut(\Gamma)$ such that $\sigma\cdot \alpha = (\zeta, 1, \ldots, 1)$ by \cref{ModCycl}.
Therefore, replacing $\rho$ with $\sigma\cdot\rho$ we have $$\rho = (\zeta z + z_1, z + z_2, z + z_3, \ldots, z + z_{2g}).$$
Now we can conjugate $\rho$ by a translation so that $z_1 = 0$. The product of the commutators $[\rho(a_i), \rho(b_i)]$ is $[\zeta z, z + z_2]$ and must be trivial. Thus $\zeta z$ commutes with $z + z_2$ hence $z_2 = 0$.
\end{proof}

There are even simpler representations in the closure of the orbit of $\rho$ under the action of $\Aut(\Gamma)\times \mathrm{Aff}(\C)$.
\begin{lemma}\label{ecrasement}
Let $\rho_\infty\in \Hom(\Gamma, \mathrm{Isom}^+(\mathbb E^2))$ be the representation defined by 
$$\rho_\infty = \left\{
\begin{array}{ll}
	(z + 1, z, z,\ldots z) & \mbox{if } n= 1\\
	(\zeta z, z, z + 1, z, \ldots, z, z)  & \mbox{if } n \geqslant 2.
\end{array}
\right.$$
We have $\rho_\infty \in \overline{\Aut(\Gamma)\times \mathrm{Aff}(\C)\cdot \rho}$.
\end{lemma}
\begin{proof}
We first suppose that $n = 1$.
Let us write  $\rho = (z + z_1, \ldots, z + z_{2g})$.
We have supposed that $\rho$ is not conjugated in $\SO$, hence there exists $1\leqslant i \leqslant 2g$ such that $z_i\neq 0$. Exchanging the handles if necessary we can assume that $1\leqslant i\leqslant 2$. Applying a Dehn twist in the first handle if needed, we may even assume that $i = 2$. Conjugating $\rho$ with  $z\mapsto z_2 z$, we may assume that $z_2 = 1$. A power of a Dehn twist along a curve homotopic to $b_1$ replaces $\rho$ with $$\rho_N = (z + N + z_1, z + 1, z + z_3, \ldots z + z_{2g}).$$
The representation $h_N^{-1}\rho_N h_N$ converges to $\rho_\infty$, where $h_N = \begin{pmatrix}
\sqrt{N} & 0\\
0 & 1/\sqrt N
\end{pmatrix}$.
We now suppose that $n\geqslant 2$.
We may start with $\rho$ as in \cref{OOOD}.
 Since $\rho$ is not conjugated in $\SO$, there exists $3\leqslant i \leqslant 2g$ such that $z_i\neq 0$. We then proceed as in the case $n=1$.
\end{proof}

If $n\in \{1, 2\}$ then the group $\mathrm{GL}_2^+(\R)$ acts on $\Hom(\Gamma, \Z_n\rtimes \C)$. This action is defined as follows: if $A\in \mathrm{GL}_2^+(\R)$ and $\rho\in \Hom(\Gamma, \Z_2\rtimes \C)$, then $A\cdot \rho$ is defined to be the representation $\rho' : \gamma\mapsto \alpha(\gamma)z + A\cdot t(\gamma)$, where $\alpha = \mathrm{Li}\circ \rho$ and $t = \mathrm{Tr}\circ\rho$.
Let us make an observation that will be useful later in the case $n = 2$.
\begin{lemma}\label{lemmeGenre2}
Let $x,y\in \C$.
The representation $(-z, z, z + x, z + y)$ in $\Hom(\Gamma_2, \Z_2\rtimes \C)$ is in the same orbit as the representation $(z+x, -z + y, z-x, -z)$ under the action of $\Aut(\Gamma_2)$.
\end{lemma}
\begin{proof}
Let us start with the representation $\rho = (-z, z, z + x, z + y)$. We may apply an automorphism of $\Aut(\Gamma_1)$ in the first handle so that $\rho = (z, -z, z + x, z + y)$ by \cref{ModCycl}. We then apply a Dehn twist along a curve homotopic to $a_2^{-1}b_1$, that replaces $\rho$ with $(-z + x, -z, z + x, -z + x + y)$. A Dehn twist in the second handle replaces $\rho(b_2)$ with $-z + y$. We may now apply a Dehn twist in the first handle and then interchange the two handles so that $\rho = (z+x, -z + y, z-x, -z)$.
\end{proof}

Let us consider the group $\Lambda = \Lambda(\rho)\subset \C$ of $\mathrm{Tr}(\rho(\gamma))$ for $\gamma\in \ker \alpha$. The group $\Lambda$ may be a lattice in $\C$ or not. This property is invariant under the action of $\Aut(\Gamma)$. Kapovich gave in \cite{Kapovich} a description of the closure of orbits of the $\Aut(\Gamma)\times \mathrm{GL}_2^+(\R)$ action on $\Hom(\Gamma, \C)$, see also \cite[Proposition 2.10]{CalsamigliaDeroin}. As a corollary of this description we show that if $g\geqslant 3$ and $n = 1$ and $\rho$ of positive volume then there exists $\rho_\infty \in \overline{\Aut(\Gamma)\cdot \rho}$ such that $\Lambda(\rho_\infty)$ is a lattice with area a given divisor of $\mathrm{Vol}(\rho)$. Note that if $n = 1$ then $\rho$ may be seen as a representation in $\Hom(\Gamma, \C)$ and we have $\Lambda(\rho) = \rho(\Gamma)$.

\begin{lemma}\label{LatticeClosure}
Let us suppose that $g \geqslant 3$. Let $\chi\in \Hom(\Gamma, \C)$ be such that $\mathrm{Vol}(\chi) > 0$ and $\chi(\Gamma)$ is not a lattice in $\C$. For every $m\geqslant 1$ there exists $\chi_\infty\in \overline{\Aut(\Gamma)\cdot \chi}$ such that $\Lambda = \chi_\infty(\Gamma)$ is a lattice, and $\mathrm{Vol}(\chi_\infty) / \mathrm{Area}(\Lambda) = m$.
\end{lemma}

\begin{proof}
It follows from \cite[Proposition 2.10]{CalsamigliaDeroin} that, replacing $\chi$ with $A\cdot \chi$, $A\in \mathrm{GL}_2^+(\R)$ if necessary, the closure of $\Aut(\Gamma)\cdot \chi$ contains the set of representations $\chi'\in \Hom(\Gamma, \C) \simeq \C^{2g}$ such that $\mathrm{Vol}(\chi') = \mathrm{Vol}(\chi)$ and $\Im(\chi')\in \Z^{2g}$ and $\gcd(\Im(\chi')) = 1$. In particular, $\chi' = (\frac{\mathrm{Vol}(\chi)}{m}, im, i, 0, \ldots, 0)$ is in $\overline{\Aut(\Gamma)\cdot \chi}$.
\end{proof}

%Let us suppose that $\alpha$ has finite image of order $n\geqslant 1$, and that $\rho$ as  the form of \cref{OOOD}.
Let us first classify the orbits of the $\Aut(\Gamma)$ action on the set $\Omega(L)$ of representations in $\Hom(\Gamma, \C)$ with image a given lattice $L$ in $\C$. If $\rho\in \Omega(L)$ then $\mathrm{Vol}( \rho)$ is a multiple of $\mathrm{Area}(L)$.
\begin{lemma}\label{imagelattice}
If $g\geqslant 2$, then we define a bijection $\Omega(L)\to \Z$ by $$[\rho']\mapsto \mathrm{Vol}(\rho')/\mathrm{Area}(L).$$
\end{lemma}
\begin{rmk}
In particular, for each $\rho'\in \Hom(\Gamma, \C)$ with image $L$, there exists $\sigma\in \Aut(\Gamma)$ such that $\sigma\cdot\rho' = (mu, v, u, 0, \ldots, 0)$, where $m = \mathrm{Vol}(\rho') / \mathrm{Area}(L)$.
\end{rmk}
We refer to \cite[Proposition 2.2]{Moi} for a proof of \cref{imagelattice}.

We suppose that $\rho = (\zeta z, z, z + x_1, z + y_1, \ldots, z + x_{g-1}, z + y_{g-1})$. Observe that $\Lambda$ is the group generated by the $\xi x_i$ and $\xi y_i$ for $1\leqslant i < g$ and $\xi$ a $n$-th root of unity.

Let us first suppose that $\Lambda$ is not a lattice in $\C$. We show that we may assume that the group generated by the $x_i$ and $y_i$, $1 \leqslant i < g$, is not a lattice if $g\geqslant 3$.
\begin{lemma}\label{notAlattice}
Suppose $g\geqslant 3$.
Replacing $\rho$ with some $\sigma\cdot\rho$ if necessary, we can assume that the group generated by the $x_i$ and $y_i$, $1\leqslant i < g$, is not a lattice.
\end{lemma}
\begin{proof}
Let us suppose that it is a lattice, and let $u, v$ be generators of this lattice. There exists $\xi$ an $n$-th root of unity such that $\xi u$ or $\xi v$ is not in $\mathbb Q u \oplus \mathbb Q v$. 
We can assume that $x_2 = u$, $y_2 = v$, $x_3 = mu$, $y_3 = v$ and $x_j = y_j = 0$ for every $j\geqslant 4$, where $m\in \Z$ by \cref{imagelattice} and \cref{decoupe}. Again by \cref{decoupe} we may conjugate $\rho(\gamma)$ by $\xi$, $\gamma\in \{a_1, b_1, a_2, b_2\}$, since $\xi$ is a power of $\rho(a_1)$. We thus have replaced $x_2$ and $y_2$ with $\xi x_2$ and $\xi y_2$. These new $x_i$ and $y_i$ do not generate a lattice.
\end{proof}

We now suppose that $\Lambda$ is a lattice. Recall that $\zeta = e^{2\pi i / n}$. We can assume that the map $z\mapsto \zeta z$ is in the image of $\rho$ by \cref{OOOD}; conjugating $\rho$ with a translation if necessary. Thus if $a\in \Lambda$, we also have $\zeta a\in \Lambda$ because one can conjugate $z+a$ with $\zeta z$. We also have $\zeta^2 a\in \Lambda$, thus $\zeta^2\in \mathbb Q \oplus \zeta \mathbb Q$. 
Therefore $[\mathbb Q(\zeta) : \mathbb Q]\leqslant 2$. Since $\varphi(n) = [\mathbb Q(\zeta) : \mathbb Q]$, we have $\varphi(n) \leqslant 2$, where $\varphi$ denotes the Euler's totient function. This implies $n\in \{1, 2, 3, 4, 6\}$.
%\begin{rmk}
%\end{rmk}
In fact we may even characterize the lattices that arise in this fashion for $n\in \{3, 4, 6\}$. Let $\omega = \exp(2\pi i /3)$.

\begin{lemma}
There exists $v\in \C^\star$ such that 
$$v\Lambda = \left\{
\begin{array}{ll}
	\Z[i]& \mbox{if } n= 4\\
	\Z[\omega] & \mbox{if } n \in \{3,6\}.
\end{array}
\right.$$
\end{lemma}
\begin{proof}
Choose a vector of minimal length $v$ in $\Lambda\setminus \{0\}$, and replace $\Lambda$ with $\frac 1 v \Lambda$. Since $\Lambda$ is stable by multiplication by $e^{2\pi i / n}$, $\mathbb Z[e^{2\pi i / n}]$ is a lattice included in $\Lambda$. Let us suppose that $n=4$. Let $v\in \Lambda$. There exists $w\in \Z[i]$ such that $v - w = a + ib$ satisfies $-\frac 1 2 \leqslant a,b\leqslant \frac{1}{2}$. Therefore $|v-w| < 1$ and $w = v$. Thus $\Lambda = \Z[i]$. Similarly if $n\in \{3, 6\}$, let $v\in \Lambda$. There exists $w\in \Z[\omega]$ such that $v - w = a + \omega b$, where $-\frac 1 2\leqslant a,b \leqslant \frac{1}{2}$. We have $|v-w| \leqslant |a| + |b| \leqslant 1$. Moreover one of these inequalities must be strict, hence $w = v$ thus $\Lambda = \Z[\omega]$.
%Let $u\in \Lambda$. Translating $u$ with a vector of $\mathbb Z[e^{2\pi i / n}]$, we may assume that $u$ is in a fundamental domain $D$ pictured in [fig]. Since $|u|\geqslant 1$, and $\{z\in \C\mid |z|\geqslant 1\}\cap D \subset \mathbb Z[e^{2\pi i / n}]$, we have $u\in\mathbb Z[e^{2\pi i / n}]$.
\end{proof}

We can thus assume that $\Lambda = \mathbb Z[i]$ if $n = 4$ and $\Lambda = \mathbb Z[\omega]$ if $n\in \{3, 6\}$. We now show that we may also assume that the group $\Lambda$ is generated by the $x_i$ and $y_i$ for $1\leqslant i < g$.

\begin{lemma}
We may assume that the group generated by $x_i, y_i$ is $\Lambda$, replacing $\rho$ with some $\rho'\in \Aut(\Gamma)\times \mathrm{Aff}(\C)\cdot \rho$.
\end{lemma}

\begin{proof}
Let us suppose that the group $L$ generated by the $x_i,y_i$ is different from $\Lambda$. Observe we cannot have $n \in \{1, 2\}$, thus $n\in \{3, 4, 6\}$.
The subgroup $L$ of $\Lambda$ is a lattice in $\C$: the vectors $x_i,y_i$ cannot be colinear since $\mathrm{Vol}(\rho) >0$. Let $u, v$ be generators of $L$. Let us show that we can assume that the angle between $u$ and $v$ is in $[\pi / 3, 2\pi/3]$. We can normalize the lattice $L$ so that $v = 1$. The action of $\SL\Z$ on the generators of $L$ is then the action of $\SL \Z$ on $\mathbb H^2$. It is well known that the domain $D = \{x\in \mathbb H^2\mid -1/2\leqslant \Re (x)\leqslant 1/2\}\cap \{x\in \mathbb H^2\mid |x|\geqslant 1\}$ is a fundamental domain for this action, see for example \cite[Chapter 7]{Serre}. Therefore we now assume that the angle $\theta$ between $u$ and $v$ is in $[\pi / 3, 2\pi / 3]$. We may also assume that $u = e^{i \theta}v$ if $\theta\in \{\pi / 3, 2\pi / 3\}$. In this case $L = v\Z[\omega]$  and $\Lambda = L$ if $n\in \{3, 6\}$, since $L$ is invariant under multiplication by $e^{i\pi/3}$.
By \cref{imagelattice} we can suppose that $$\rho = (\zeta z, z,  z + u, z + v,z + u, z + mv, z,\ldots, z)$$
where $m\in \mathbb Z$.
%On fait d'abord en sorte que $x_i, y_i = (nu, v, u, 0, \ldots 0)$ avec $u, v$ des vecteurs tq l'angle entre les deux est $[\pi / 3, 2\pi /3]$ avec égalité sur l'une des extrémités ssi le réseau est $j$ invariant.
By \cref{decoupe} we may conjugate $\rho(\gamma)$ with $\rho(a_1)$ for $\gamma\in \{a_1, b_1, a_2, b_2\}$, and thus replace $\rho$ with
%We apply the Dehn twist along a curve freely homotopic to $[a_2, b_2]a_1$ as in [fig]. This replaces $\rho$ with 
$$\rho = (\zeta z, z,  z + \zeta u, z + \zeta v,z + u, z + mv, z,\ldots, z).$$
The angle between $u$ and $\zeta v$ is $\theta + \frac{2\pi}{n}$ therefore they generate a lattice unless $(n, \theta)\in \{(4, \frac \pi 2), (3, \frac \pi 3), (6, \frac {2\pi } 3)\}$. Indeed $0 < |\sin(\theta + \frac{2\pi} n)| < |\sin \theta|$ for others $(n, \theta)$, see \cref{graphe}. Since we are done in the case where $(n, \theta)\notin \{(3, \frac \pi 3), (6, \frac{2 \pi} 3)\}$, let us suppose that $(n, \theta) = (4, \frac \pi 2)$. Observe that if $|u| = |v|$, then $L$ is invariant by multiplication by $i$, thus $L = \Lambda$. We can thus assume that $|u| < |v|$. 
There exist $\sigma\in \Aut(\Gamma)$ and $m\in \Z$ such that 
$$\sigma\cdot \rho = (i z, z, z +  u, z, z + u, z+mv, z, \ldots, z).$$
By \cref{decoupe}, we may conjugate $\rho(\gamma)$ by $\rho(a_1)$ for $\gamma\in \{a_1, b_1, a_2, b_2\}$ and thus replace $\rho$ as before with
$$\rho = (i z, z, z + iu, z, z + u, z + mv, z\ldots, z).$$
The new lattice $L'$ satisfies $\langle u, iu\rangle \subset L'$. Since $|\det(u, iu)| = |u|^2$ and $|u||v| = |\det(u, v)|$, we have $\mathrm{Area}(L') \leqslant |\det(u, iu)| < |\det(u, v)| \leqslant \mathrm{Area}(L)$.
Therefore in each case we may replace the lattice $L$ with another lattice $L'$ such that either $L' = \Lambda$ or $[\Lambda : L'] < [\Lambda : L]$. 
%Therefore we must end up with $L = \Lambda$ after a finite number of iterations of this process.
By iterating this process, we end up after a finite number of steps with $L = \Lambda$.
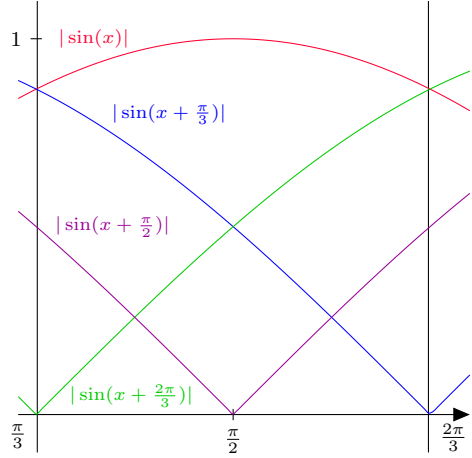
\begin{figure}[h]
    \centering
    \def\svgwidth{\columnwidth}
%  	\def\svgwidth{1\textwidth}
%	\hspace*{-1cm}
\definecolor{zzqqzz}{rgb}{0.6,0,0.6}
\definecolor{qqccqq}{rgb}{0,0.8,0}
\definecolor{qqqqff}{rgb}{0,0,1}
\definecolor{ffqqtt}{rgb}{1,0,0.2}
\begin{tikzpicture}[line cap=round,line join=round,>=triangle 45,x=5cm,y=5cm]
\draw[->,color=black] (1,0) -- (2.2,0);
%\foreach \x in {1,1.2,1.4,1.6,1.8,2,2.2}
%\draw[shift={(\x,0)},color=black] (0pt,2pt) -- (0pt,-2pt) node[below] {\footnotesize $\x$};
\draw[shift={(2*3.1415926535/3,0)},color=black]
 node[below right] {\footnotesize $\frac {2\pi} 3$};
\draw[shift={(3.1415926535/3,0)},color=black]
 node[below left] {\footnotesize $\frac {\pi} 3$};
\draw[shift={(3.1415926535/3,1)},color=black](2pt,0pt) -- (-2pt,0pt)
 node[left] {\footnotesize $1$};
\draw[shift={(3.1415926535/2,0)},color=black](0pt,2pt) -- (0pt,-2pt)
 node[below] {\footnotesize $\frac \pi 2$};

\clip(1,-0.1) rectangle (2.2,1.1);
\draw (1.05,-0.1) -- (1.05,1.1);
\draw (2.09,-0.1) -- (2.09,1.1);
\draw[color=ffqqtt, smooth,samples=100,domain=1.0:2.2] plot(\x,{abs(sin(((\x))*180/pi))});
\draw[color=qqqqff, smooth,samples=100,domain=1.0:2.2] plot(\x,{abs(sin(((\x)+3.1415926535/3)*180/pi))});
\draw[color=qqccqq, smooth,samples=100,domain=1.0:2.2] plot(\x,{abs(sin(((\x)+2*3.1415926535/3)*180/pi))});
\draw[color=zzqqzz, smooth,samples=100,domain=1.0:2.2] plot(\x,{abs(sin(((\x)+3.1415926535/2)*180/pi))});
\begin{scriptsize}
\draw[color=ffqqtt] (1.2,1) node {$|\sin(x)|$};
\draw[color=qqqqff] (1.4,0.8) node{$|\sin(x + \frac \pi 3)|$};
\draw[color=qqccqq] (1.3,0.05) node{$|\sin(x + \frac {2\pi} 3)|$};
\draw[color=zzqqzz] (1.25,0.5) node{$|\sin(x + \frac {\pi} 2)|$};
\end{scriptsize}
\end{tikzpicture}
\caption{Sinusoidal functions}
    \label{graphe}
\end{figure}
\end{proof}
%
%\subsubsection{Trivial linear part}
%We suppose here that $\alpha$ is the trivial homomorphism, that is $n=1$. Therefore $\rho\in \Hom(\Gamma, \C)$, where $\C$ is subgroup of $\mathrm{Aff}(\C)$ of translations.
%
%
%
%\subsubsection{Non trivial linear part}
%We suppose here that $n\geqslant 2$.
%
%\begin{lemma}\label{ecrasetourne}
%There exists $\rho_\infty\in \overline{\Aut(\Gamma)\times \mathrm{Aff}(\C)\cdot \rho}$ of the form 
%$$\rho_\infty = (\zeta z, z, z + 1, z, z, z, \ldots, z, z)$$ 
%where $\zeta = \exp(2i\pi/n)$.
%\end{lemma}
%
%\begin{proof}
%We start with $\rho$ as in \cref{OOOD}.
% Since $\rho$ is not conjugated in $\SO$, there exists $3\leqslant i \leqslant 2g$ such that $z_i\neq 0$. We then proceed as in \cref{ecrasement}.
%\end{proof}
%
%\subsubsection{Lattice?}
%Let $\rho\in \Hom(\Gamma, \mathrm{Isom}^+(\mathbb E^2))$ be such that $\alpha = \mathrm{Li}\circ \rho$ has finite image of order $n\geqslant 1$. Let $\Lambda$ be the subgroup of $\C$ of $\mathrm{Tr}(\gamma)$ for $\gamma\in \ker \alpha$.
\subsubsection{Euclidean representation with infinite linear part}
In \cite{Ghazouani}, Ghazouani studied the action of the mapping class group on $\Hom(\Gamma, \mathrm{Aff}(\C)) / \mathrm{Aff}(\C)$ and in particular he gave a precise description of the closure $H_\alpha\cdot \rho$ in $\Hom(\Gamma, \mathrm{Aff}(\C))/\mathrm{Aff}(\C)$, where $\rho\in \Hom(\Gamma, \mathrm{Aff}(\C))$ and $\alpha = \mathrm{Li}\circ \rho$ and $H_\alpha = \{\sigma\in \Aut(\Gamma) \mid \sigma\cdot \alpha = \alpha\}$. In particular he showed the following,  see \cite[Lemma 5.3, Proposition 6.4]{Ghazouani}.
\begin{prop}[Ghazouani]\label{SelimG}
If $\rho\in \Hom(\Gamma, \mathrm{Isom}^+(\mathbb E^2))$ is such that
 $\mathrm{Vol}(\rho) > 0$ and $\mathrm{Li}\circ \rho$ has infinite image, then the closure of $H_\alpha\times \mathrm{Aff}(\C) \cdot \rho$ is the set of $\rho'\in \Hom(\Gamma, \mathrm{Aff}(\C))$ such that $\mathrm{Vol}(\rho') > 0 $ and $\mathrm{Li}\circ \rho' = \mathrm{Li}\circ \rho$.
\end{prop}
It follows that under the same hypotheses, the closure of $H_\alpha\times \mathrm{Isom}^+(\mathbb E^2)\cdot \rho$ is the set of $\rho'\in \Hom(\Gamma, \mathrm{Isom}^+(\mathbb E^2))$ such that $\mathrm{Vol}(\rho) = \mathrm{Vol}(\rho')$.
Since we have seen that $\Aut(\Gamma)\cdot \alpha$ is dense in $\Hom(\Gamma, \mathbb S^1)$ when $\alpha\in \Hom(\Gamma, \mathbb S^1)$ has infinite image, we get the following.
\begin{prop}\label{orbiteEucl}
Let $\rho\in \Hom(\Gamma, \mathrm{Isom}^+(\mathbb E^2))$ be such that $\mathrm{Vol}(\rho) > 0$ and $\mathrm{Li}\circ \rho \in \Hom(\Gamma, \mathbb S^1)$ has infinite image.
The closure of $\Aut(\Gamma)\times \mathrm{Isom}^+(\mathbb E^2)\cdot \rho$ is the set of $\rho'\in \Hom(\Gamma, \mathrm{Isom}^+(\mathbb E^2))$  such that $\mathrm{Vol}(\rho') = \mathrm{Vol}(\rho)$.
\end{prop}

\begin{proof}
Let $\rho'\in \Hom(\Gamma, \mathrm{Isom}^+(\mathbb E^2))$ be such that $\mathrm{Vol}(\rho') = \mathrm{Vol}(\rho)$.
Let us first show that $\rho'\in \overline{\Aut(\Gamma)\times \mathrm{Aff}(\C)\cdot \rho}$. Let $U$ be a neighborhood of $\rho'$ in $\Hom(\Gamma, \mathrm{Aff}(\C))$. There exist a neighborhood $U_1$ of $\mathrm{Li}\circ\rho' $ and a neighborhood $U_2$ of $\mathrm{Tr}\circ \rho'$ in the set of maps $\Gamma\to \C$ endowed with the pointwise convergence topology so that $\varphi^{-1}(U_1\times U_2)\subset U$, where $$\varphi : p\in \Hom(\Gamma, \mathrm{Aff}(\C))\mapsto (\mathrm{Li}\circ p, \mathrm{Tr}\circ p).$$

There exists $\sigma\in \Aut(\Gamma)$ such that $\mathrm{Li}\circ \rho''\in U_1$, where $\rho'' = \sigma\cdot \rho$ by \cref{ptitlemme}.
By \cref{SelimG}, there exists $f\in \mathrm{Aff}(\C)$ and $\sigma'\in H_{\mathrm{Li}\circ \rho''}$ such that $(\sigma', f)\cdot \rho''\in U$. Therefore $\rho'\in \overline{\Aut(\Gamma)\times \mathrm{Aff}(\C)\cdot \rho}$.
Let us now show that $\rho'\in \overline{\Aut(\Gamma)\times \mathrm{Isom}^+(\mathbb E^2)\cdot \rho}$.
There exists a sequence $(f_n)_n\in \Hom(\Gamma, \mathrm{Aff}(\C))^{\mathbb N}$ and $(\sigma_n)_n\in \Aut(\Gamma)^{\mathbb N}$ such that $(\sigma_n, f_n)\cdot\rho\to \rho'$. Let us write $f_n = \lambda_n z + \mu_n$, $\lambda_n = \mathrm{Li}(f_n)$ and $\mu_n = \mathrm{Tr}(f_n)$. The sequence $|\lambda_n|$ tends to $1$ since $\mathrm{Vol}((\sigma_n, f_n)\cdot \rho) = |\lambda_n|^2\mathrm{Vol}(\rho)$. The sequence defined by  $h_n = \lambda_n/|\lambda_n| z + \mu_n/|\lambda_n|$ is in $\Hom(\Gamma, \mathrm{Isom}^+(\mathbb E^2))$ and is such that $(\sigma_n, h_n)\cdot \rho\to \rho'$.
\end{proof}

\section{Branched projective structures}\label{BPS}
\subsection{Definitions}
% Le but de cette section est de rappeler les définitions précises de BPS, puis d'expliquer la chirurgie d'éclatement des points branchés, et le bubbling généralisé.
%Let us denote by $\mathrm C_n$ the topological space $[0, 1]\times [0, 2\pi n]/\sim$, where $\sim$ is the equivalence relation generated by $(0, \theta) = (0, \theta')$ and by $(r, 0)$

\begin{defn}
A branched projective atlas on $\Sigma$ is the datum of an open covering $(U_\alpha)_\alpha$ of $\Sigma$ together with maps, called \emph{charts}, $\varphi_\alpha : U_\alpha\to V_\alpha$, where $V_\alpha$ is an open set of $\CP$. We require that the $\varphi_\alpha$ are topologically conjugated to the map from the unit disk to itself $z\mapsto z^n$, for some $n\geqslant 1$. We also require that when $U_\alpha\cap U_\beta\neq \emptyset$, there exists $h\in \PSL \C$ such that $\varphi_\alpha = h\circ \varphi_\beta$ on $U_\alpha\cap U_\beta$.
\end{defn}

Two branched atlases are said to be equivalent if their union is still a branched atlas.
Let us denote by $\mathbf{BP}(\Sigma)$ the set of equivalence classes of branched projective atlases on $\Sigma$, or equivalently the set of maximal atlases.
Given a maximal atlas we may associate to each point $p\in \Sigma$ an integer $n\geqslant 0$ as follows. We can pick a chart $\varphi_\alpha$ around $p$ that is topologically conjuated to $z\mapsto z^{n+1}$ and such that $p$ correspond to $0$. This integer $n\geqslant 0$ does not depend on the chart and vanishes everywhere but on a discrete, hence finite, set. We thus define a divisor on $\Sigma$ called the \textit{branch divisor}. A point where this integer $n$ is at least $1$ is called a branch point, or a conical point. We say that $n$ is the order of this branch point, or that this conical point has angle $2\pi(n+1)$.
The group $\mathrm{Homeo}_+(\Sigma)$ acts on $\mathbf{BP}(\Sigma)$ by precomposing the charts.
\begin{defn}
A branched projective structure on $\Sigma$ is an element of $\mathcal{BP}(\Sigma) = \mathbf{BP}(\Sigma)/\mathrm{Homeo}_0(\Sigma)$.
\end{defn}
The branch divisor of a branched projective structure is not well-defined but the datum of the orders of the branch points is. Hence $\mathcal {BP}$ decomposes as the union of $\mathcal P(n_1, \ldots, n_k)$ of branched projective structures with $k$ branch points of order $n_1, \ldots, n_k$.

Given $X\in \mathbf{BP}(\Sigma)$, we may choose a chart $U_\alpha$ and consider its analytic continuation $f : \widetilde \Sigma\to \CP$ as usual with geometric structures, see \cite[Chapter 3]{Thurston}. There exists a unique $\rho\in \Hom(\Gamma, \PSL \C)$ such that $f$ is $\rho$-equivariant: $f(\gamma\cdot z) = \rho(\gamma)\cdot f(z)$ for all $z\in \widetilde \Sigma$, and $\gamma\in \Gamma$. Choosing another chart would change $f$ to $h\circ f$ and $\rho$ to $h\rho h^{-1}$ for some $h\in \PSL \C$. The map $f$ is called a developing map of $X$ and $\rho$ the holonomy of $X$.
Note that one can give an element of $\mathbf{BP}(\Sigma)$ by giving its developing map.
Observe that the holonomy $\Hol(X) = [\rho]\in \Hom(\Gamma, \PSL \C)/\PSL \C$ satisfies $\Hol(h\cdot X) = \Hol(X)$ for every $h\in \mathrm{Homeo}_0(\Sigma)$. Indeed if $h$ is isotopic to the identity, then it admits a $\Gamma$-equivariant lift $H : \tilde \Sigma\to \tilde \Sigma$. A developing map for $g\cdot X$ is thus $\mathrm{dev}(X)\circ H$, that is $\rho$-equivariant.
Therefore we have a well-defined map \[\Hol : \mathcal {BP}\to \Hom(\Gamma, \PSL \C) / \PSL \C.\]
In the rest of the paper, we will abuse notation and identify $\Hol(\mathcal P(n_1, \ldots, n_k))$ with the subset of $\Hom(\Gamma, \PSL \C)$ of representations $\rho$ whose conjugacy class is in $\Hol(\mathcal P(n_1, \ldots, n_k))$.

Let us give some examples of branched projective surfaces.
\begin{ex}
\begin{enumerate}
\item A (possibly branched) hyperbolic structure on $\Sigma$ gives a developing map $\tilde \Sigma\to \mathbb H^2$ that is $\rho$-equivariant for some $\rho\in \Hom(\Gamma, \PSL \R)$.
\item Suppose $S$ is a surface endowed with a branched projective structure $X$. Let us consider $f : \Sigma\to S$ a branched cover. We define the \textit{pullback} of $X$ by $f$, that we denote $f^* X$, as follows. A developing map for $f^* X$ is given by $\mathrm{dev}\circ \tilde f$, where $\mathrm{dev} : \tilde S\to \CP$ is a developing map for $X$ and $\tilde f : \tilde \Sigma\to \tilde S$ is a lift of $f$.
\end{enumerate}
\end{ex}

\subsection{The mapping class group action on $\mathcal{BP}$}

The mapping class group $\Mod(\Sigma) = \mathrm{Homeo}_+(\Sigma) / \mathrm{Homeo}_0(\Sigma)$ acts naturally on $\mathcal{BP}$. This action preserves the sets $\mathcal P(n_1, \ldots, n_k)$. The map $\Hol$ is equivariant with respect to the actions of $\Mod(\Sigma)$ on $\mathcal{BP}$ and on $\Hom(\Gamma, \PSL \C)/ \PSL \C$.
%
%Let $G$ be a subgroup of $\PSL \C$.
%The group $\Aut(\Gamma)$ acts on $\Hom(\Gamma, G)$ by precomposition and $G$ acts on $\Hom(\Gamma, G)$ by conjugation. These actions commute and thus yield an action of $\Aut(\Gamma)\times G$ on $\Hom(\Gamma, G)$. This action preserves the sets $\Hol(\mathcal P(n_1, \ldots, n_k))\cap \Hom(\Gamma, G)$.
\begin{prop}
Let $\rho\in \Hom(\Gamma, \PSL \C)$ and $\sigma\in \Aut(\Gamma)$ and $h\in \PSL \C$. For every $1\leqslant n_1\leqslant \ldots \leqslant n_k$, we have: $$\rho\in \Hol (\mathcal P(n_1, \ldots, n_k)) \iff h(\sigma\cdot \rho) h^{-1}\in \Hol(\mathcal P(n_1, \ldots, n_k)).$$
\end{prop}

\subsection{The Ehresmann-Thurston principle}
We prove a version of the Ehresman-Thurston principle, \cref{ET}, adapting an argument of Kapovich \cite[Theorem 2.7]{Kapovich}.
We first show that given a projective structure on $\Sigma$, one can find a triangulation of $\Sigma$ whose edges are sent to broken circular arcs and whose triangles develop injectively. Let us fix $P\subset \tilde \Sigma$ a fundamental polygon for the action of $\Gamma$ on $\tilde \Sigma$. 
The edges of $P$ are paired and identified in $\Sigma$ with $\gamma_1, \ldots, \gamma_m\in\Gamma$.
Let us consider the compact set \[K = P\cup \bigcup_{i=1}^m \gamma_i\cdot P.\]
\begin{lemma}\label{triangulation}
Let us consider a branched projective structure $X$ on $\Sigma$.
There exists a triangulation $T$ of $\Sigma$ such that \begin{enumerate}
\item Every conical point of $X$ is a vertex of $T$
\item Every edge of $T$ is a broken circular arc: it can be decomposed into segments that are sent to circular arcs by the developing map. We also require that two contiguous circular arcs are not tangent.
\item Each lift $\tilde T$ of a triangle $T$ that lies in $K$ is send injectively by the developing map onto a set of area less than $\frac{1}{3} \mathrm{Area}(\CP)$, where $\mathrm{Area}$ is the area for the round metric on $\CP$.
\end{enumerate}
\end{lemma}
\begin{proof}
Let us start with an arbitrary triangulation of $\Sigma$. We refine this triangulation so that each triangle is included in a chart and such that the image of each lift of these triangles in $K$ has area less than $\frac{1}{3}\mathrm{Area}(\CP)$. We can again refine the triangulation so that each branch point is a vertex of this triangulation and such that each triangle develops injectively. We then approximate every edge with a broken circular line. We thus change our triangulation so that it satisfies the hypotheses of \cref{triangulation}.
\end{proof}

We now prove \cref{ET}.

\begin{proof}
Let $X\in \mathcal P(n_1, \ldots, n_k)$ and let $\rho \in \Hom(\Gamma, \PSL \C)$ be its holonomy. We fix a sequence $(\rho_n)_n$ in $\Hom(\Gamma, \PSL \C)$ such that $\rho_n\to \rho$.
Let us choose a triangulation of $X$ as in \cref{triangulation}.
Let us consider the graph $G$ on $\Sigma$ obtained from the triangulation by subdivising its edges as follows. 
Let us color the vertices of this triangulation in black.
We add green vertices on the points that are sent to endpoints of the circular arcs.
We then add a red vertex in the interior of each edge of this new graph.
Let us lift $G$ to $\tilde G$ on the universal cover $\tilde \Sigma$ of $\Sigma$. We fix $p_1, \ldots, p_l$ a system of representatives of the vertices of $\tilde G$ modulo the action of $\Gamma$. For each vertex $p$ in $\tilde G$, we can write in a unique way $p = \gamma p_i$ with $\gamma\in \Gamma$ and $1\leqslant i\leqslant l$.
We let $f_n(p) = \rho_n(\gamma)\cdot f(p_i)$.
Let us consider two contiguous edges with vertices $e_1, e_2$ and $e_3$, with $e_2$ the only red one, which is in the middle of $e_1$ and $e_3$. We define $f_n$ on this edge so that it sends it to the circular arc of $\CP$ that passes through $f_n(e_2)$ and that is bounded by $f_n(e_1)$ and $f_n(e_3)$. Let us fix $T_1, \ldots, T_l\subset K$ a system of representatives of the triangles. If $n$ is large enough, then the image of the edges of each $T_i$ bound two triangles in $\CP$. We extend $f_n$ on $T_i$, by sending $T_i$ homeomorphically onto the one with the smallest area.
We then extend $f_n$ by $\rho$-equivariance.
Let us now show that $f_n$ is a branched covering map when $n$ is large enough, with singularities of the type $z\mapsto z^{n_1+1}, \ldots, z\mapsto z^{n_k+1}$.
We must show that $f_n$ is locally injective outside some of the black vertices, where it is a $k$-fold covering map if $f$ is a $k$-fold covering map, for $n$ large enough. By $\rho$-equivariance, it suffices to check it in the interior of $K$. Let $z$ be in the interior of $K$. If $z$ not on the graph $\tilde G$, then $f_n$ is locally injective around $z$. Let us suppose that $z$ a black vertex of the graph $\tilde G$. The angles formed by the circular arcs around $f_n(z)$ converge to the angles formed by the circular arcs around $f(z)$, and their sum is in $2\pi \mathbb Z$. Hence $f_n$ is a $k$-fold cover around $z$ for $n$ large enough, if $f$ is itself a $k$-fold cover around $z$. If $z$ is a point on the interior of an edge of the triangulation, or a vertex that is not black, then $f$ is locally injective around $z$. Indeed, one just need to check that the two triangles around $z$ are not folded on top of each other by $f_n$. Since this is the case with $f$, and that the image of the triangles by $f_n$ are converging to the image of the triangles by $f$, it does not happen when $n$ is large enough.
\end{proof}
For another point of view and proof of \cref{ET}, we refer to \cite[Theorem 1.3, Remark 2.4]{Loray}. 
Let us give a useful consequence of \cref{ET}.
\begin{cor}\label{corET}
Let $\rho\in \Hom(\Gamma, \PSL \C)$. If there exists $\rho_\infty\in \Hol(\mathcal P(n_1, \ldots, n_k))$ satisfying $\rho_\infty\in \overline{\Aut(\Gamma)\times \PSL \C \cdot \rho}$, then $\rho\in \Hol(\mathcal P(n_1, \ldots, n_k))$.
\end{cor}

\subsection{Surgeries on branched projective structures}
We now describe some ways to construct new projective structures from old ones.

\subsubsection{Cut and paste}
Let us consider a non necessarily connected surface $S$ with a branched projective structure.
Let $\gamma,\delta : [0,1]\to S$ be injective arcs on $S$ whose image intersect at most in one point, in which case this point is $\gamma(1) = \delta(0)$. Let us suppose that there exists charts $\varphi_\gamma : U_\gamma\to \CP$ and $\varphi_\delta : U_\delta\to \CP$ such that the image of $\gamma$ (resp. $\delta$) is contained in $U_\gamma$ (resp. $U_\delta$), and such that $\varphi_\gamma\circ \gamma(t) = \varphi_\delta \circ \delta(t)$ for all $0\leqslant t \leqslant 1$.
We cut open $S$ along the image of $\gamma$ and along the image of $\delta$. This yields four boundary components $\gamma^{\pm}$ and $\delta^{\pm}$. We glue back the boudary $\gamma^{\pm}$ with $\delta^{\mp}$ by identifying $\gamma(t)$ with $\delta(t)$.
See also \cite[Section 2.2]{CalsamigliaDeroin} and in particular \cite[Figure 1]{CalsamigliaDeroin}.
This surgery adds a handle connecting the two arcs and increases the conical angles at the endpoints of the arcs. 

%\color{red}
%[Question pour Maxime : est-ce que je fais un dessin ? Ce serait en gros le même que celui de Calsamiglia-Deroin]
%\color{black}
\subsubsection{Bubbling along a closed curve}
There exists a surgery on branched projective structures called \textit{bubbling}, see for example \cite[Section 2.3]{CalsamigliaDeroin}. It amounts to considering an arc $\gamma$ on $\Sigma$ that is contained in a chart and develops injectively and using the cut and paste surgery describe above on the surface $\Sigma\sqcup \CP$ with the arcs $\gamma$ and $\delta$, where $\delta$ is the image of $\gamma$ in a chart.  We now describe a slightly modified version of this surgery, that increases the branching order of a single point: bubbling along a closed curve.

\begin{prop}\label{bubbling}
Let us consider a projective structure $X\in \mathcal P(n_1, \ldots, n_k)$ with developing map $f$. Let $c \colon [0,1]\to \Sigma$ be a closed curve such that $c(0)$ is a branch point of order $n_1\geqslant 0$,  and let $\tilde c \colon [0,1]\to \tilde \Sigma$ be a lift of $c$. Suppose that $f\circ \tilde c\colon [0,1]\to \CP$ is injective. There exists a projective structure $Y\in \mathcal P(n_1+2n, \ldots, n_k)$ with the same holonomy as $X$, for every $n\geqslant 1$. 
\end{prop}
In the statment of \cref{bubbling}, $n_1 = 0$ means that $c(0)$ is not a branch point.

\begin{proof}
We may assume that $c$ does not meet any branch point except $c(0)$.
Let us cut open $\CP$ along the image of $f\circ \tilde c$, and denote by $\delta^{\pm}$ the resulting boundary components. Let us cut open $\Sigma$ along the image $c$, and denote by $\gamma^{\pm}$ the resulting boundary components. We glue back $\gamma^{\pm}$ to $\delta^{\mp}$ by identifying $c(t)$ with $f(\tilde c(t))$ for each $t\in [0,1]$. This new projective structure has a conical point of order $n_1 + 2$ if $c(0)$ is a branch point of order $n_1$, and has the same holonomy as $X$. Since $c$ still develops injectively, we can repeat this process.
\end{proof}

\subsubsection{Breaking up a conical point}
Eskin, Masur and Zorich described in \cite[8.1]{EskinMasurZorich} a surgery on translation surfaces that breaks a conical point into two conical points with smaller angles. Let us describe this surgery, that we adapt to branched projective structures. It breaks a branch point of order $n$ into two branch points of order $a,b$ such that $a + b = n$. This surgery is local: it leaves the charts of the inital projective structures unchanged outside of a disk embedded in $\Sigma$. Therefore it does not modify its holonomy.
% This surgery is not new and is described in details in the realm of translation surfaces in \cite[Section 4.2]{KontsevichZorich}.
\begin{prop}\label{explosion}
Let $X\in \mathcal P(n_1, \ldots, n_k)$ and let $a,b\geqslant 1$ be such that $a + b = n_1$. There exists $Y\in \mathcal P(a, b, n_2, \ldots, n_k)$ with the same holonomy as $X$.
\end{prop}
\begin{proof}
Locally around a branch point of order $n_1$, the projective structure is made of $2(n_1+1)$ half-disks glued together as in \cref{macron1} for $n_1 = 2$. We modify the gluing patern of their boundaries as in \cref{macron2} for $a=b=1$. We refer to \cite[Section 8.1]{EskinMasurZorich} for more details on this construction. See also \cite[Section 4.2]{KontsevichZorich}.
\end{proof}

\begin{figure}[h]
    \centering    
    \def\svgwidth{\columnwidth}
%  	\def\svgwidth{1\textwidth}
%	\hspace*{-1cm}
	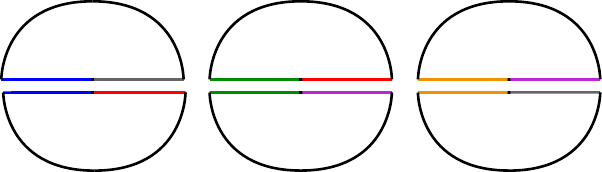
	\caption{Neighborhood of a branch point of order $2$}
    \label{macron1}
\end{figure}

\begin{figure}[h]
    \centering    
    \def\svgwidth{\columnwidth}
%  	\def\svgwidth{1\textwidth}
%	\hspace*{-1cm}
	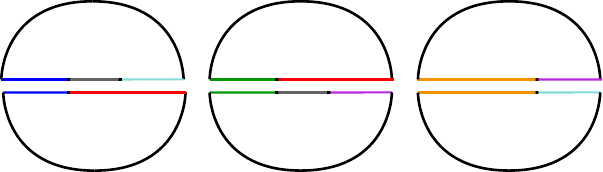
	\caption{Two branch points of order $1$}
    \label{macron2}
\end{figure}

\begin{rmk}\label{corexplo}
Let $\rho\in \Hol(\mathcal P(d))$. It follows from \cref{explosion} that $\rho\in \Hol(\mathcal P(n_1, \ldots, n_k))$ if $\sum_i n_i = d$.
\end{rmk}

\section{Geometrization}\label{SGEO}
The aim of this section is to prove \cref{mainth}. We will consider representations $\rho\in \Hom(\Gamma, \PSL \C)$ that satisfy the conditions of the obstructions and prove that they are \textit{geometric}, that is, that they are the holonomy of a projective structure with the given datum of branch points. Along the way we will also explain the obstructions that we did not explain in the introduction.
\subsection{Branched cover of the sphere}
Let us fix $g\geqslant 0$. We consider the special case where $\rho\in \Hom(\Gamma, \PSL \C)$ is the trivial representation: $\rho(\gamma) = \pm\begin{pmatrix}
1 & 0\\
0 & 1
\end{pmatrix}$ for every $\gamma\in \Gamma$. The aim of this part is to prove the following.

\begin{prop}\label{trivial}
The trivial representation $\rho$ is in $\mathcal \Hol(\mathcal P(n_1, \ldots, n_k))$ if and only if  $\sum_i n_i = 2g+2l$ for some integer $l \geqslant 1$ such that $\max_i n_i \leqslant l$.
\end{prop}

The developing map $f : \tilde \Sigma\to \CP$ of a spherical structure with holonomy $\rho$ induces a branched cover $\Sigma\to \mathbb S^2$. Conversely, any branched cover $\Sigma\to \mathbb S^2$ gives a branched projective structure on $\Sigma$ with trivial holonomy. Therefore we are reduced to the classical problem of the existence of branched covers of the sphere with prescribed branching data. With this point of view it is clear that \cref{trivial} is equivalent to \cref{branchedcover}.
%which was proved with very different methods for $g=0$ in  \cite[Theorem C]{Tomasini} and \cite{KapovichSphere}. 
As observed in \cite{MondelloPanov}, it is the $g=0$ case a necessary and sufficient condition to the existence of a spherical metric with conical angles $2\pi(n_1+1), \ldots, 2\pi(n_k+1)$ since its holonomy must be trivial.

Let us recall the classical statement of the Hurwitz realization problem, following the notations of \cite{EdmondsKulkarni}.
A \textit{branch data} is a set of partitions of an integer $d>0$. A branched cover $\Sigma\to \mathbb S^2$ of degree $d$ induces a branch data in the following way: for each point $p\in \mathbb S^2$ which is the image of a branch point, the orders $a_i$ of its pre-images give a partition of $d$. A branch data that arises in this way is said to be \textit{realizable}. The Hurwitz realization problem is to characterize the branch data that are realizable. 
The Riemann-Hurwitz formula gives obstructions to realizability. For a partition  $A = [a_1, \ldots, a_l]$, let $v(A) = \sum_i (a_i - 1)$, and for a branch data $\mathcal D = \{A_1, \ldots, A_k\}$, let $v(\mathcal D) = \sum_i v(A_i)$. If $\mathcal D$ is realizable, then the Riemann-Hurwitz formula gives \[\chi(\Sigma) = 2d - v(\mathcal D).\]
In particular, $v(\mathcal D)$ must be even and at least $2d-2$. However it is well known that these conditions are not sufficient, see for example \cite{EdmondsKulkarni, Pervova}. This problem was considered by Hurwitz in \cite{Hurwitz} where he showed the following theorem. We identify the conjugacy classes of the symmetric group $\Sy d$ with the partitions of $\{1, \ldots, d\}$ as follows: the lengths of the orbits of the action of the group generated by $\sigma\in \Sy d$ on $\{1, \ldots, d\}$ determine a partition of $d$.

\begin{hurwitz}
The branch data $\mathcal D = \{A_1, \ldots, A_k\}$ is realizable if and only if there exists permutations $\alpha_1\in A_1, \ldots \alpha_k\in A_k$ such that
\begin{enumerate}
\item $\alpha_1\alpha_2\ldots\alpha_k = \id$
\item The group generated by the $\alpha_i$ acts transitively on $\{1, \ldots, d\}$.
\end{enumerate}
\end{hurwitz}
We show the following statement, that implies \cref{trivial}.

\begin{prop}\label{realiz}
Let $d \geqslant 2$ and $1\leqslant n_1\leqslant\ldots \leqslant n_k < d$. The branch data $\mathcal D = \{A_1, \ldots, A_k\}$ where $A_i = [n_i+1, 1,\ldots, 1]$ is realizable if and only $v(\mathcal D) = \sum_i n_i$ is even and $v(\mathcal D) \geqslant 2d-2$.
\end{prop}
We follow the reduction methods of \cite{EdmondsKulkarni}, and prove this proposition by induction. Let us first check this proposition for $k=3$.

\begin{lemma}\label{prodcycle}
%Let $1\leqslant n \leqslant m$. There exists a $(n+1)$-cycle $\alpha$ and a $(m+1)$-cycle $\beta$ such that $\alpha\beta$ is a $(p+1)$-cycle for every $0\leqslant p< d$ such that $m-n \leqslant p \leqslant m+n$ and $p \equiv m+n [2]$. Moreover, if $m + n \geqslant d-1$, we may assume that the group generated by $\alpha$ and $\beta$ acts transitively on $\{1, \ldots, d\}$.
\cref{realiz} holds for $k = 3$.
\end{lemma}

\begin{proof}
We proceed by induction on $d$. If $d = 2$ then $n_1=n_2=n_3=1$ and there is no realizable branch data. 
Let us suppose that $v(\mathcal D) = n_1 + n_2 + n_3$ is even and that $v(\mathcal D) \geqslant 2d-2$.
If $n_1 = n_2 = n_3 = d-1$, then $d$ is odd and we can take $\alpha_1 = \alpha_2$, with $\alpha_1 = (1, \ldots, d)$. Then $\alpha_3 = \alpha_1^{-2}$ is a $d$-cycle, and these $\alpha_i$ satisfy the Hurwitz conditions.
Let us now suppose that  $\min\{n_1, n_2, n_3\} < d-1$. We may assume that $n_3 < d-1$. Let $n'_1 = n_1-1$, $n'_2 = n_2 - 1 $ and $n'_3 = n_3$ and $\mathcal D' = \{[n'_1+1, 1\ldots, 1], [n'_2 + 1, 1\ldots, 1], [n'_3 + 1, 1\ldots 1]\}$ . We have $v(\mathcal D') = v(\mathcal D) - 2$. Hence $v(\mathcal D')$ is even and at least $2d' - 1$, where $d' = d-1$. Therefore by induction the datum $\mathcal D'$ is realizable. Hence there exist $\alpha'_1$, $\alpha'_2$ and $\alpha'_3$ that are cycles in $\Sy {d-1}$ of length $n'_1+1$, $n'_2 + 1$ and $n'_3+1$ and that generate a group that acts transitively on $\{1, \ldots, d-1\}$ and such that  $\alpha'_1\alpha'_2\alpha'_3=\id$. There exists $1\leqslant i \leqslant d-1$ that is in the support of both $\alpha'_1$ and $\alpha'_2$, because $\alpha'_1\alpha'_2$ is a cycle. We may suppose, conjugating the $\alpha'_i$ if necessary, that $i = d-1$ is in their support. We let $\alpha_1 = \alpha'_1 (d-1, d)$, $\alpha_2 = (d-1, d)\alpha'_2$ and $\alpha_3 = \alpha'_3$.
The permutations $\alpha_1$, $\alpha_2$ and $\alpha_3$ act transitively on $\{1, \ldots, d\}$ and $\alpha_i \in A_i$ for $1\leqslant i \leqslant 3$. Moreover $\alpha_1\alpha_2\alpha_3 = \id$.
%If $p=n+m$, then we take $\alpha = (1,\ldots, n+1)$ and $\beta = (n+1, m+2, \ldots, n+m+1)$. We now assume that $p < n+m$, and thus by parity $p \leqslant m + n - 2$.
%If $m=n=p = d-1$, then since $p\equiv m+n[2]$, we must have $d\equiv 1 [2]$. We may in this case take $\alpha = (1, \ldots, d)$ and $\beta = \alpha$.
%We thus now assume that $\min\{m, n, p\} < d-1$. We are looking for cycles such that $\alpha\beta = \gamma$. This equality may be written as $ \beta\gamma^{-1} = \alpha$, hence we may assume that $p < d-1$, interchanging the roles of $\alpha$, $\beta$ and $\gamma$ if necessary.
%
%Let us search for $\alpha$ and $\beta$ of the form $\alpha = (x_1, x_2, \ldots, x_{n-1}, d-1, d)$ and $\beta = (d, d-1, y_1, \ldots y_{m-1})$. We want $\alpha\beta = (x_1, \ldots, x_{n-1}, d-1)(d-1, y_1, \ldots, y_{m-1})$ to be a $(p+1)$-cycle. By induction, this is possible for some $1\leqslant x_i, y_i < d-1$.
\end{proof}

We now prove \cref{realiz}.

\begin{proof}
We proceed by induction on $k$.
It is obvious for $k=2$ since in this case $n_1 = n_2 = d-1$. It thus suffices to consider $d$-cycles $\alpha_1$ and  $\alpha_2= \alpha_1^{-1}$ to satisfy the Hurwitz conditions. 
%We turn to the case $k=3$. Observe that $n_2 - n_1\leqslant  n_2\leqslant n_3\leqslant n_2 + n_1$. Indeed if $n_3 > n_1 + n_2$, then since $n_1 + n_2 + n_3$ is even, $n_3\geqslant n_1 + n_2 + 2$. Moreover $n_1 + n_2 + n_3 \geqslant 2d-2$, thus $n_3\geqslant d$ which is a contradiction. By \cref{prodcycle}, we may find $\alpha_1\in A_1, \alpha_2\in A_2$ and $\alpha_3\in A_3$ such that $\alpha_1\alpha_2\alpha_3 = \id$. Moreover, $n_1 + n_2 \geqslant d-1$, and it follows from the proof of \cref{prodcycle} that we can choose $\alpha_1$ and $\alpha_2$ to act transitively on $\{1, \ldots, d\}$. 
We now assume that $k > 3$. Let us first suppose that there exist $1\leqslant i < k$ such that $n_i + n_{i+1} \leqslant d-1$. There exist a $(n_i+1)$-cycle $\alpha$ and a $(n_{i+1}+1)$-cycle $\beta$ such that $\alpha\beta$ is a $(n_i+n_{i+1}+1)$-cycle. Indeed we can give explicit examples, or use \cref{prodcycle} in $\Sy {n_i + n_{i+1} + 1}$. Let $A$ be the partition associated to $\alpha\beta$. 
We may assume that $i = 1$.
The branch data $\mathcal D' = \{A, A_3, \ldots, A_k\}$ is realizable by induction, since $v(\mathcal D') = v(\mathcal D)$. Therefore there exist $\sigma\in A$, $\alpha_3\in A_3, \ldots, \alpha_k\in A_4$ that act transitively on $\{1, \ldots, d\}$ such that $\sigma \alpha_3\ldots\alpha_k = \id$. There exists $\gamma\in \Sy d$ such that $\sigma = \gamma \alpha\beta\gamma^{-1}$. Let $\alpha_1 = \gamma \alpha \gamma^{-1}$ and $\alpha_2 = \gamma \beta \gamma^{-1}$. The permutations $\alpha_1,\ldots, \alpha_k$ satisfy the Hurwitz conditions.

We now suppose that for every $1\leqslant i < k $, we have $n_i + n_{i+1} \geqslant d$. By \cref{prodcycle}, there exist $\alpha\in A_1$ and $\beta\in A_2$  such that $\alpha\beta$ is a $d$-cycle or a $(d-1)$-cycle. 
This can also be found in \cite[Corollary 4.4, Lemma 4.5]{EdmondsKulkarni}. As before let $A$ be the partition associated to $\alpha\beta$. The branch data $\mathcal D' = \{A, A_3, \ldots, A_k\}$ satisfies $v(\mathcal D') \geqslant d-2 + n_3 + n_4 \geqslant  2d-2$. Therefore by induction there exist $\sigma\in A$, $\alpha_3\in A_3, \ldots, \alpha_k\in A_4$ that act transitively on $\{1, \ldots, d\}$ and such that $\sigma \alpha_3\ldots\alpha_k = \id$. There exists $\gamma\in \Sy d$ such that $\sigma = \gamma \alpha\beta\gamma^{-1}$. Let $\alpha_1 = \gamma \alpha \gamma^{-1}$ and $\alpha_2 = \gamma \beta \gamma^{-1}$. The permutations $\alpha_1, \ldots, \alpha_k$ satisfy the Hurwitz conditions.
\end{proof}

\subsection{Spherical structures}
In this section we consider representations that fix a point in $\mathbb H^3$. We can conjugate $\rho$ so that $\rho\in \Hom(\Gamma, \SO)$.
\subsubsection{General obstructions}
Suppose we have a branched projective structure $X$ on $\Sigma$ with holonomy $\rho\in \Hom(\Gamma, \SO)$. Let us denote by $f : \tilde \Sigma\to \CP$ its developing map. The pullback $f^*\omega$ of the volume form $\omega$ associated with the round metric is $\Gamma$-invariant. Indeed for every $\gamma\in \Gamma$, $$\gamma^* f^*\omega = (f\circ \gamma)^* \omega = (\rho(\gamma)\circ f)^*\omega = f^* \rho(\gamma)^* \omega = f^* \omega.$$
 Therefore it induces a form $\omega_\Sigma$ on $\Sigma$. The Gauss-Bonnet formula gives, as observed in \cite[Section 1.2]{MondelloPanov}:
\[\sum_{i=1}^k n_i - \frac 1 {2\pi} \int_\Sigma \omega = 2g-2.\]
Therefore, we must have $\sum_i n_i > 2g-2$. This explains \cref{obstrMin} in the spherical case.

Let us suppose that $\rho$ takes its values in a finite subgroup of $\SO$ of order $n\geqslant 1$. We can consider the cover $S$ of $\Sigma$ associated with the subgroup $\ker \rho$ of $\Gamma$. The developing map $f$ induces a branched cover $f : S\to \CP$ of degree $d$. The Riemann-Hurwitz formula gives
\[n\chi(\Sigma) = 2d - n\sum_{i=1}^k n_i.\]
The degree $d$ must satisfy $d\geqslant \max_i n_i + 1$. This explains \cref{obstrRH}.

\subsubsection{Cyclic holonomy}\label{sectcycl}
We focus here on representations $\rho\in \Hom(\Gamma, \SL \C)$ with finite image included in $\mathrm{SO}_2(\R)$. 
Let us observe that any representation $\rho\in \Hom(\Gamma, \mathrm{SO}_2(\R))$ lifts to $\SL \C$. Indeed the $2$-sheeted cover of $\mathrm{SO}_2(\R) = \{\pm \begin{pmatrix}
\alpha & 0\\
0 & \alpha^{-1}
\end{pmatrix} \mid \alpha\overline \alpha = 1\}$ is the subgroup $\widetilde{\mathrm{SO}_2(\R)} = \{\begin{pmatrix}
\alpha & 0\\
0 & \alpha^{-1}
\end{pmatrix} \mid \alpha\overline \alpha = 1\}$ of $\PSL \C$ that is abelian, hence every commutator in $\widetilde{\mathrm{SO}_2(\R)}$ is trivial. The finite subgroups of $\mathrm{SO}_2(\R)$ are cyclic and we fix $\Z_n$ such a cyclic group of order $n\geqslant 2$.
Let $\rho$ be a surjective representation in $\Hom(\Gamma, \Z_n)$.
\begin{prop}\label{cyclic}
Let $1\leqslant n_1\leqslant \ldots \leqslant n_k$ be such that $\sum_i n_i = 2g + 2l$ for some $l\geqslant 0$. We have $\rho\in \mathcal \Hol(\mathcal P(n_1, \ldots, n_k))$ if and only if $n_k < n(l+1)$.
\end{prop}

Let us denote by $P_n$ a pyramid with basis a regular $n$-gon. We can identify a branched projective structure with holonomy in $\Hom(\Gamma, \Z_n)$ with a branched structure modeled on $P_n$ with holonomy in the cyclic of order $n$ isometry group of $P_n$, see \cref{pyramid}. 
\begin{figure}[h]
    \centering    
    \def\svgwidth{\columnwidth}
%	\hspace*{2cm}
  	\def\svgwidth{0.3\textwidth}

	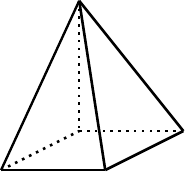
	\caption{The pyramid model $P_4$.}
    \label{pyramid}
\end{figure}

One benefit of this point of view is that a surface equipped with such a structure inherits a conical Euclidean structure as studied for example in \cite{Troyanov}: it is obtained by gluing Euclidean triangles along their boundaries. Let us recall the Gauss-Bonnet formula in this setting.
\begin{lemma}[Troyanov]\label{GBplat}
Let us consider a conical Euclidean structure on the surface $\Sigma_{g_S}$ of genus $g_S$. Let us denote by $2\pi(\theta_1+1), \ldots, 2\pi(\theta_s+1)$ the total angles of its conical points. We have 
$$\sum_{i=1}^s \theta_i = 2g_S-2.$$
\end{lemma}
We refer to \cite[Section 3]{Troyanov} for a short proof of \cref{GBplat}.
%
%\begin{proof}
%The Euler characteristic formula gives $2g_S-2 = e-v-f$ where $e$ is the number of edges of the triangulation of $S$, $v$ its number of vertices and $f$ its number of faces.
%We have $2e = 3f$, hence $2\pi(2g_S-2) = \pi f - 2\pi v = \sum_{i=1}^v (\alpha_i - 2\pi) = 2\pi \sum_i \theta_i$, where $\alpha_i$ is the total angle at $i$-th vertex.
%\end{proof}
We now prove \cref{cyclic}.
It suffices to exhibit a branched projective structure with given branch data and surjective holonomy in $\Hom(\Gamma, \Z_n)$ by \cref{ModCycl}.

\begin{proof}
We are going to explain how to construct a genus $g$ surface with holonomy a surjective homomorphism of $\Hom(\Gamma, \Z_n)$ and branch data given by $n_1, \ldots, n_k$.
Let us start by considering $l+1$ copies of $\CP$. 
Let us choose an annulus neighborhood of the equator in $\CP$. We identify this neighborhood with a rectangle $[0, n]\times [0,1]$ with its vertical sides identified. A horizontal translation of length $1$ in this model corresponds to an order $n$ rotation in $\CP$. 
We are going to cut open the corresponding annuli in these spheres along segments and glue the boundaries back in a different pattern.
Let us begin with the case $k = 1$ and $n$ even. We cut open an annulus along a length one horizontal segment and then along another, that is obtained by applying a length $1$ horizontal translation to the first one. We then make $g+l-1$ other slits in the annuli, along segments that are translates of the first one by an integer, see \cref{fig1}. We make sure that there is at least one slit on each annulus. We also make sure that these segments do not intersect. Observe that we can put $\frac{n}{2}$ of these line segments in each annulus. We have room to cut along these $g+ l+1$ segments since we have $n_1 = 2g + 2l$, thus $\frac n 2(l+1)> \frac{n_1} 2\geqslant g+l$. We glue back the boundary components with translations as indicated in \cref{fig1}. 
\begin{figure}[h]
    \centering    
    \def\svgwidth{\columnwidth}
%	\hspace*{2cm}
	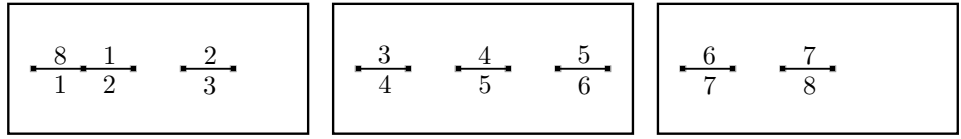
	\caption{Construction for $(n,k,l,g)=(6,1,2,5)$.}
    \label{fig1}
\end{figure}

The surface we obtain has genus $g$. Indeed let us denote by $g'$ this genus. Using the pyramidal model, we get a flat structure on this surface. The sum of the angles $\theta_i$ on a pyramid is $-2$ by \cref{GBplat}, thus again by \cref{GBplat} we have $2g'-2 = n_1 - 2(l+1) = 2g - 2$. Moreover the holonomy of the resulting surface is a surjective homomorphism of $\Hom(\Gamma, \Z_n)$.

Let us now turn to the case $k = 1$ and $n$ odd. Observe that we can now put $\frac {n+1} 2$ line segments in the first annulus, but only $\frac {n-1} 2$ ones in the others using the same techniques. If there is enough room to put the $g + l + 1$ line segments as before, then we proceed as in \cref{fig1}. If not, we may increase the total angle contained in two annuli by replacing the cuts of \cref{fig1} with those of \cref{figtrick}.

\begin{figure}[h]
    \centering    
    \def\svgwidth{\columnwidth}
%	\hspace*{2cm}
	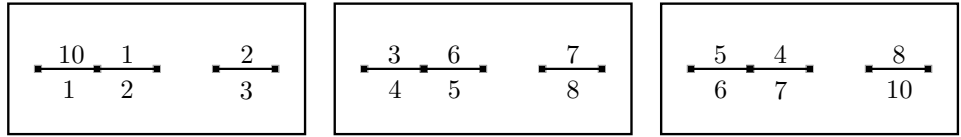
	\caption{Construction for $(n,k,l,g)=(5,1,2,5)$.}
    \label{figtrick}
\end{figure}
Each of the annuli in which this pattern is used contributes to the total angle by $2\pi n$. Since $n_1 < n(l+1)$, it is possible the get an angle of $2\pi(n_1 + 1)$. It follows from \cref{GBplat} that we get a genus $g$ surface. It has a projective structure with holonomy a surjective homomorphism of $\Hom(\Gamma, \Z_n)$.

Let us now turn to the case $k \geqslant 2$. We are again going to make slits in the annuli and glue the boundaries back using a pattern that we now describe.
We consider $n_1 +1$ vertical line segments that are obtained one from the other by horizontal translations. We will identify their boundaries in a cyclic way, see \cref{fig2}. Let us suppose that $n_1 < n_2$. We put $n_2 - n_1$ parallel line segments below, such that the top of exactly one of these segments touches one of the $n_1 +1$. We will identify these segments in a cylic way, as indicated in \cref{fig2}. Making slits along these segments and identifying the boundaries following this pattern will create three conical points: one of angle $2\pi(n_1 + 1)$, one of angle $2\pi(n_2 + 1)$ and one of angle $2\pi(n_2 - n_1 + 1)$ lines above. We continue similarly: we put $n_3 - (n_2 - n_1)$ line segments below, so that each touches the one above in a single point. We continue this construction until the gluing process gives the desired angles $n_1, \ldots, n_{k-1}$. We then change the last angle with horizontal slits as in \cref{fig2}.
If $n_1 = n_2$, we consider $n_1+1$ line segments that we identify in a cyclic way, as in the top of \cref{fig2}. We thus make the construction of the pattern for $n_3, \ldots, n_k$ below these segments. 
\begin{figure}[h]
    \centering    
    \def\svgwidth{\columnwidth}
%	\hspace*{2cm}
	\scalebox{0.6}{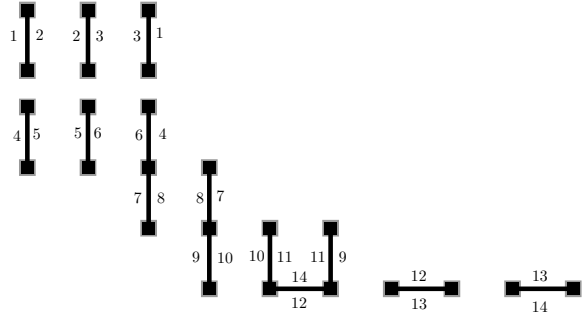}
	\caption{Gluing pattern for $k=6$ and $(n_i)_{1\leqslant i\leqslant 6} = (2, 2, 2, 3, 3, 6)$.}
    \label{fig2}
\end{figure}

This pattern has $g+l$ connected components, unless $k$ is even and $n_{2i} = n_{2i-1}$ for every $1\leqslant i \leqslant \frac k 2$, where it has $g+l+1$ connected components. Indeed one can count directly or argue as follows. If we place each connected component of this pattern on a different sphere, cut open along the segments and glue as indicated in the pattern, we get a genus $1$ surface with conical points $n_1, \ldots, n_k$. It follows from \cref{GBplat} that the number $C$ of connected components of this pattern satisfies $\sum_i n_i - 2 C = 0$. In the special case $n_{2i} = n_{2i-1}$ for each $i$, then we get a genus $0$ surface, thus $\sum_i n_i -2C = -2$.

We put the connected components of this pattern on the annuli as in \cref{cycliquegeom}. 
\begin{figure}[h]
    \centering    
    \def\svgwidth{\columnwidth}
%	\hspace*{2cm}
	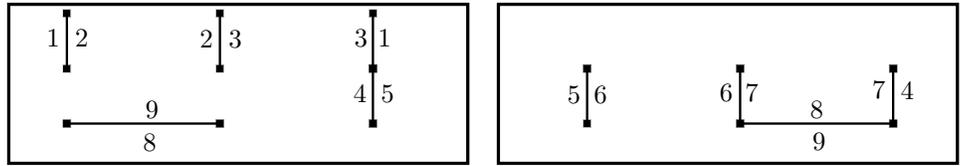
	\caption{Construction for $(n,k,l,g, n_1, n_2, n_3) = (3, 3, 1, 5, 2, 5, 5)$.}
    \label{cycliquegeom}
\end{figure}
We make sure that there is at least one line segment on each of the $l+1$ spheres. We have room to place the vertical segments in the annuli since $n_i < n(l+1)$. As before, there may not be enough room to put the horizontal segments when $n$ is odd. We can then use the pattern described in the case $k=1$ in \cref{figtrick} for the horizontal line segments.
The resulting surface is connected, has genus $g$ by \cref{GBplat}, and has holonomy a surjective homomorphism of $\Hom(\Gamma, \Z_n)$.
%We make a horizontal slit of length one in one of the spheres.
\end{proof}
\subsubsection{Dihedral holonomy}

We now consider representations $\rho\in \Hom(\Gamma, \D_n)$. We show that \cref{obstrSW}, \cref{obstrMin} and \cref{obstrRH} are the only obstructions of being in $\Hol(\mathcal P(n_1, \ldots, n_k))$.
Let us consider a bypiramid $B_n$ obtained by gluing two pyramids $P_n$ along their base, see \cref{bipyramid}. As in the cyclic holonomy case, a branched projective structure with holonomy in $\Hom(\Gamma, \D_n)$ may be identified with a structure modeled on $B_n$ with holonomy in its isometry group, which is isomorphic to $\D_n$. It thus inherits an Euclidean structure as before.
\begin{figure}[h]
    \centering    
    \def\svgwidth{\columnwidth}
%	\hspace*{2cm}
  	\def\svgwidth{0.2\textwidth}
	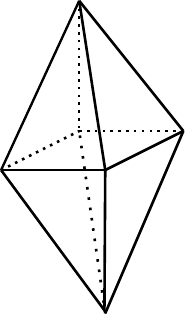
	\caption{The bipyramid model $B_4$.}
    \label{bipyramid}
\end{figure}

\begin{prop}\label{geodiedr}
Let $\rho\in \Hom(\Gamma, \mathrm{D}_n)$ be a surjective homomorphism.
The representation $\rho$ is in $\Hol(\mathcal P(n_1, \ldots n_k))$ if and only if
\begin{enumerate}
\item  $\sum_i n_i = 2g + 2l$, where $l\geqslant 0$ is such that $\max_i n_i  < 2n(l+1)$ if $\rho$ lifts to $\SL \C$.
\item $\sum_i n_i = 2g + 2l - 1$ where $l\geqslant 0$ is such that $\max_i n_i  < n(2l+1)$ if $\rho$ does not lift to $\SL \C$.
\end{enumerate}
\end{prop}

\paragraph{The genus $1$ case.} Let us suppose that $g=1$.
Recall that there is no surjective homomorphism in $\Hom(\Gamma_1, \D_n)$ if $n\geqslant 3$ since $\Gamma_1\simeq \Z^2$ is abelian.
Let us also recall that the action of $\Aut(\Gamma_1)$ on the set of surjective homomorphisms of $\Hom(\Gamma_1, \D_2)$ is transitive by \cref{DG1}.
The surjective homomorphisms of $\Hom(\Gamma_1, \D_2)$ do not lift to $\SL \C$. \cref{obstrRH} is trivially satisfied since $\sum_i n_i = 2l + 1$ thus $\max_i n_i < 2(2l+1)$. Therefore we may rephrase \cref{geodiedr} as follows.
\begin{lemma}
Let $\rho\in \Hom(\Gamma_1, \D_2)$ be a surjective homomorphism. We have $\rho\in \Hol(\mathcal P(n_1, \ldots, n_k))$ if and only if $\sum_i n_i $ is odd. 
\end{lemma}
\begin{proof}
Let $s=\rho(a_1)$ and $t=\rho(b_1)$.
Let us consider the round metric on $\CP \simeq \mathbb S^2$. Let $C$ be the great circle on $\mathbb S^2$ passing through the fixed points of both $s$ and $t$. Let us denote by $S$ and $N$ the fixed points of $t$. Let $d$ be the distance on $\CP$ associated with the round metric. We decompose $C$ into $4$ arcs: $P_N = \{x\in C\mid d(x, N) \leqslant 1\}$, $P_S = s(P_N)=\{x\in C\mid d(x, S) \leqslant 1\}$, and $P^\pm$ the (closure of the) two connected components of $C\setminus (P_S\cup P_N)$ in $C$. Let us cut $\mathbb S^2$ along $C$ to get a half-sphere with boundary $C$. We glue back its boundary in the following way. We glue $P_N$ to $P_S$ with $s$ and $P^-$ to $P^+$ with $t$, see \cref{JP}.
We thus get a torus with holonomy $\rho$ and a single conical point of total angle $2\pi(1 + 1)$.
There exists a curve based at the conical point, namely the image of $P_S$ in the torus, that develops injectively. Hence by \cref{bubbling} we have $\rho\in \Hol(\mathcal P(2d + 1))$ for every $d\geqslant 0$ and $\rho\in\Hol(\mathcal P(n_1, \ldots, n_k))$ for every $1\leqslant n_1\leqslant\ldots\leqslant n_k$ such that $\sum_i n_i$ is odd by \cref{corexplo}.
\end{proof}

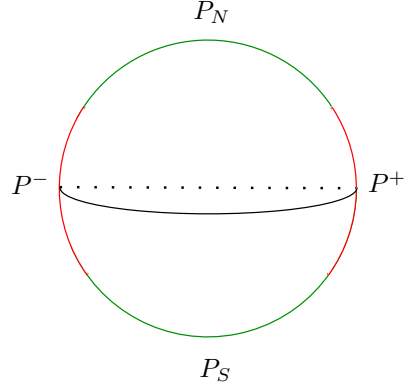
\begin{figure}
    \centering 
    \def\svgwidth{\columnwidth}
  	\def\svgwidth{0.5\textwidth}
	\hspace*{1cm}
	%% Creator: Inkscape 1.0.1 (0767f8302a, 2020-10-17), www.inkscape.org
%% Accompanies image file 'jp.pdf' (pdf, eps, ps)
\begingroup%
  \makeatletter%
  \providecommand\color[2][]{%
    \errmessage{(Inkscape) Color is used for the text in Inkscape, but the package 'color.sty' is not loaded}%
    \renewcommand\color[2][]{}%
  }%
  \providecommand\transparent[1]{%
    \errmessage{(Inkscape) Transparency is used (non-zero) for the text in Inkscape, but the package 'transparent.sty' is not loaded}%
    \renewcommand\transparent[1]{}%
  }%
  \providecommand\rotatebox[2]{#2}%
  \newcommand*\fsize{\dimexpr\f@size pt\relax}%
  \newcommand*\lineheight[1]{\fontsize{\fsize}{#1\fsize}\selectfont}%
  \ifx\svgwidth\undefined%
    \setlength{\unitlength}{54.29772834bp}%
    \ifx\svgscale\undefined%
      \relax%
    \else%
      \setlength{\unitlength}{\unitlength * \real{\svgscale}}%
    \fi%
  \else%
    \setlength{\unitlength}{\svgwidth}%
  \fi%
  \global\let\svgwidth\undefined%
  \global\let\svgscale\undefined%
  \makeatother%
  \begin{picture}(1,0.82517705)%
    \lineheight{1}%
    \setlength\tabcolsep{0pt}%
    \put(0,0){\includegraphics[width=\unitlength,page=1]{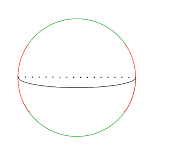}}%
    \put(0.37722573,0.76920675){\makebox(0,0)[lt]{\lineheight{1.25}\smash{\begin{tabular}[t]{l}$P_N$\end{tabular}}}}%
    \put(0.38957349,0.01737381){\makebox(0,0)[lt]{\lineheight{1.25}\smash{\begin{tabular}[t]{l}$P_S$\end{tabular}}}}%
    \put(-0.00611501,0.40179239){\makebox(0,0)[lt]{\lineheight{1.25}\smash{\begin{tabular}[t]{l}$P^-$\end{tabular}}}}%
    \put(0.74442866,0.40532405){\makebox(0,0)[lt]{\lineheight{1.25}\smash{\begin{tabular}[t]{l}$P^+$\end{tabular}}}}%
  \end{picture}%
\endgroup%

	\caption{Half-sphere.}
    \label{JP}
\end{figure}

\paragraph{The genus $g\geqslant 2$ case.} We now prove \cref{geodiedr} for $g\geqslant 2$.
%We denote as in \cref{diedrMCG} by $s\in \D_n\setminus \ker \epsilon$ a rotation along an axis of the Euclidean plane.
\begin{proof}
Let $n_1, \ldots, n_k$ satisfying the hypotheses of \cref{geodiedr}. By \cref{diedr}, it suffices to exhibit a projective structure with branch data $n_1, \ldots, n_k$ and holonomy a surjective homomorphism of $\Hom(\Gamma, \D_n)$.
Let us begin with the case where $\sum_i n_i$ is even. 
Let us consider $l+1$ copies of $\mathbb S^2 = \CP$. We consider an annular neighborhood of some longitudinal line in the northern hemisphere of $\mathbb S^2$, and its image by the flip $s$ that interchanges the north pole with the south pole. We therefore have $2(l+1)$ annuli, see \cref{noel}. 
\begin{figure}[h]

\begin{tikzpicture}
%equateur
\draw (2,0) arc(0:-180:2cm and 0.2cm) ;
\draw[dashed] (2,0) arc(0:180:2cm and 0.2cm) ;

%cercle le plus haut
\draw[color=blue] (1.4142, 1.4142) arc(0:-180:1.41cm and 0.1cm) ;
\draw[dashed, color=blue] (1.4142, 1.4142) arc(0:180:1.41cm and 0.1cm) ;
%Haut - Bas
\draw[color=blue] (1.7321, 1) arc(0:-180: 1.73cm and 0.15cm) ;
\draw[dashed, color=blue] (1.7321, 1) arc(0:180: 1.73cm and 0.15cm) ;

% Bas Bas 
\draw[color=blue] (1.4142, -1.4142) arc(0:-180:1.41cm and 0.19cm) ;
\draw[dashed, color=blue] (1.4142, -1.4142) arc(0:180:1.41cm and 0.19cm) ;

% Bas Haut
\draw[color=blue] (1.7321, -1) arc(0:-180: 1.73cm and 0.18cm) ;
\draw[dashed, color=blue] (1.7321, -1) arc(0:180: 1.73cm and 0.18cm) ;

\draw (0,0) circle (2) ;

\end{tikzpicture}

%    \centering    
%    \def\svgwidth{\columnwidth}
%  	\def\svgwidth{0.3\textwidth}
%
%%	\hspace*{2cm}
%	\input{noel.pdf_tex}
	\caption{Two annuli embedded in $\CP$.}
    \label{noel}
\end{figure}
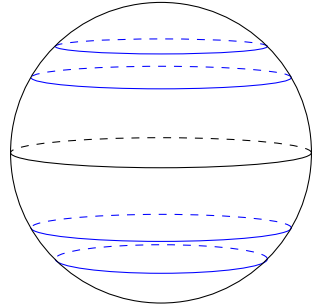
We make the same slits as in the proof of \cref{cyclic}. We make sure that there is a slit in each sphere, that there are two slits in an annulus glued with an order $n$ rotation. We also make sure that there are at least two slits, one in the northern annulus and one in the southern annulus of one of the spheres, that are glued together by $s$. This is possible since the pattern of the slits described in \cref{cyclic} has at least $g + l $ connected components and $g\geqslant 2$. Moreover recall that in the case where it has exactly $g+l$ connected components, there are two slits glued by a rotation of order $n$ in the pattern.
We get a projective structure on a surface of genus $g$ by the same argument as in \cref{cyclic}: it is a consequence of \cref{GBplat}. Its holonomy is a surjective homomorphism of $\Hom(\Gamma, \D_n)$ that lifts to $\SL \C$. Indeed this follows from the fact that $\sum_i n_i$ is even.

Let us now turn to the case where $n$ is even and $\sum_i n_i$ is odd: we have $\sum_i n_i = 2g + 2l - 1$.
We consider $l$ copies of $\mathbb S^2 = \CP$ and a half-sphere: we cut $\CP$ along the great circle $C$ that goes through both the fixed points of $s$ and the fixed points $N$ and $S$ of the order $2$ element of $\ker \epsilon$. 
Since $\sum_i n_i$ is odd, at least one of the $n_i$ is odd, say $n_{i_0}$.
Let us decompose the boundary of the half-sphere into $4$ circular arcs as before. One of these arcs is the set $P_N = \{x\in C \mid d(N, x) < 1\}$, where $d$ denotes the distance associated with the round metric on $\mathbb S^2$, and $P_S = s(P_N) = \{x \in C\mid d(S, x) < 1\}$. The other two parts $P^\pm$ are the (closure of the) connected components of $C\setminus (P_N\cup P_S)$, see the circle in green and red of the left bypyramid of \cref{noelbis}. We glue $P_N$ with $P_S$, these are the green arcs in \cref{noelbis}, with the map $s$.
We now make $\frac {n_{i_0}-1} 2$ slits along circular arcs obtained from $P^- $ by a rotation of order $n$ in the spheres and identify the boundaries components, see the red arcs in \cref{noelbis}. We have room to make these slits since we may put $\frac n 2 -1$ of them in the half-sphere, and $n$ in each of the $l$ spheres and $n_{i_0} < 2nl + n$.
We now consider $2l$ annular neighborhood of longitudinal circles in the spheres as before, away from these slits, see the blue annuli in the right bypiramid of \cref{noelbis}. We also consider the same annular neighborhood of length $n$ in the half-sphere: see the blue annulus on the left side of \cref{noelbis}.
\begin{figure}[h]
    \centering    
    \def\svgwidth{\columnwidth}
  	\def\svgwidth{0.5\textwidth}
%	\hspace*{-1cm}
	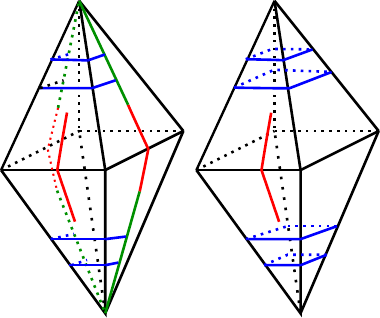
	\caption{Construction for $(n, k,l,g,n_1) = (4, 1, 1, 2, 5)$.}
    \label{noelbis}
\end{figure}
In these annuli, we make surgeries associated with the branch data $\{n_1, \ldots, n_k\}\setminus \{n_{i_0}\}$ as explained above in the case where $\rho$ lifts in $\SL \C$.
The gluing pattern we have described has at least $\frac 1 2 (\sum_i n_i -1) = g + l - 1$ connected components, and since $g\geqslant 2$ it is possible to put a slit on each sphere. Since $\max_i n_i < 2nl + n$, we have room to put these slits in the annuli. We thus get a connected surface, with charts in $\CP$, branch datum corresponding to $n_1, \ldots, n_k$. This surface has genus $g$ since by \cref{GBplat}, its genus $g'$ satisfies $\sum_i n_i = 2g' + 2l - 1$.
Its holonomy is a surjective homomorphism of $\Hom(\Gamma, \D_n)$ that does not lift so $\SL \C$ since $\sum_i n_i$ is odd.
\end{proof}

\subsubsection{Tetrahedron}

A branched projective structure with holonomy in $\Hom(\Gamma, \mathfrak A_4)$ may be identified as before with a branched structure modeled on the regular tetrahedron with holonomy in its isometry group $\mathfrak A_4$. We show that the first three obstructions are the only ones to being in $\Hol(\mathcal P(n_1, \ldots, n_k))$, where $1\leqslant n_1\leqslant\ldots \leqslant n_k$.

\begin{prop}\label{geotetra}
Let $\rho\in \Hom(\Gamma, \mathfrak A_4)$ be a surjective homomorphism. We have $\rho\in \Hol(\mathcal P(n_1, \ldots, n_k))$ if and only if 
\begin{enumerate}
\item $\sum_i n_i = 2g + 2l $, for $l\geqslant 0$ such that $\max_i n_i < 12(l+1)$ if $\rho$ lifts to $\SL \C$
\item $\sum_i n_i = 2g + 2l - 1$ for $l\geqslant 0$ such that $\max_i n_i < 12l + 6$ if $\rho$ does not lift to $\SL \C$.
\end{enumerate}
\end{prop}

\begin{proof}
%The necessary condition comes from the Riemann-Hurwitz formula and is \cref{obstrRH}.
Let us consider $n_1, \ldots, n_k$ satisfying the hypotheses of \cref{geotetra}. It suffices to exhibit a branched projective structure in $\mathcal P(n_1, \ldots, n_k)$ with holonomy a surjective homomorphism of $\Hom(\Gamma, \A 4)$ thanks to \cref{AutA4}.
Let us start with the case where $\sum_i n_i = 2g + 2l$ for some $l\geqslant 0$.
Let us consider $l+1$ copies of the regular tetrahedron $T$. We will make slits along line segments in the tetrahedra and glue the boundaries back with isometries of $T$. We use the pattern of \cref{fig2}. Let us describe how to put this pattern in the tetrahedra. Start by considering a line segment in one of the triangles of one of the tetrahedra as in \cref{T}. We consider $n_1$ line segments obtained from this one by applying an isometry of $T$. We identify the boundaries of the slits obtained by cutting open the tetrahedra along these segments with the corresponding isometries, following the pattern of \cref{fig2}. We continue to put the pattern of the line segments in this way.

We have room to make these slits since $\max_i n_i < 12(l+1)$.
We make sure to have one slit on each tetrahedron and that one of the tetrahedra has slits that are glued together with isometries generating $\mathfrak A_4$.
This construction gives a projective structure on a closed surface of genus $g$ with holonomy a surjective homomorphism of $\Hom(\Gamma, \mathfrak A_4)$. Indeed the genus $g'$ of the resulting surface can be computed using \cref{GBplat}: it satisfies $2g' - 2 = \sum_i n_i - 2(l+1) = 2g - 2$.

Let us now turn to the case $\sum_i n_i = 2g + 2l - 1$ for some $l\geqslant 0$. 
We consider $l$ copies of the regular tetrahedron and a half-tetrahedron: we cut a tetrahedron in half along a plane going through the midpoints of two pairs of opposite edges. We divide the resulting boundary into $2$ pairs as in the dihedral case: two red line segments and two green ones. We identify the green lines with an isometry of the tetrahedron.
One of the $n_i$, say $n_{i_0}$ must be odd.
Let us make $\frac{n_{i_0} - 1} 2$ slits on the tetrahedra and the half-tetrahedron, along line segments obtained by applying isometries to one of the red boundaries. We identify the resulting boundaries in a cyclic way.
We then make the slits and gluing associated with the branching data $\{n_1, \ldots, n_k\}\setminus \{n_{i_0}\}$.
We have room to make all these slits because of the assumption $\max_i n_i < 12 l + 6$.
We make sure to obtain a connected surface (each tetrahedron must have at least one slit) and that the resulting holonomy is a surjective homomorphism of $\Hom(\Gamma, \A 4)$. By \cref{GBplat}, the genus of the resulting surface is $g$ and its holonomy does not lift to $\SL \C$ since $\sum_i n_i$ is odd.
\end{proof}
\begin{figure}
    \centering    
    \def\svgwidth{\columnwidth}
  	\def\svgwidth{0.5\textwidth}

%	\hspace*{2cm}
	%% Creator: Inkscape 1.0.1 (0767f8302a, 2020-10-17), www.inkscape.org
%% Accompanies image file 't.pdf' (pdf, eps, ps)
\begingroup%
  \makeatletter%
  \providecommand\color[2][]{%
    \errmessage{(Inkscape) Color is used for the text in Inkscape, but the package 'color.sty' is not loaded}%
    \renewcommand\color[2][]{}%
  }%
  \providecommand\transparent[1]{%
    \errmessage{(Inkscape) Transparency is used (non-zero) for the text in Inkscape, but the package 'transparent.sty' is not loaded}%
    \renewcommand\transparent[1]{}%
  }%
  \providecommand\rotatebox[2]{#2}%
  \newcommand*\fsize{\dimexpr\f@size pt\relax}%
  \newcommand*\lineheight[1]{\fontsize{\fsize}{#1\fsize}\selectfont}%
  \ifx\svgwidth\undefined%
    \setlength{\unitlength}{157.83600926bp}%
    \ifx\svgscale\undefined%
      \relax%
    \else%
      \setlength{\unitlength}{\unitlength * \real{\svgscale}}%
    \fi%
  \else%
    \setlength{\unitlength}{\svgwidth}%
  \fi%
  \global\let\svgwidth\undefined%
  \global\let\svgscale\undefined%
  \makeatother%
  \begin{picture}(1,1.00002518)%
    \lineheight{1}%
    \setlength\tabcolsep{0pt}%
    \put(0,0){\includegraphics[width=\unitlength,page=1]{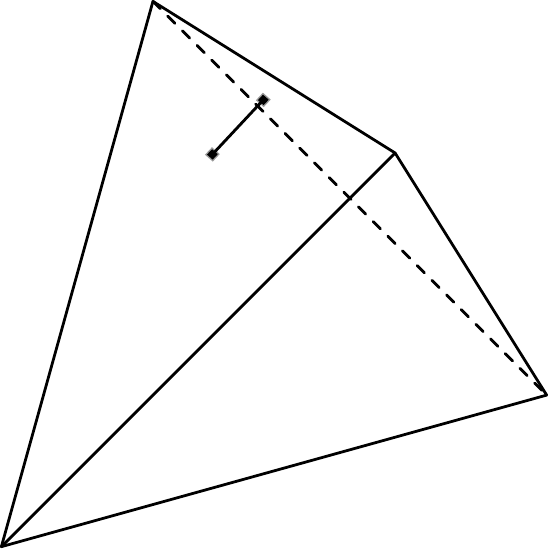}}%
    \put(0.41552109,0.76845726){\rotatebox{46.93468952}{\makebox(0,0)[lt]{\lineheight{1.25}\smash{\begin{tabular}[t]{l}$1$\end{tabular}}}}}%
    \put(0.46918514,0.72647251){\rotatebox{59.01833376}{\makebox(0,0)[lt]{\lineheight{1.25}\smash{\begin{tabular}[t]{l}$2$\end{tabular}}}}}%
    \put(0,0){\includegraphics[width=\unitlength,page=2]{t.pdf}}%
    \put(0.34885396,0.50210699){\rotatebox{95.51741967}{\makebox(0,0)[lt]{\lineheight{1.25}\smash{\begin{tabular}[t]{l}$2$\end{tabular}}}}}%
    \put(0.41289641,0.50643244){\rotatebox{95.51741961}{\makebox(0,0)[lt]{\lineheight{1.25}\smash{\begin{tabular}[t]{l}$3$\end{tabular}}}}}%
    \put(0,0){\includegraphics[width=\unitlength,page=3]{t.pdf}}%
    \put(0.68937753,0.27173008){\rotatebox{83.94866336}{\makebox(0,0)[lt]{\lineheight{1.25}\smash{\begin{tabular}[t]{l}$1$\end{tabular}}}}}%
    \put(0.62070588,0.27905831){\rotatebox{83.94866336}{\makebox(0,0)[lt]{\lineheight{1.25}\smash{\begin{tabular}[t]{l}$3$\end{tabular}}}}}%
  \end{picture}%
\endgroup%

	\caption{Construction for $(l,g,k,n_1, n_2) = (0, 2, 2, 2, 2).$}
    \label{T}
%\begin{figure}
%\begin{tikzpicture}[line join = round, line cap = round, thick]
%\pgfmathsetmacro{\factor}{1/sqrt(2)};
%\coordinate [label = right:A] (A) at (2, 2, 2);
%\coordinate [label = left:B]  (B) at (2, -2, -2);
%\coordinate [label = above:C] (C) at (-2, 2, -2);
%\coordinate [label = below:D] (D) at (-2, -2, 2);
%
%%\draw[->] (0,0) -- (3,0,0) node[right] {$x$};
%%\draw[->] (0,0) -- (0,3,0) node[above] {$y$};
%%\draw[->] (0,0) -- (0,0,3) node[below left] {$z$};
%%\foreach \i in {A, B, C, D}{
%%  \draw[dashed] (0, 0) -- (\i);
%%  \draw%[-, fill = red!30, opacity = .5] (A) -- (D) -- (B) --cycle;
%%  \draw[-, fill = green!30, opacity = .5] (A) -- (D) -- (C) --cycle;
%%  \draw[-, fill = purple!30, opacity = .5] (B) -- (D) -- (C) --cycle;
%\draw (A) -- (B);
%\draw (C) -- (D);
%\draw (A) -- (C);
%\draw (B) -- (D);
%\draw (A) -- (D);
%\draw[dashed] (C) -- (B);
%
%%  \draw[dashed] (A) -- (D) --cycle;
%%  \draw (A) -- (D) -- (C) --cycle;
%%  \draw (B) -- (D) -- (C) --cycle;
%
%%}
%\end{tikzpicture}
\end{figure}

\subsubsection{Cube} We can identify the branched projective strutures with holonomy in the group of the isometries of the cube with a branched structure modeled on the cube with holonomy in its isometry group.
Again we may geometrize a surjective homomorphism $\rho\in \Hom(\Gamma, \Sy 4)$ that satisfies the conditions of the first three obstructions.

\begin{prop}\label{geocube}
Let $\rho\in \Hom(\Gamma, \mathfrak S_4)$ be a surjective homomorphism. We have $\rho\in \Hol(\mathcal P(n_1, \ldots, n_k))$ if and only if 
\begin{enumerate}
\item $\sum_i n_i = 2g + 2l $, for $l\geqslant 0$ such that $\max_i n_i < 24(l+1)$ if $\rho$ lifts to $\SL \C$
\item $\sum_i n_i = 2g + 2l - 1$ for $l\geqslant 0$ such that $\max_i n_i < 24l + 12$ if $\rho$ does not lift to $\SL \C$.
\end{enumerate}
\end{prop}

\begin{proof}
The construction is analogous to the one of \cref{geotetra}, see \cref{cubcub}.
\end{proof}

\begin{figure}[h]
    \centering    
    \def\svgwidth{\columnwidth}
  	\def\svgwidth{0.4\textwidth}

%	\hspace*{2cm}
	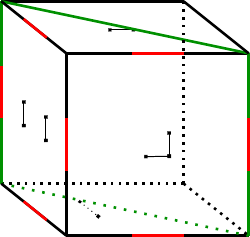
	\caption{Construction for $(l,g,k,n_1, n_2) = (0, 11, 2, 10, 11).$}
    \label{cubcub}
\end{figure}
\subsubsection{Icosahedron}

As before, a branched projective structure with holonomy in $\Hom(\Gamma, \mathfrak A_5)$ may be identified with a branched structure modeled on the icosahedron. Therefore we will geometrize the surjective representations of $\Hom(\Gamma, \mathfrak A_5)$ by gluing together regular icosahedra.

\begin{prop}
Let $\rho\in \Hom(\Gamma, \mathfrak A_5)$ be a surjective homomorphism. We have $\rho\in \Hol(\mathcal P(n_1, \ldots, n_k))$ if and only if 
\begin{enumerate}
\item $\sum_i n_i = 2g + 2l $, for $l\geqslant 0$ such that $\max_i n_i < 60(l+1)$ if $\rho$ lifts to $\SL \C$.
\item $\sum_i n_i = 2g + 2l - 1$ for $l\geqslant 0$ such that $\max_i n_i < 60l + 30$ if $\rho$ does not lift to $\SL \C$.
\end{enumerate}
\end{prop}

Again, the constructions are analogous to those of \cref{geotetra} and \cref{geocube}, see \cref{Icoco}.

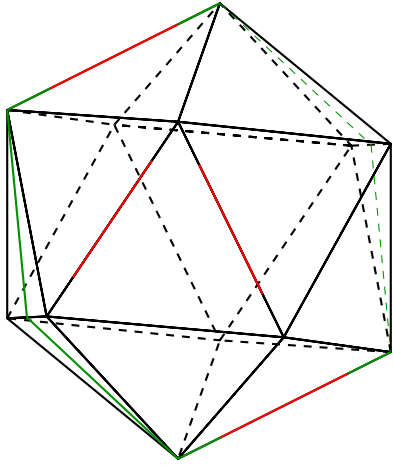
\begin{figure}[h]
\begin{center}
\definecolor{mycolor}{RGB}{0,144,0}
    \tdplotsetmaincoords{60}{100}
    \begin{tikzpicture}[tdplot_main_coords,scale=0.8,line join=round]
    \pgfmathsetmacro\a{2}
    \pgfmathsetmacro{\phi}{\a*(1+sqrt(5))/2}
    \path 
    coordinate(A) at (0,\phi,\a)
    coordinate(B) at (0,\phi,-\a)
    coordinate(C) at (0,-\phi,\a)
    coordinate(D) at (0,-\phi,-\a)
    coordinate(E) at (\a,0,\phi)
    coordinate(F) at (\a,0,-\phi)
    coordinate(G) at (-\a,0,\phi)
    coordinate(H) at (-\a,0,-\phi)
    coordinate(I) at (\phi,\a,0)
    coordinate(J) at (\phi,-\a,0)
    coordinate(K) at (-\phi,\a,0)
    coordinate(L) at (-\phi,-\a,0)
   	coordinate(M) at (-\a*0.8, -\phi*0.2, \phi*0.8 + \a*0.2)
   	coordinate(N) at (-\a*0.2, -\phi*0.8, \phi*0.2 + \a*0.8)
   	coordinate(O) at (\a*0.8, \phi*0.2, -\phi*0.8  -\a*0.2)
   	coordinate(P) at (\a*0.2, \phi*0.8,-\phi*0.2 -\a*0.8)
   	coordinate(Q) at (\phi*0.5, -\a/2-\phi*0.5, -\a*0.5)
   	coordinate(R) at (-\phi*0.5, \a/2+\phi*0.5, \a*0.5)
   	coordinate(S) at (\phi*0.8 + \a*0.2, \a*0.8, \phi*0.2)
   	coordinate(T) at (\phi*0.2 + \a*0.8, \a*0.2, \phi*0.8)
   	coordinate(U) at (\a*0.8 + \phi*0.2, -\a*0.2, \phi*0.8)
   	coordinate(V) at (\a*0.2 + \phi*0.8, -\a*0.8, \phi*0.2)

   	%(D + J)/2
; 
% 	coordinate(M) at (
    \draw[dashed, thick]    (B) -- (H) -- (F) 
    (D) -- (L) -- (H) --cycle 
    (K) -- (L)
    (H) --(K)
    (K) -- (L) -- (G) --cycle
    (C) -- (L) (B)--(K) (A)--(K)
    ;
    
        \draw[thick]
        (A) -- (I) -- (B) --cycle 
        (F) -- (I) -- (B) --cycle 
        (F) -- (I) -- (J) --cycle
        (F) -- (D) -- (J) --cycle
        (C) -- (D) -- (J) --cycle
        (C) -- (E) -- (J) --cycle
        (I) -- (E) -- (J) --cycle
        (I) -- (E) -- (A) --cycle
        (G) -- (E) -- (A) --cycle
        (G) -- (E) -- (C) --cycle
        ; 
        
\draw[thick, color=red] (M) -- (N) 
							  (O) -- (P);
\draw[thick, color=mycolor](N) -- (C)
							  (M) -- (G)
							  (O) -- (F)
							  (P) -- (B)
							  (Q) -- (F)
							  (Q) -- (C);
\draw[dashed, color=mycolor]						  
							  (G) -- (R)
							  (R) -- (B)
							  ;
\draw[thick, color=red]
								(S) -- (T);
\draw[thick, color=red]
								(U) -- (V);				  
%\draw[ultra thick, color=red] (O) -- (P) ;
% \foreach \point/\position in {A/right,B/below,C/above,D/left,E/{above right},F/below,G/above,H/left,I/below,J/right,K/below,L/left}
%{
%    \fill (\point) circle (1.5pt);
%    \node[\position=3pt] at (\point) {$\point$};
%}
\end{tikzpicture}
\caption{Construction for $(l,g,k,n_1) = (0,3, 1, 5)$}
\label{Icoco}
\end{center}
\end{figure}
\subsubsection{Holonomy with infinite image in $\SO$}
We characterize the representations $\Hom(\Gamma, \SO)$ with infinite image that are in $\Hol(\mathcal P(n_1, \ldots, n_k))$. The infinite closed subgroups of $\SO$ are isomorphic to $\mathrm{SO}_2(\R)$, $\mathrm O_2(\R)$ or are the whole group $\SO$.
Let us fix $1\leqslant n_1\leqslant\ldots\leqslant n_k$.

\paragraph{Holonomy with dense image in $\mathrm{SO}_2(\R)$} 
%Let $\rho\in \Hom(\Gamma, \mathrm{SO}_2(\R))$ with infinite image.
Recall that any representation with image in $\mathrm{SO}_2(\R)$ lifts to $\SL \C$. We show that \cref{obstrSW} and \cref{obstrMin} are the only obstructions in this case.
\begin{prop}
A representation $\rho\in \Hom(\Gamma, \mathrm{SO}_2(\R)$ is in $\Hol(\mathcal P(n_1, \ldots, n_k))$ if and only if $\sum_i n_i$ is even and at least $2g$.
\end{prop}
\begin{proof}
Let us suppose that $\sum_i n_i$ is even and at least $2g$. By \cref{cyclic}, there exists a representation in $\Hom(\Gamma, \mathrm{SO}_2(\R))$ with finite image in $\Hol(\mathcal P(n_1, \ldots, n_k))$. Therefore since the closure $\Aut(\Gamma)\cdot \rho$ contains such a representation by\cref{ptitlemme}, it follows from the Ehresmann-Thurston principle (\cref{corET}) that $\rho\in \Hol(\mathcal P(n_1, \ldots, n_k))$.
\end{proof}

\paragraph{Holonomy with dense image in $\mathrm{O}_2(\R)$} We now suppose that $\rho\in \Hom(\Gamma, \mathrm{O}_2(\R))$ has dense image.
\begin{prop}\label{dense}
The representation $\rho$ is in $\Hol(\mathcal P(n_1, \ldots, n_k))$ if and only if $\sum_i n_i > 2g-2$ and $\sum_i n_i$ is even (resp. odd) if $\rho$ lifts to $\SL \C$ (resp. does not lift to $\SL \C$). 
\end{prop}
\begin{proof}
Let $\rho\in \Hom(\Gamma, \mathrm O_2)$ with dense image. There exists $\rho'\in \Hom(\Gamma, \D_n)$ in $\Hol(\mathcal P(n_1\ldots, n_k))$ for $n$ large enough by \cref{geodiedr}.
By  \cref{denseDiedr}, $\rho'$ is in the closure of $\Aut(\Gamma)\cdot\rho$ thus $\rho\in \Hol(\mathcal P(n_1, \ldots, n_k))$ by the Ehresmann-Thurston principle \cref{corET}.
\end{proof}
\paragraph{Holonomy with dense image in $\SO$}

\begin{prop}
A representation $\rho\in \Hom(\Gamma, \SO)$ with dense image is in $\Hol(\mathcal P(n_1, \ldots, n_k))$ if and only if both $\sum_i n_i > 2g-2$ and $\sum_i n_i$ is even if $\rho$ lifts to $\SL \C$ and odd otherwise.
\end{prop}

\begin{proof}
As in \cref{dense}, there exists $\rho'\in \Hom(\Gamma, \SO)$ in $\Hol(\mathcal P(n_1, \ldots, n_k))$. For example we can take $\rho'$ with its image in a dihedral group $\D_n$ by \cref{geodiedr}.
This also follows from \cite[Theorem A]{MondelloPanov} where the existence of spherical metrics with conical points of prescribed angles satisfying the Gauss-Bonnet constraint is shown.
%Recall that this representation $\rho'$ lifts to $\SL \C$ if and only if $\sum_i n_i$ is even.
By \cref{previtxia}, $\rho'$ is in the closure of $\Aut(\Gamma)\times \SO\cdot \rho$. Therefore $\rho\in \Hol(\mathcal P(n_1, \ldots, n_k))$ by the Ehresmann-Thurston principle, \cref{corET}.
\end{proof}

\subsection{Affine holonomy}
In this subsection we  geometrize affine representations.
An example of a branched projective structure with affine holonomy is an affine structure, as studied for example in \cite{Ghazouani}.
\begin{defn}
An \textit{affine structure} is a branched projective structure with affine holonomy whose developing map takes its values in $\C\subset \CP$.
\end{defn}

This definition generalizes for example the notion of translation surface. 
\begin{defn}
A translation surface structure $X$ on $\Sigma$ is an affine structure with holonomy in $\Hom(\Gamma, \C)$, where $\C\subset \mathrm{Aff(\C)}$ is the subgroup of translations. An even half-translation structure is an affine structure with holonomy in $\Hom(\Gamma, \{\pm 1\}\rtimes \C)$.
\end{defn}

\begin{rmk}
This amounts to requiring that $\alpha = \mathrm{Li\circ \rho}$ is the trivial homomorphism for the translation surface case, and that $\alpha$ has its image in $\{\pm 1\}$ for the half-translation case.
\end{rmk}

\begin{rmk}
Our definition of even half-translation surface is equivalent to the usual definitions of half-translation surfaces, see for example \cite{Zorich, Lanneau}, with the extra requirement that the total angle of every conical point is even in the following sense: it is of the form $k\pi$, with $k\in 2\mathbb N$. In other words, following the notations of \cite{Lanneau}, we are looking at the half-translation surfaces that are either translation surfaces or correspond to quadratic differentials in  $$\bigcup_{n_i\geqslant 1 \mid \sum_i n_i = 2g-2} \mathcal Q(2n_1, 2n_2, \ldots, 2n_k).$$
This extra requirement is necessary to obtain a holonomy defined on $\Gamma$, and not on the fundamental group of a punctured surface.
\end{rmk}

\subsubsection{General obstructions}
Let us show that the divisor of a branched projective structure with affine holonomy has degree at least $2g-2$ and that we may read in this degree whether such a structure is an affine one or not.

\begin{prop}\label{liredegre}
If $\rho\in \Hom(\Gamma, \mathrm{Aff}(\C))$ is the holonomy of a projective structure in $\mathcal P(n_1, \ldots, n_k)$ with developing map $f$, then $\sum_i n_i \geqslant 2g-2$. Moreover $\sum n_i = 2g-2 $ if and only if $f(\tilde \Sigma)\subset \C$. 
\end{prop}

The following proof of \cref{liredegre} is an adaptation of \cite[Section 6]{Kapovich}, where Kapovich shows the same result in the realm of holonomies with values in the subgroup of translations $\C\subset \mathrm{Aff(\C)}$.

\begin{proof}
Let us assume that $\rho$ is the holonomy of a branched projective structure with holomorphic developing map $f : \tilde \Sigma\to \CP$ for a complex structure on $\Sigma$.
Let us choose a lift $\tilde \rho\in \Hom(\Gamma, \SL\C)$ of $\rho$ and define $\varphi\in \Hom(\Gamma, \C^\star)$ by $\varphi(\gamma) = \lambda_\gamma$ where $\tilde \rho(\gamma) = \begin{pmatrix}
\lambda_\gamma^{-1} & t_\gamma\\
0  & \lambda_\gamma
\end{pmatrix}$. Let us consider the vector bundles $V = \tilde \Sigma\times \C^2/\Gamma$ and  $\Lambda = \tilde \Sigma\times \C/\Gamma$, where $\Gamma$ acts by $\gamma\cdot(\tilde x,v) = (\gamma\cdot \tilde x , \tilde \rho(\gamma) \cdot v)$ and $\gamma\cdot(\tilde x, z) = (\gamma \cdot \tilde x, \varphi(\gamma)\cdot z)$. The projection on the second coordinate gives a morphism of holomorphic bundles $\phi : V\to \Lambda$. The developing map gives for each point $\tilde x\in \tilde \Sigma$ a line $f(\tilde x)\subset \C^2$, and by $\rho$-equivariance induces a line subbundle $L\subset V$. We restrict $\phi$ to get a morphism of line bundles $\phi : L\to \Lambda$, which is injective if and only if $f(\tilde \Sigma)$ does not contain $\infty\in \CP$.

\begin{lemma}
We have  $\deg L \leqslant \deg \Lambda$ with equality if and only if $\phi$ is injective.
\end{lemma}
\begin{proof}
Let us pick a smooth section $s : \Sigma\to L$ transverse to the zero section whose zeroes are different from those of $\phi : L\to \Lambda$. We also assume that $\phi\circ s$ is transverse to the zero section. The degree of $L$ is the intersection number between $s(\Sigma)$ and the zero section: $\deg L = \sum_{p\in s^{-1}(0)} \mathrm{sign}_p(s)$, where $\mathrm{sign}_p(s)$ is the sign of $\det(ds_p)$ where $s$ is viewed as $s : U\subset \C\to \C$ in a local chart around $p$, and a local trivialization of $L$.
The zeroes of the section $\phi\circ s$ of $\Lambda$ are of two types: the zeroes of $s$ and the zeroes of $\phi$. Each zero of $\phi$ has sign $1$, because in a local chart centered at $0$ and a local trivialization, we have $\phi\circ s : z\mapsto \alpha(z) s(z)$ where $\alpha$ is holomorphic, $\alpha(0) = 0$ and $s(0) \neq 0$. Therefore $d(\phi\circ s) = s(0)\cdot d\alpha_0$ has positive determinant. If $p$ is a zero of $s$, we have $\alpha(0) \neq 0$ and $s(0) = 0$, thus $d(\phi\circ s)_0 = \alpha(0)\cdot ds_0$ and $\mathrm{sign}_p(\phi\circ s) = \mathrm{sign}_p(s)$. We have
\[
\deg \Lambda = \sum_{p\in (\phi\circ s)^{-1}(0)} \mathrm {sign}_p(\phi\circ s)\\
			 =  \sum_{p\in s^{-1}(0)} \mathrm{sign}_p(\phi\circ s) + \#\{p\in \Sigma \mid \phi = 0\}.
\]
Therefore, $\deg \Lambda = \deg L + \#\{p\in\Sigma \mid \phi = 0 \} \geqslant \deg L$ with equality if and only if $\phi$ is injective.
\end{proof}

\begin{lemma}
We have $\deg \Lambda = 0$.
\end{lemma}

\begin{proof}
It suffices to show that this degree does not depend on the homomorphism $\varphi \in \Hom(\Gamma, \C^\star)$, since for $\varphi : \gamma\mapsto 1$, the bundle $\Lambda$ is trivial. Let $(\varphi_t)_{t\in [0,1]}$ be a path in $\Hom(\Gamma, \C^\star)$, and $\Lambda_t$ be the bundle over $\Sigma$ defined as $\Lambda$ with the action of $\varphi_t$. Let $W$ be the vector bundle over $\Sigma\times [0,1]$ defined by $W = \tilde \Sigma\times \C\times [0,1]/\Gamma$ where $\Gamma$ acts by $\gamma\cdot (\tilde x, z, t) = (\tilde \gamma\cdot x, \varphi_t(\gamma)\cdot z, t)$. Let us choose a linear connection on $W$. The parallel transport along curves $x\times [0,1]$, $x\in \Sigma$, defines an isomorphism between the underlying oriented real bundles $\Lambda_0$ and $\Lambda_1$. Therefore $\Lambda_0$ and $\Lambda_1$ have the same Euler class and thus the same degree. Since $\Hom(\Gamma, \C^\star) \simeq (\C^\star)^{2g}$ is path-connected, the degree of $\Lambda$ does not depend on $\varphi$.
\end{proof}
It is shown in \cite[Section C]{GKM} that \[\deg L = g - 1 - \frac{1}{2}\sum_i n_i.\]
Therefore we have $\sum_i n_i \geqslant 2g-2$, with equality if and only if $\infty \notin f(\tilde \Sigma)$.
\end{proof}

We have seen in \cref{finite_lin} that the group $\mathrm{GL}_2^+(\R)$ acts on $\Hom(\Gamma, \Z_2\rtimes \C)$, the set of homomorphisms with image in the group $\mathrm{Aff}(\C)$ and linear part with finite image of order at most $2$. As a corollary of \cref{liredegre}, we see that this action preserves the property of being in the image of $\Hol$ if $\sum_i n_i = 2g-2$.
\begin{cor}
Let $\rho\in \Hom(\Gamma, \Z_2\rtimes \C)$ and $A\in \mathrm{GL}_2^+(\R)$. For every $1\leqslant n_1\leqslant\ldots\leqslant n_k$ such that $\sum_i n_i = 2g-2$, we have:
\[\rho\in \Hol(\mathcal P(n_1, \ldots, n_k)) \iff A\cdot \rho\in \Hol(\mathcal P(n_1, \ldots, n_k)).\]
\end{cor}

\begin{proof}
If $\rho\in \Hol(\mathcal P(n_1, \ldots, n_k))$, then there exists a developing map $f : \tilde \Sigma\to \CP$ that is $\rho$-equivariant. By \cref{liredegre}, $f$ takes its values in $\C$. Let us see $A$ as a map $A : \C\to \C$.
The map $A\circ f$ is a developing map for a projective structure in $\mathcal P(n_1, \ldots, n_k)$ with holonomy $A\cdot \rho$.
\end{proof}
\subsubsection{Euclidean holonomy}\label{Euclide}

We now focus on the case where $\rho$ is Euclidean. Recall that it means that the linear part $\alpha = \mathrm {Li}\circ \rho\in \Hom(\Gamma, \C^\star)$ of $\rho$ has its image in $\mathbb S^1 = \{z\in \C \mid z\overline{z} = 1\}$. 
Let us explain \cref{obstrVol}.
\begin{prop}
Suppose that $\rho\in \Hom(\Gamma, \mathrm{Isom}^+(\mathbb E^2))$ is the holonomy of an affine structure. We have $\mathrm{Vol}(\rho) > 0$.
\end{prop}

\begin{proof}
%A developing map $f$ gives a branched atlas $(U_i, \varphi_i)_{1\leqslant i\leqslant k}$ with transition maps in $\mathrm{Isom}^+(\mathbb E^2)$. Let us choose a partition $(f_i)_{1\leqslant i\leqslant k}$ of unity subordinate to the open cover $(U_i)_{1\leqslant i\leqslant k}$. The induced form $\omega = f^*(dx\wedge dy)$ is equal to $\sum_{i=1}^k \omega_i$, where ${\omega_i}_{|U_i} = f_i \omega_{|U_i}$, and $\omega_i = 0$ on $\Sigma\setminus U_i$. We have $\int_\Sigma \omega = \sum_i \int_\Sigma \omega_i$.

We can identify $\tilde \Sigma$ with $Y$, where $Y = \mathbb H^2$ if $g\geqslant 2$ and $Y = \C$ if $g=1$,  by the uniformization theorem and consider a holomorphic developing map $f : Y\to \C$. Let us choose $D\subset Y$ a compact fundamental domain for the action of $\Gamma$ on $Y$. We have: $$\mathrm{Vol}(\rho) = \int_D f^*(dx\wedge dy) = \int_D \det(df) dx\wedge dy.$$
Since $f$ is holomorphic and nonconstant, $\det(df) \geqslant 0$ everywhere and $\det(df)$ is continuous and nonzero.
\end{proof}

\paragraph{Trivial linear part}

Let us first assume that $\alpha$ is the trivial homomorphism.
Let $1\leqslant n_1 \leqslant \ldots \leqslant n_k$ be such that $\sum_i n_i = 2g-2$. 
%We want to understand under what conditions $\rho$ is in $\Hol(\mathcal P(n_1, \ldots, n_k))$.
By \cref{liredegre} we are asking if $\rho$ is the holonomy of a translation surface with prescribed singularities. This question was addressed by the author in \cite{Moi} and independently by Bainbridge, Johnson, Judge and Park in \cite{Judge}, where the following is proven.
\begin{HauptMoi}\label{I}
We have $\rho\in \Hol(\mathcal P(n_1, \ldots, n_k))$ if and only if 
\begin{enumerate}
\item $\mathrm{Vol}(\rho) > 0$
\item $\mathrm {Vol}(\rho) \geqslant (\max_i n_i+1)\mathrm{Area}(\C / \Lambda)$ if $\Lambda = \{z_0\in \C\mid z + z_0\in \rho(\Gamma)\}$ is a lattice in $\C$.
\end{enumerate}
\end{HauptMoi}
Let us remark that \cref{I} shows that \cref{obstrVol} and \cref{obstrHaupt} are the only obstructions of being in $\Hol(\mathcal P(n_1, \ldots, n_k))$ for a representation $\rho\in \Hom(\Gamma, \C)$ if $\sum_i n_i = 2g-2$.

We now suppose $\sum_i n_i \geqslant 2g$.
Since $\sum_i n_i$ is not minimal, the developing maps we are considering have $\infty\in \CP$ in their image. Therefore our projective structures may be obtained by performing surgeries on $\CP$.
\begin{prop}\label{trivialLin}
Let $1\leqslant n_1 \leqslant \ldots \leqslant n_k$ be such that $\sum_i n_i$ is even and at least $2g$. There exists a branched projective structure in $\mathcal P(n_1, \ldots, n_k)$ with holonomy $\rho$.
\end{prop}
\begin{proof}
We construct a projective structure on $\Sigma_g$ with branching data $n_1, \ldots, n_k$ and holonomy in $\overline{(\Aut(\Gamma)\times \mathrm{Aff}(\C))\cdot \rho}$. It will then follow from \cref{corET} that $\rho\in \Hol(\mathcal P(n_1, \ldots, n_k))$.

Let us consider the Riemann sphere $\CP$. We choose $x\in \C\subset \CP$ and cut open $\CP$ along the segment lines $l_0 = [x, x + 1]$ and $l_k = [x + 2k-1, x + 2k]$ for $1\leqslant k \leqslant g$. We glue the bottom side of $l_i$ to the top side of $l_{i+1}$, in cyclic notation, as indicated in \cref{trivlin}.
\begin{figure}[h]
    \centering    
    \def\svgwidth{\columnwidth}
  	\def\svgwidth{0.5\textwidth}

%	\hspace*{2cm}
	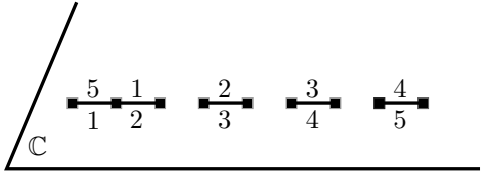
	\caption{Construction for $g=4.$}
    \label{trivlin}
\end{figure}

We thus get a genus $g$ surface with charts in $\CP$, and a unique branch point of order $2g$. Moreover its holonomy is given by $(z + 1, z+1, z + 2, z, \ldots z + 2, z)$, that lies in $\Aut(\Gamma)\cdot \rho_\infty$, where $\rho_\infty$ comes from \cref{ecrasement}. Hence this holonomy is in $\overline{\Aut(\Gamma)\times \mathrm{Aff}(\C)\cdot \rho}$ by \cref{ecrasement}. Observe that the closed curve following $l_0$ develops injectively. Therefore by \cref{bubbling}, $\rho_\infty\in \Hol(\mathcal P(2k))$ for every $k\geqslant g$. It follows from \cref{corexplo} and the Ehresmann-Thurston principle \cref{corET} that $\rho\in \Hol(\mathcal P(n_1, \ldots, n_k))$.
\end{proof}

\paragraph{Nontrivial linear part}
We now suppose that $\alpha$ has finite image of order $n\geqslant 2$.
Let  us first suppose that $1\leqslant n_1\leqslant \ldots\leqslant n_k$ are such that $\sum_i n_i$ is even and at least $2g$.
We may perform the same surgeries as in \cref{trivialLin}.
\begin{prop}\label{degree_at_least_2g}
There exists a branched projective structure in $\Hol(n_1, \ldots, n_k)$ with holonomy $\rho$.
\end{prop}

\begin{proof}
Our construction is very similar to the case $n=1$, where we cut and glue $\CP$ as in \cref{trivlin}. We will just flip one of the segment lines, see \cref{finitehol}. More precisely, we start as before with the Riemann sphere $\CP$. We cut open the sphere $\CP$ along the line segments $[x, x+1]$, $[x + 1, x + 2], [x + 3 , x + 4],\ldots, [x + 2g-3, x + 2g-2]$ and $[\zeta x, \zeta x + \zeta ]$, where $x\in \C$ is chosen so that these line segments do not intersect. We glue back the resulting boundary components as indicated in \cref{finitehol}.
\begin{figure}[h]
    \centering    
    \def\svgwidth{\columnwidth}
  	\def\svgwidth{0.5\textwidth}

%	\hspace*{2cm}
	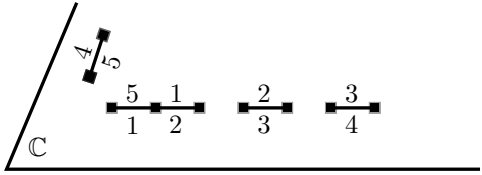
	\caption{Construction for $g=4$ and $n = 4$.}
    \label{finitehol}
\end{figure}
We thus get a genus $g$ surface with charts in $\CP$, a single branch point of order $2g$ and holonomy in $\Aut(\Gamma)\cdot \rho_\infty$, thus in $\overline{\Aut(\Gamma)\times \mathrm{Aff}(\C)\cdot \rho}$. 
The curve following the line segment $[x, x+1]$ develops injectively, thus $\rho_\infty\in \Hol(\mathcal P(2k))$ for every $k\geqslant g$ by \cref{bubbling}. Therefore $\rho\in\Hol(\mathcal P(n_1, \ldots, n_k))$ by \cref{corexplo} and \cref{corET}.
\end{proof}

\subparagraph{Minimal degree}
We suppose that $\alpha$ has finite image of order $n\geqslant 2$ and that $1\leqslant n_1\leqslant \ldots\leqslant n_k$ are such that $\sum_i n_i = 2g-2$. 
Let us first make an observation on the volume of the pullback.
\begin{rmk}
Suppose that $f : \Sigma\to S$ is a branched cover and that $S$ is endowed with a branched projective structure $X$ with holonomy $\rho$.
The volume of the holonomy of $f^* X$ is $\deg(f)\cdot \mathrm{Vol}(\rho)$.
\end{rmk}

This observation leads to another obstruction to being in $\Hol(\mathcal P(n_1, \ldots, n_k))$ that we now describe. Let $\Lambda$ be the subgroup of $\C$ generated by the complex numbers $a$ such that the M\"obius transformation $z+a$ belongs to $\ker \alpha$. Suppose that $\Lambda$ is a lattice in $\C$. If $\rho$ is the holonomy of the developing map $f$, then since $f$ has image in $\C$ by \cref{liredegre}, it induces a branched cover $S\to \C/\Lambda$, where $S$ is the cover of $\Sigma$ associated with $\ker \alpha$. The degree $d$ of this branched cover must satisfy $d \geqslant \max_i n_i + 1$, since the branched cover has the local form $z\mapsto z^{n_i+1}$ at the conical points. Moreover $d$ is also given by $d = n \mathfrak{\mathrm{vol}(\rho)}/{\mathrm{Area(\Lambda)}}$. This explains \cref{obstrHaupt}. Note that when $n=1$ this obstruction is included in the  statement of the refined Haupt's theorem.

\begin{ex}\label{exemple}
Every half-translation surface of genus $2$ with a single conical point is a translation surface. Indeed let $\rho$ be the holonomy of an even half-translation surface with nontrivial $\alpha = \mathrm{Li}\circ \rho$. Since $\alpha(\Gamma) = \{\pm 1\}$, we can suppose that $\rho = (-z, z, z + z_2, z + z_3)$ by \cref{OOOD}. We have $\Lambda = \langle z_2, z_3 \rangle$ and $\mathrm{Area}(\Lambda) = \det(z_2, z_3) = \mathrm{Vol}(\rho)$. Thus the group $\Lambda$ generated by $z_2$ and $z_3$ is a lattice in $\C$. \cref{obstrHaupt} gives $2\mathrm{Vol}(\rho)\geqslant (n_i +1)\mathrm{Vol}(\rho)$, thus $n_i = 1$ for all $i$. Since $\sum_i n_i = 2g-2 = 2$, there are $2$ conical points. In other words, we have $\mathcal Q(4) = \emptyset$, following the notation of \cite{Lanneau} ; see also
 \cite{MasurSmillie, Lanneau} for alternative proofs of this fact.%\color{red} bien vérifier que c'est différent \color{black}
\end{ex}
%
%However as we shall now see, it is rare that $\Lambda$ is a lattice: we must have $n\in \{1, 2, 3, 4, 6\}$. Indeed let $\zeta$ be a primitive $n$th root of unity. Conjugating $\rho$ with a translation if necessary, we can assume that the map $z\mapsto \zeta z$ is in the image of $\rho$ by \cref{OOOD}. Thus if $a\in \Lambda$, we also have $\zeta a\in \Lambda$, because one can conjugate $z\mapsto z+a$ with $z\mapsto \zeta z$. We also have $\zeta^2a\in \Lambda$, thus $\zeta^2\in \mathbb Q \oplus \zeta \mathbb Q$. 
%Therefore $[\mathbb Q(\zeta) : \mathbb Q]\leqslant 2$. Since $\varphi(n) = [\mathbb Q(\zeta) : \mathbb Q]$, we have $\varphi(n) \leqslant 2$, where $\varphi$ denotes the Euler's totient function. 
%%We leave it to the reader to check that it 
%This implies $n\in \{1,2,3,4,6\}$. Moreover for $n$ in that range, $\Lambda$ may be a lattice, for example if $\Lambda$ is the ring of Gaussian integers, or the Eisenstein integers.
%Our goal is to show that there are the only obstructions.
Let us show that \cref{obstrVol} and \cref{obstrHaupt} are in this case the only obstructions to being in $\Hol(\mathcal P(n_1, \ldots, n_k))$.
\begin{prop}\label{mineucl}
If $\mathrm{Vol}(\rho) > 0$ and $n\mathrm{Vol}(\rho)\geqslant (\max_i n_i +1)\mathrm{Area}(\Lambda)$ if $\Lambda$ is a lattice, then $\rho\in \Hol(\mathcal P(n_1, \ldots, n_k))$.
\end{prop}

We suppose that $\rho$ satisfies the conditions of \cref{obstrVol} and \cref{obstrHaupt}.
By \cref{OOOD}, we may assume that $$\rho = (e^{2\pi i / n} z, e^{2\pi i / n} z, z + x_2, z + y_2, \ldots, z + x_g, z + y_g).$$

\paragraph{Case $g = 2$.}
Let us first suppose that $g = 2$ and $n\geqslant 3$. 
We consider a parallelogram $P\subset \C$ whose vertices are $z_0$, $z_0 + x_2$, $z_0 + y_2$ and $z_0 + x_2 + y_2$, where $z_0\in \C$ is such that $0$ is in the interior of $P$. Note that it is indeed a parallelogram since $\mathrm{Vol}(\rho) = \det(x_2, y_2) > 0$.
 We are going to add a handle to the torus obtained by gluing its opposite edges. Let $\epsilon > 0$ be such that the disk $D = \{z \in \C \mid  |z| < \epsilon\}$ is included in the interior of $P$. Let us cut open $P$ along the circular segment joining $z_0 \in D$ and $e^{2\pi / n} z_0$, and along the circular segment joining $e^{2\pi / n}z_0$ and $e^{4\pi / n} z_0$. We then glue the boundary components as indicated in \cref{genre2fin}. 
\begin{figure}[h]
    \centering    
    \def\svgwidth{\columnwidth}
  	\def\svgwidth{0.5\textwidth}

%	\hspace*{2cm}
	%% Creator: Inkscape 1.0.1 (0767f8302a, 2020-10-17), www.inkscape.org
%% Accompanies image file 'genre2fin.pdf' (pdf, eps, ps)
\begingroup%
  \makeatletter%
  \providecommand\color[2][]{%
    \errmessage{(Inkscape) Color is used for the text in Inkscape, but the package 'color.sty' is not loaded}%
    \renewcommand\color[2][]{}%
  }%
  \providecommand\transparent[1]{%
    \errmessage{(Inkscape) Transparency is used (non-zero) for the text in Inkscape, but the package 'transparent.sty' is not loaded}%
    \renewcommand\transparent[1]{}%
  }%
  \providecommand\rotatebox[2]{#2}%
  \newcommand*\fsize{\dimexpr\f@size pt\relax}%
  \newcommand*\lineheight[1]{\fontsize{\fsize}{#1\fsize}\selectfont}%
  \ifx\svgwidth\undefined%
    \setlength{\unitlength}{105.72760631bp}%
    \ifx\svgscale\undefined%
      \relax%
    \else%
      \setlength{\unitlength}{\unitlength * \real{\svgscale}}%
    \fi%
  \else%
    \setlength{\unitlength}{\svgwidth}%
  \fi%
  \global\let\svgwidth\undefined%
  \global\let\svgscale\undefined%
  \makeatother%
  \begin{picture}(1,0.43271575)%
    \lineheight{1}%
    \setlength\tabcolsep{0pt}%
    \put(0,0){\includegraphics[width=\unitlength,page=1]{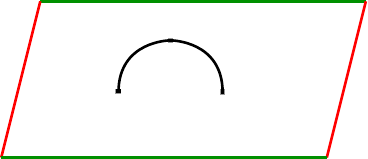}}%
    \put(0.62322967,0.161536){\makebox(0,0)[lt]{\lineheight{1.25}\smash{\begin{tabular}[t]{l}$\epsilon$\end{tabular}}}}%
    \put(0.44293367,0.34184544){\makebox(0,0)[lt]{\lineheight{1.25}\smash{\begin{tabular}[t]{l}$i\epsilon$\end{tabular}}}}%
    \put(0.23046118,0.1759888){\makebox(0,0)[lt]{\lineheight{1.25}\smash{\begin{tabular}[t]{l}$-\epsilon$\end{tabular}}}}%
    \put(0.52697437,0.24197357){\makebox(0,0)[lt]{\lineheight{1.25}\smash{\begin{tabular}[t]{l}$+$\end{tabular}}}}%
    \put(0.58512442,0.27310748){\makebox(0,0)[lt]{\lineheight{1.25}\smash{\begin{tabular}[t]{l}$-$\end{tabular}}}}%
    \put(0.30977307,0.27023896){\makebox(0,0)[lt]{\lineheight{1.25}\smash{\begin{tabular}[t]{l}$+$\end{tabular}}}}%
    \put(0.36611614,0.25204262){\makebox(0,0)[lt]{\lineheight{1.25}\smash{\begin{tabular}[t]{l}$-$\end{tabular}}}}%
  \end{picture}%
\endgroup%

	\caption{Construction for $g=2$ and $n = 4$.}
    \label{genre2fin}
\end{figure}
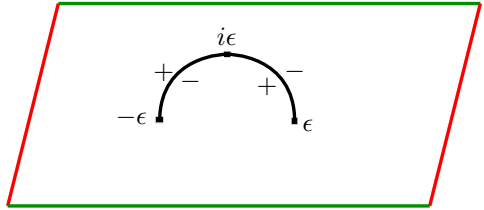
The topological result of this surgery is to add a handle. Moreover the resulting genus $2$ surface comes naturally with charts in $\C$ and has holonomy $\rho$. Thus $\rho\in \Hol(\mathcal P(2))$ and also $\rho\in \Hol(\mathcal P(1, 1))$ by \cref{explosion}.
If $n = 2$, then we have seen in \cref{exemple} that we cannot have $\rho\in \Hol(\mathcal P(2))$. However we have $\rho\in \Hol(\mathcal P(1, 1))$: it suffices to cut open $P$ along a line segment $l$ and along $-l$. We make sure that these line segments do not intersect and are both in $P$. We then identify the opposites sides of the boundary components as in \cref{genre2finfin} with the map $z\mapsto -z$. Note that this shows that $\mathcal Q(2, 2) \neq \emptyset$.

\begin{figure}[h]
    \centering    
    \def\svgwidth{\columnwidth}
  	\def\svgwidth{0.5\textwidth}

%	\hspace*{2cm}
	%% Creator: Inkscape 1.0.1 (0767f8302a, 2020-10-17), www.inkscape.org
%% Accompanies image file 'genre2finfin.pdf' (pdf, eps, ps)
\begingroup%
  \makeatletter%
  \providecommand\color[2][]{%
    \errmessage{(Inkscape) Color is used for the text in Inkscape, but the package 'color.sty' is not loaded}%
    \renewcommand\color[2][]{}%
  }%
  \providecommand\transparent[1]{%
    \errmessage{(Inkscape) Transparency is used (non-zero) for the text in Inkscape, but the package 'transparent.sty' is not loaded}%
    \renewcommand\transparent[1]{}%
  }%
  \providecommand\rotatebox[2]{#2}%
  \newcommand*\fsize{\dimexpr\f@size pt\relax}%
  \newcommand*\lineheight[1]{\fontsize{\fsize}{#1\fsize}\selectfont}%
  \ifx\svgwidth\undefined%
    \setlength{\unitlength}{105.72760377bp}%
    \ifx\svgscale\undefined%
      \relax%
    \else%
      \setlength{\unitlength}{\unitlength * \real{\svgscale}}%
    \fi%
  \else%
    \setlength{\unitlength}{\svgwidth}%
  \fi%
  \global\let\svgwidth\undefined%
  \global\let\svgscale\undefined%
  \makeatother%
  \begin{picture}(1,0.43271575)%
    \lineheight{1}%
    \setlength\tabcolsep{0pt}%
    \put(0,0){\includegraphics[width=\unitlength,page=1]{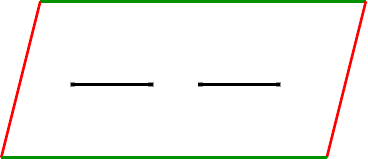}}%
    \put(0.50852171,0.1563949){\makebox(0,0)[lt]{\lineheight{1.25}\smash{\begin{tabular}[t]{l}$x$\end{tabular}}}}%
    \put(0.72965518,0.15822701){\makebox(0,0)[lt]{\lineheight{1.25}\smash{\begin{tabular}[t]{l}$y$\end{tabular}}}}%
    \put(0.39327266,0.15871072){\makebox(0,0)[lt]{\lineheight{1.25}\smash{\begin{tabular}[t]{l}$-x$\end{tabular}}}}%
    \put(0.13827305,0.15693985){\makebox(0,0)[lt]{\lineheight{1.25}\smash{\begin{tabular}[t]{l}$-y$\end{tabular}}}}%
    \put(0.62155368,0.15453612){\makebox(0,0)[lt]{\lineheight{1.25}\smash{\begin{tabular}[t]{l}$+$\end{tabular}}}}%
    \put(0.62862008,0.2139645){\makebox(0,0)[lt]{\lineheight{1.25}\smash{\begin{tabular}[t]{l}$-$\end{tabular}}}}%
    \put(0.28877008,0.21695048){\makebox(0,0)[lt]{\lineheight{1.25}\smash{\begin{tabular}[t]{l}$-$\end{tabular}}}}%
    \put(0.28080782,0.15775882){\makebox(0,0)[lt]{\lineheight{1.25}\smash{\begin{tabular}[t]{l}$+$\end{tabular}}}}%
  \end{picture}%
\endgroup%

	\caption{Construction for $g=2$ and $n = 2$.}
    \label{genre2finfin}
\end{figure}
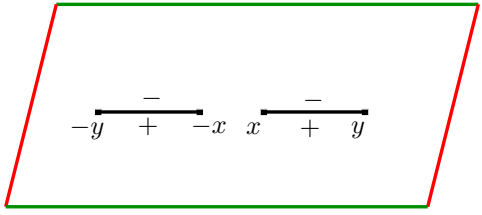
\begin{rmk}\label{Q(vide)}
Observe that if $g = 1$, then $\rho$ must be spherical by \cref{OOOD}. Indeed, we can suppose that $\rho = (e^{2\pi i /n} z, z)$, therefore we cannot have $\sum_i n_i = 0$. In particular this shows that $\mathcal Q(\emptyset) = \emptyset$.
\end{rmk}
We now suppose that $g\geqslant 3$.
\paragraph{The generic case.}
Let us first deal with the case where the group $\Lambda$ that is generated by the $e^{2\pi ki / n}x_j$ and $e^{2\pi ki / n}y_j$, $k\in \mathbb Z$ is not a lattice. We also suppose that $n\geqslant 3$. It follows from \cref{notAlattice} that we can assume that the group generated by the $x_j$ and $y_j$ is not a lattice. 
We may consider a translation surface structure on $\Sigma_{g-1}$ with holonomy $(z + x_2, \ldots, z + y_g)$ and a single conical point $p\in \Sigma_{g-1}$ by the refined Haupt's theorem. We can translate the charts so that there exist $\epsilon > 0$ and a chart $\varphi : U\to V$, where $U\subset S$ and $V\subset \C$ are open, such that $p\in U$, and $\varphi(p)\in D(0, \epsilon)\setminus \{0\}$ and $D(0, \epsilon)\subset V$. We can then add a handle with holonomy $(e^{2\pi i / n} z, z)$ as in the genus $g=2$ case. We make sure that the new slit starts from the conical point of the surface. The resulting surface has holonomy $\rho$ and a single conical point of order $2g-2$. It follows from \cref{corexplo} that $\rho\in \Hol(\mathcal P(n_1, \ldots, n_k))$.

Let us turn to the case $n = 2$.
If $g\geqslant 4$, there exists $\rho_\infty \in \Aut(\Gamma)\cdot \rho$ such that $\Lambda_\infty  = \{a\in \C\mid z + a\in \rho_\infty(\Gamma)\}$ is a lattice satisfying $\mathrm{Vol}(\rho_\infty) \geqslant (\max_i n_i +1)\mathrm{Area}(\Lambda_\infty)$ by \cref{LatticeClosure}.
% and \cref{decoupe}. 
We will see in the next paragraph that this implies $\rho_\infty\in \Hol(\mathcal P(n_1, \ldots, n_k))$. Therefore $\rho\in \mathcal \Hol(\mathcal P(n_1, \ldots, n_k))$ by the Ehresmann-Thurston principle \cref{corET}.
We now turn to the case $g = 3$ and $n = 2$. 
Let us recall a result from \cite[Section 4, Proposition 4.2]{Moi}.
\begin{lemma}\label{Moi4} Let $\chi \in \Hom(\Gamma_2, \C)$ with positive volume and image $\chi(\Gamma_2)$ that is not a lattice. There exists $\chi'\in \Aut(\Gamma_2)\times \mathrm{GL}_2^+(\R)\cdot\chi$ such that $\chi'(a_1) = 1$, $\chi'(b_1) = i$ and $|\chi'(a_2)| < 1$.
\end{lemma}

We may assume that our representation has the form $\rho = (-z, z, x_1, y_1, x_2, y_2)$, with the representation $\chi'$ defined by $\chi'(a_i) = x_i$ and $\chi'(b_i) = y_i$ for $i = 1,2$ as in \cref{Moi4}, by \cref{decoupe}.
We consider a parallelogram $P$ in $\C$ whose sides are given by $x_1$ and $y_1$. Let us consider a line segment $\ell = [z_0, z_0 + x_2]$ that is contained in $P$. We may translate both $P$ and $\ell$, so that $\ell$ is so close to $0$ that $-\ell$ is also contained in $P$. We make slits along $\ell$ and $-\ell$, and glue back a handle given by identifying two sides of a parallelogram whose sides are given by $x_2$ and $y_2$, as indicated in \cref{g3n2}.

\begin{figure}[h]
    \centering    
    \def\svgwidth{\columnwidth}
  	\def\svgwidth{0.7\textwidth}

%	\hspace*{2cm}
	%% Creator: Inkscape 1.0.1 (0767f8302a, 2020-10-17), www.inkscape.org
%% Accompanies image file 'g3n2.pdf' (pdf, eps, ps)
\begingroup%
  \makeatletter%
  \providecommand\color[2][]{%
    \errmessage{(Inkscape) Color is used for the text in Inkscape, but the package 'color.sty' is not loaded}%
    \renewcommand\color[2][]{}%
  }%
  \providecommand\transparent[1]{%
    \errmessage{(Inkscape) Transparency is used (non-zero) for the text in Inkscape, but the package 'transparent.sty' is not loaded}%
    \renewcommand\transparent[1]{}%
  }%
  \providecommand\rotatebox[2]{#2}%
  \newcommand*\fsize{\dimexpr\f@size pt\relax}%
  \newcommand*\lineheight[1]{\fontsize{\fsize}{#1\fsize}\selectfont}%
  \ifx\svgwidth\undefined%
    \setlength{\unitlength}{105.48390069bp}%
    \ifx\svgscale\undefined%
      \relax%
    \else%
      \setlength{\unitlength}{\unitlength * \real{\svgscale}}%
    \fi%
  \else%
    \setlength{\unitlength}{\svgwidth}%
  \fi%
  \global\let\svgwidth\undefined%
  \global\let\svgscale\undefined%
  \makeatother%
  \begin{picture}(1,0.28977816)%
    \lineheight{1}%
    \setlength\tabcolsep{0pt}%
    \put(0,0){\includegraphics[width=\unitlength,page=1]{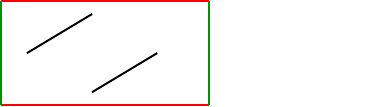}}%
    \put(0.27337219,0.14915323){\makebox(0,0)[lt]{\lineheight{1.25}\smash{\begin{tabular}[t]{l}$0$\end{tabular}}}}%
    \put(0,0){\includegraphics[width=\unitlength,page=2]{g3n2.pdf}}%
    \put(0.33229854,0.04133102){\rotatebox{32.194397}{\makebox(0,0)[lt]{\lineheight{1.25}\smash{\begin{tabular}[t]{l}$++$\end{tabular}}}}}%
    \put(0.86430829,0.0808815){\rotatebox{32.217192}{\makebox(0,0)[lt]{\lineheight{1.25}\smash{\begin{tabular}[t]{l}$++$\end{tabular}}}}}%
    \put(0.76879146,0.19493392){\rotatebox{32.194397}{\makebox(0,0)[lt]{\lineheight{1.25}\smash{\begin{tabular}[t]{l}$--$\end{tabular}}}}}%
    \put(0.15998702,0.15569625){\rotatebox{32.194397}{\makebox(0,0)[lt]{\lineheight{1.25}\smash{\begin{tabular}[t]{l}$--$\end{tabular}}}}}%
    \put(0.12547498,0.18761516){\rotatebox{32.194397}{\makebox(0,0)[lt]{\lineheight{1.25}\smash{\begin{tabular}[t]{l}$==$\end{tabular}}}}}%
    \put(0.30322722,0.08096382){\rotatebox{32.194397}{\makebox(0,0)[lt]{\lineheight{1.25}\smash{\begin{tabular}[t]{l}$==$\end{tabular}}}}}%
  \end{picture}%
\endgroup%

	\caption{Construction for $g=3$ and $n = 2$.}
    \label{g3n2}
\end{figure}
We thus get a genus $3$ surface with holonomy given by 
%\color{red} Vérifier encore... \color{black}
\[ (z + 1, z + i, z + x, -z + y, z-x, -z).\]
It follows from \cref{lemmeGenre2} that this representation is in $\Aut(\Gamma)\times \mathrm{GL}_2^+(\R)\cdot \rho$. Therefore $\rho\in \Hol(\mathcal P(n_1, \ldots, n_k))$.

\paragraph{The lattice case.}

Let us suppose that $n = 4$ and $\Lambda = L = \mathbb Z[i]$. We consider a rectangle $R$ whose sides are given by $m = \mathrm{Vol}(\rho)$ and $1$. Let us identify the opposite sides of $R$ with translations.
Let us draw $m$ annuli in $R$ obtained by translating one of them by integers, see \cref{lattice4}. We make the slits and gluing in these annuli explained in \cref{sectcycl} for the branched datum $n_1, \ldots, n_k$. Since $\max_i n_i < 4m$, we have room to make the slits. Since $g\geqslant 3$, we can also make sure that the resulting holonomy $\rho_0$ is a surjective homomorphism of $\Hom(\Gamma, \Z_4\rtimes \C)$, with $\Lambda(\rho_0) = \mathbb Z[i]$ and such that $\mathrm{Vol}(\rho) = m$. It follows from \cref{imagelattice} that $\rho\in \Hol(\mathcal P(n_1, \ldots, n_k))$.
\begin{figure}[h]
    \centering    
    \def\svgwidth{\columnwidth}
  	\def\svgwidth{0.5\textwidth}

%	\hspace*{2cm}
	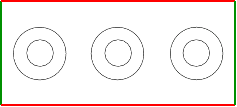
	\caption{Annuli in $R$ for $(n,m)=(4,3)$}
    \label{lattice4}
\end{figure}

If $\Lambda = L = \mathbb Z[\omega]$, that is $n\in \{3, 6\}$, we consider the parallelogram whose sides are given by $m = \mathrm{Vol}(\rho)$ and $\omega = e^{2\pi i /3}$, and draw $m$ annuli obtained by integer translations. We then make slits and gluing in these annuli corresponding to $n_1,\ldots, n_k$.

If $n = 2$, we may assume that $L = \Lambda = \mathbb Z[i]$: we may replace $\rho$ with $A\cdot \rho$, $A\in \mathrm{GL}_2^+(\R)$ if necessary. The construction is then similar to the case $n=4$.

\paragraph{Infinite linear part} Let us suppose that $\rho\in \Hom(\Gamma, \mathrm{Isom}^+(\mathbb E^2))$ is such that $\alpha = \mathrm{Li}\circ \rho$ is infinite. Suppose $1\leqslant n_1\leqslant \ldots\leqslant n_k$ are such that $\sum_i n_i = 2g-2$. Let us show that here, \cref{obstrVol} is the only obstruction to being in $\Hol(\mathcal P(n_1, \ldots, n_k))$.
\begin{prop}
The representation $\rho$ is in $\Hol(\mathcal P(n_1, \ldots, n_k))$ if and only if $\mathrm{Vol}(\rho) > 0$.
\end{prop}
\begin{proof}
Let us suppose that $\mathrm{Vol}(\rho) > 0$.
There exists $\rho'\in\Hom(\Gamma, \mathrm{Isom}^+(\mathbb E^2))$ such that $\rho'\in \Hol(\mathcal P(n_1, \ldots, n_k))$ and $\mathrm{Vol(\rho')} = \mathrm{Vol(\rho)}$. Indeed we can for example consider a translation surface with branch data $n_1, \ldots, n_k$ and rescale this translation surface so that it has the required volume. We then take $\rho'$ to be its holonomy. 
Since $\rho'$ is in the closure of $\Aut(\Gamma)\times \mathrm{Isom}^+(\mathbb E^2)\cdot \rho$ by \cref{orbiteEucl}, we have $\rho\in \Hol(\mathcal P(n_1, \ldots, n_k))$ by the Ehresmann-Thurston principle \cref{corET}.
\end{proof}

\subsubsection{Strictly affine holonomy}

In this section we geometrize \textit{strictly affine }representations.
By strictly affine, we mean that $\rho\in \Hom(\Gamma, \mathrm{Aff}(\C))$ is such that $\alpha = \mathrm{Li}\circ \rho\in \Hom(\Gamma, \C^\star)$ does not have its image contained in $\mathbb S^1\subset \C^\star$. This is equivalent to saying that $\rho$ is not Euclidean. We show that the only obstruction to being in $\Hol(\mathcal P(n_1, \ldots, n_k))$ is that $\sum_i n_i$ must be even and at least $2g-2$.

\begin{prop}\label{strictaffine}
Let $\rho$ be a strictly affine representation. We have $\rho \in \Hol(\mathcal P(n_1, \ldots, n_k))$ if $\sum_i n_i$ is even and at least $2g-2$.
\end{prop}
By \cref{lineaire} and the Ehresmann-Thurston principle \cref{corET}, it suffices to show that $\mathrm{Li}\circ\rho\in \Hol(\mathcal P(n_1, \ldots, n_k))$.

\begin{defn}
The exponential of a representation $\chi\in \Hom(\Gamma, \C)$ is the representation defined by $\exp(\chi) : \gamma\in \Gamma\mapsto  \exp(\chi(\gamma))z \in \mathrm{Aff}(\C)$.
\end{defn}

\begin{rmk}
Suppose $\chi\in \Hom(\Gamma, \C)$ is the holonomy of a translation surface structure with developing map $f : \tilde \Sigma\to \C$. The map $\exp \circ f$ is $\exp(\chi)$-equivariant.
\end{rmk}

\begin{lemma}
There exists a representation $\chi\in \Hom(\Gamma, \C)$ such that: 
\begin{enumerate}
\item $\exp(\chi) = \mathrm{Li}\circ \rho$
\item $\chi$ is the holonomy of a translation surface with a single conical point.
\end{enumerate}
\end{lemma}

\begin{proof}
These properties are invariant by the action of $\Aut(\Gamma)$. We can thus modify $\rho$ so that both $\alpha(a_1)$ and $\alpha(a_2)$ are not in $\mathbb S^1$, where $\alpha = \mathrm{Li}\circ \rho$.
Let us first choose $\chi(\gamma)\in \C$ such that $\alpha(\gamma) = \exp(\chi(\gamma))$ for each $\gamma\in \{a_1,\ldots,  b_g\}$. We thus have $\exp(\chi) = \alpha$. We may now replace $\chi(b_1)$ with $\chi(b_1) + 2\pi i k$ where $k\in \mathbb Z$ so that $\det(\chi(a_1), \chi(b_1))\neq 0$. We can then replace $\chi(b_2)$ with $\chi(b_2) + 2\pi i k$ so that $\mathrm{Vol}(\chi) \geqslant 2g |\det (\chi(a_1), \chi(b_1))|$. Therefore if the group $\Lambda$ generated by the $\chi(\gamma)$ for $\gamma\in \{a_1, \ldots, b_g\}$ is a lattice, its area is bounded above by $|\det(\chi(a_1), \chi(b_1))|$, and $\mathrm{Vol}(\chi) \geqslant 2g\mathrm{Area}(\Lambda)$. It follows from the refined Haupt's theorem that $\chi$ is the holonomy of a translation surface with a single conical point of angle $2\pi(2g-1)$.
\end{proof}

The explicit constructions made in \cite{Moi} show that if $\chi$ is the holonomy of a translation surface with a single conical point, then $\chi$ is also the holonomy of another translation surface with a single conical point such that there exists a simple closed curve $c$ that develops injectively on a line that is not parallel to the imaginary axis. The exponential of the associated developing map gives a translation surface with a single conical point and holonomy $\alpha = \exp(\chi)$. The closed curve $c$ based on the conical point develops injectively. Therefore by \cref{bubbling}, we have $\alpha\in \Hol(\mathcal P(2d))$ for each $d\geqslant g-1$, and thus $\alpha\in \Hol(\mathcal P(n_1, \ldots, n_k))$ if $\sum_i n_i$ is even and at least $2g-2$ by \cref{corexplo}.
%\color{red}
%[FAUT IL (RE)FAIRE LES DESSINS DU PAPIER SUR HAUPT ?]
%\color{black}
\subsection{Dihedral holonomy}

In this section, we suppose that $\rho \in \Hom(\Gamma, \PSL \C)$ is \textit{dihedral} but not affine, that is its image $\rho(\Gamma)$ fixes globally a pair of points in $\CP$, but not pointwise. After conjugating $\rho$, we may assume that $\rho(\Gamma)$ is included in the group $\{z\mapsto c z^\epsilon \mid c\in \mathbb C, \epsilon = \pm 1\}$. We require that $\rho$ is not conjugated into $\SO$ nor in $\mathrm{Aff}(\C)$, that is there exist $\gamma$ and  $\delta$ in $\Gamma$ such that $\rho(\gamma) = c_\gamma z$ with $|c_\gamma|\neq 1$ and $\rho(\delta) = c_\delta z^{\epsilon(\delta)}$ with $\epsilon(\delta) = -1$.

Observe that $\epsilon : \Gamma\to \{\pm 1\}$ is a group homomorphism. Thus we may assume that it satisfies $\epsilon(a_1) = -1$ and $\epsilon(\gamma) = 1$ for $\gamma\in \{b_1, a_2, \ldots b_g\}$ by \cref{ModCycl}. We may conjugate $\rho$ with an isometry of the form $\begin{pmatrix}
\lambda & 0\\
0 & \lambda^{-1}
\end{pmatrix}$ so  that $c_\gamma = 1$ for $\gamma = a_1$. Since $\rho([a_i, b_i]) = \id$ for $i\geqslant 2$, we must have $\rho([a_1, b_1]) = \id$. Therefore $c_\gamma\in \{1, -1\}$ for $\gamma = b_1$.
It follows from \cref{sw1} that $\rho$ lifts to $\SL \C$ if and only if $c_\gamma = 1$, where $\gamma = b_1$.
Note that $\rho(\Gamma)$ is not abelian. Therefore we assume that $g\geqslant 2$ in this section.

\subsubsection{Even Stiefel Whitney class}

We first assume that $\rho$ lifts to $\SL \C$: $c_\gamma = 1$ for $\gamma = b_1$.
We may extend the definition of the exponential of a representation to the representations of the form $\rho' : \gamma \mapsto \epsilon(\gamma)z + c_\gamma$, where $\epsilon\in \Hom(\Gamma, \{\pm 1\})$.

\begin{defn}
\begin{enumerate}
\item The exponential of $\rho'$ is the representation $$\exp(\rho') : \gamma\mapsto \exp(c_\gamma)z^{\epsilon(\gamma)}.$$
\item The exponential of a half-translation surface $Y$ with developing map $f$ is the projective structure $X = \exp(Y)$ with developing map $\exp\circ f$.
\end{enumerate}
\end{defn}

\begin{rmk}
We have $\Hol\circ \exp = \exp \circ \Hol$.
\end{rmk}

\begin{prop}
Let $X\in \mathcal P(n_1, \ldots, n_k)$ with $\sum_i n_i = 2g-2$ and holonomy $\rho$. There exists a half-translation surface structure $Y$ on $\Sigma$ such that $X = \exp(Y)$.
\end{prop}

\begin{proof}
Let us consider a developing map $f : \tilde{\Sigma}\to \C$ associated with $X$. We consider $S$ the double-cover of $\Sigma$ associated with $\ker \epsilon\triangleleft \Gamma$. The developing map induces an affine structure on $S$ with a branching divisor of degree $-2\chi(\Sigma) = -\chi(S)$. Therefore the image of $f$ does not contain $\infty\in \CP$ by \cref{liredegre}. There exists $\lambda\in \C^\star$ such that $z\mapsto \lambda z^{-1}$ is in $\rho(\Gamma)$, thus $f(\tilde{\Sigma})$ does not contain $0$ either.
Now let $\tilde f : \Sigma\to \C$ be the lift of $f : \tilde{\Sigma}\to \C^\star$ to the universal cover $\C$ of $\C^\star$. Recall that $\exp : \C\to \C^\star$ is the associated covering map. 
For every $\gamma\in \Gamma$, and for every $z\in \C$, we have $\exp(\tilde f(\gamma\cdot z)) = f(\gamma\cdot z) = c_\gamma f(z)^{\epsilon(\gamma)}$, where $\rho(\gamma) : z \mapsto c_\gamma z^{\epsilon(\gamma)}$. Therefore, $\tilde f(\gamma\cdot z) = \epsilon(\gamma) \tilde f(z) + l_\gamma$ where $\exp(l_\gamma) = c_\gamma$. By continuity, $l_\gamma$ does not depend on $z$. We thus define a map $\rho' : \Gamma\to \mathrm{Isom}^+(\mathbb E^2)$ by $\rho'(\gamma) = \epsilon(\gamma) z + l_\gamma$. It follows from the $\rho'$-equivariance of $\tilde f$ that $\rho'$ is a group homomorphism. Indeed the relation $$\tilde f(\gamma_1\gamma_2 \cdot z) = \rho'(\gamma_1\gamma_2)\cdot\tilde f(z) = \rho'(\gamma_1)\rho'(\gamma_2)\cdot\tilde f(z)$$ gives $\rho'(\gamma_1\gamma_2) = \rho'(\gamma_1)\rho'(\gamma_2)$ since $\tilde f$ is open and any element of $\PSL\C$ is determined by its action on $3$ points in $\CP$. The half-translation surface structure $Y$ with developing map $\tilde f$ satisfies $\exp(Y) = X$. 
\end{proof}

We now explain \cref{obstrGenre2}. Let us consider a projective structure $X$ on $\Sigma_2$ with branching divisor of degree $2$ and with dihedral holonomy that is not affine.
There exists a half-translation structure $Y$ on $\Sigma_2$ such that $X = \exp(Y)$. We have seen in \cref{exemple} that it is not possible for $Y$ to have a single conical point, thus it must have two conical points of angle $4\pi$ each and so must $X$.

Let us see that this obstruction only occurs when $(g,k)=(2,1)$.
\begin{prop}\label{justavan}
Suppose $1\leqslant n_1 \leqslant \ldots \leqslant n_k$ are such that $\sum_i n_i = 2g-2$. The representation $\rho$ is in $\Hol(\mathcal P(n_1, \ldots, n_k))$ if and only if $(g,k) \neq (2,1)$. 
\end{prop}

\begin{proof}
Let us first suppose that $g\geqslant 3$. We may suppose that $|c_\gamma|\neq 1$ for $\gamma \in \{a_2, a_3\}$, since the map $\gamma\in \Gamma_{g-1}\mapsto \log(|c_\gamma|)\in \R$ is a non trivial homomorphism of $\Hom(\Gamma_{g-1}, \R)$. We take a representation $\rho' : \gamma\mapsto \epsilon (\gamma)z + l_\gamma$ with $\exp(\rho') = \rho$. We may also suppose that $\rho'$ is geometric: changing $\rho'(\gamma)$ with $\rho'(\gamma) + 2k\pi i$ for $\gamma \in\{b_2, b_3\}$, we can make sure that $\mathrm{Vol}(\rho') > 0$ and that $2\mathrm{Vol}(\rho') \geqslant (n_k + 1) \mathrm{Area}(\Lambda)$ in the case where $\Lambda = \{z_0\in \mathbb C \mid z + z_0\in \rho'(\Gamma)\}$ is a lattice in $\C$. Therefore, there exists by \cref{mineucl} a half-translation structure $Y$ with holonomy $\rho'$ and a single conical point. We consider its exponential $X = \exp(Y)$.

Let us now turn to the case $g = 2$ and $k \geqslant 2$. Observe that $k = 2$ and $n_1 = n_2 = 1$ since $\sum_i n_i = 2$. We may consider $\rho'$ as before such that $\rho = \exp(\rho')$, where $\rho' = (z + x, z + y, -z, z)$. We may also assume that $\mathrm{Vol}(\rho) = \det(x_1, y_1)$ is positive; replacing $x$  with $x + 2k\pi i$ for some $k\in \Z$ if necessary (or $y$ with $y + 2k\pi i$). We geometrize $\rho'$ by \cref{mineucl} and take the exponential of the associated developing map.
\end{proof}

We now suppose that $1\leqslant n_1\leqslant \ldots \leqslant n_k$ are such that $\sum_i n_i$ is even and at least $2g$.
We show that there are no more obstructions to being in $\Hol(\mathcal P(n_1, \ldots, n_k))$.
\begin{prop}
The representation $\rho$ is in $\Hol(\mathcal P(n_1, \ldots, n_k))$. 
\end{prop}
\begin{figure}[h]
    \centering    
    \def\svgwidth{\columnwidth}
  	\def\svgwidth{0.75\textwidth}

%	\hspace*{2cm}
	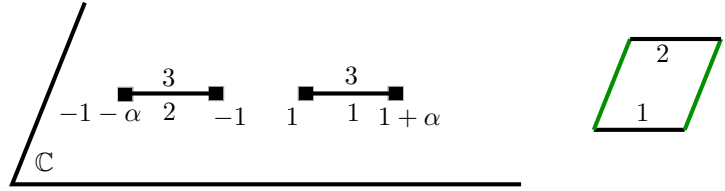
	\caption{Construction for $(g,k,n_1)=(2,1,4)$.}
    \label{diedrFig}
\end{figure}

\begin{proof}
Let us first suppose that $g\geqslant 3$.
We have seen that $\rho\in \Hol(\mathcal P(2g-2))$. We may check on the constructions we have made that we can suppose that there exists a curve based at the conical point of $Y$ whose exponential develops injectively, where $Y$ comes from the proof of \cref{justavan}. Therefore by \cref{bubbling} we have $\rho\in \Hol(\mathcal P(2d))$ for every $d\geqslant g-1$, and $\rho\in \Hol(\mathcal P(n_1, \ldots, n_k))$ by \cref{explosion}.

Let us now suppose that $g = 2$. Let us denote $\rho = (z^{-1}, z, c_2 z, c_3 z)$. We can assume that $|c_2|\neq 1$, since $\rho$ is not spherical. Let us take $\alpha$ and $\beta \in \C$ such that $e^\alpha = c_2$ and $e^\beta = c_3$. We may assume that $\det(\alpha, \beta) > 0$, replacing $\beta$ with $\beta + 2\pi i k$ if necessary. Let us consider a parallelogram in $\C$ whose sides are given by $\alpha$ and $\beta$. We glue the sides corresponding to $\beta$ with the map $z\mapsto z + \alpha$. This gives projective charts on an annulus. We define new charts on this annulus by considering the exponential of the associated developing map.
Let us consider the path $\delta_1 = \exp\circ c$, where $c$ is the path along the line segment in $\C$ joining $1$ to $1+\alpha$, and $\delta_2 = \exp\circ d$, where $d = -c$ is the path along the line segments from $-1$ to $-1-\alpha$.
Let us cut open $\CP$ along the image of $\delta_i$ for $i = 1,2$. We glue the resulting boundaries to those of the annulus, see \cref{diedrFig}.
We obtain a genus $2$ surface with holonomy $\rho_1 = (c_2 z, c_3 z^{-1}, c_2^{-1} z, z^{-1}) = \exp(\rho')$, where $\rho' = (z + \alpha, -z + \beta, z -\alpha, -z)$. It follows from \cref{lemmeGenre2} that the representation $\rho'' = (-z, z, z + \alpha, z + \beta)$ is in $\Aut(\Gamma)\cdot \rho'$. Therefore, $\rho = \exp(\rho'')$ is in $\Aut(\Gamma)\cdot\rho_1$.
Since the curve $\delta$ develops injectively, we have $\rho\in \Hol(\mathcal P(2d))$ for every $d\geqslant 2$ by \cref{bubbling}.
 Hence $\rho\in \Hol(\mathcal P(n_1, \ldots, n_k))$ by \cref{explosion}.
\end{proof}

\subsubsection{Odd Stiefel Whitney class}
In this subsection we suppose that $\rho$ does not lift to $\SL \C$, that is $c_\gamma = -1$ for $\gamma = b_1$.

\begin{prop}
We have $\rho\in \mathrm{Hol}(\mathcal P(n_1, \ldots, n_k))$ if and only if $\sum_i n_i$ is odd and at least $ 2g-1$. 
\end{prop}
\begin{proof}
There exists a branched projective structure $X$ on $\Sigma_{g-1}$ with a single branch point of order $2g-4$ and holonomy $(c_{a_2}z, c_{b_2}z, \ldots, c_{b_g}z)$ by \cref{strictaffine}, which is obtained by taking the exponential of a translation surface structure $Y$.
Observe that we may translate the charts of $Y$, so that the conical point is sent to $z\in \mathbb R_{>0}$, such that the line segment $[-z, z]$ is included in the image of this chart. We will now make a surgery in the corresponding chart of $X$ to add a handle to this surface. Let us cut $\Sigma_{g-1}$ open along the segment $[e^{-z}, e^z]$. We will glue the resulting boundaries to the boundary  of the hemisphere $\overline{ \mathbb H^2} = \mathbb H^2\cup \mathbb {RP}^1\subset \CP$. We identify the bottom boundary of the slit with the line segment $[e^{-z}, e^z]\subset \partial  {\mathbb H^2}$ with the identity. We glue the top boundary of the slit with the line segment $[-e^z, -e^{-z}]$ of $\partial \mathbb H^2$ with the map $z\mapsto -z$. Finally we glue the two line segments $[-e^{-z}, e^{-z}]$ and $(\mathbb R\cup \{\infty\})\setminus [-e^z, e^z]$ of $\partial \mathbb H^2$ with the map $z\mapsto \frac{1}{z}$, see \cref{oddsw}.

\begin{figure}[h]
    \centering    
    \def\svgwidth{\columnwidth}
%  	\def\svgwidth{1\textwidth}
%	\hspace*{-1cm}
	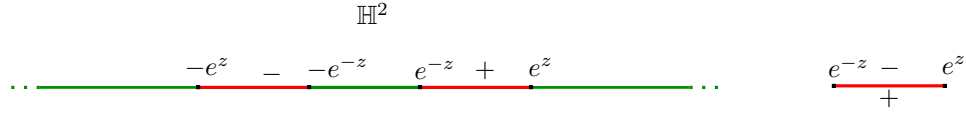
	\caption{$\overline{\mathbb H^2}$ on the left side. Its boundary is glued to those on the right side.}
    \label{oddsw}
\end{figure}

Topologically we have added a handle to $\Sigma_{g-1}$ and thus we get a genus $g$ surface. Moreover it has natural charts in $\CP$, that give a projective structure holonomy $\rho$. We have added  $3\times 2\pi$ to the total angle of the conical point thus $\rho\in \Hol(\mathcal P(2g - 1))$. The curve on which we have glued the handle develops injectively, hence $\rho\in \Hol(\mathcal P(2g-1 + 2k))$ for every $k\geqslant 0$ by \cref{bubbling}. The result follows from \cref{corexplo}.
\end{proof}

\subsection{Nonelementary holonomy}\label{SNE}

The celebrated theorem of Gallo, Kapovich and Marden proven in \cite{GKM} is in fact enhanced in the same paper, see \cite[Theorem 11.2.4]{GKM}.

\begin{theo}[Gallo-Kapovich-Marden]
Let $\rho\in \Hom(\Gamma, \PSL \C)$ be a nonelementary representation. We have $\rho\in \Hol(\mathcal P(n_1, \ldots, n_k))$ if and only if $\sum_i n_i$ has the same parity as $\sw(\rho)$.
\end{theo}

Their proof is based on the existence of Schottky decompositions of nonelementary representations, shown in \cite[Part A]{GKM}.
By a Schottky decomposition of $\rho$ we mean a pants decomposition $\Sigma = \cup_i P_i$ such that each restriction $\rho_{|P_i} : \pi_1(P_i)\to \PSL \C$ is an isomorphism onto a Schottky group.
Indeed, the techniques of construction of projective structures with given holonomy of Gallo, Kapovich and Marden allow them to prove directly this theorem when $\rho$ has a Schottky decomposition; see \cite[Part B, Theorem 11.2.4]{GKM}.
In genus $g=2$, there exist nonelementary representations $\rho\in \Hom(\Gamma, \PSL \C)$ that do not admit a Schottky decomposition: the  \textit{pentagon representations}, see \cite{Moi1}.
In the case where $\rho$ is a pentagon representation, there exists a branched projective structure on $\Sigma_2$ with a single conical point of total angle $2\pi(1 + 1)$ and holonomy $\rho$, see \cite[Section 5]{Moi1}. Moreover we can assume that there exists a curve based at the conical point that develops injectively, see the red curve in  \cite[Figure 8]{Moi1}, whose exponential develops injectively. Thus we have $\rho\in \Hol(\mathcal P(2d +1))$ for every $d\geqslant 0$ by \cref{bubbling} and $\rho\in\Hol( \mathcal P(n_1, \ldots, n_k))$ if $\sum_i n_i$ is odd by \cref{corexplo}.
Therefore \cref{obstrSW} is the only obstruction to being in $\Hol(\mathcal P(n_1, \ldots, n_k))$ for nonelementary representations.
\section{Geometrization in minimal degree}\label{last_section}
Let us conclude this article with a short section dedicated to the proof of \cref{dintro}.
For every $\rho\in \Hom(\Gamma, \PSL \C)$, we compute the minimal degree of the branching divisor of a projective structure having $\rho$ as holonomy. 
In other words, we compute the minimal value of $\sum_i n_i$ such that $\rho\in \Hol(\mathcal P(n_1 ,\ldots, n_k))$.
By \cref{corexplo}, this number $d(\rho)$ is the minimal $m\geqslant 0$ such that $\rho\in \Hol(\mathcal P(\underbrace{1, \ldots, 1}_m))$.
This number was computed by Kapovich for representations $\rho\in \Hom(\Gamma, \C)$ with values in the group of translations $\C\subset \mathrm{Aff}(\C)$ in \cite{Kapovich}, and by Gallo Kapovich and Marden for non-elementary representations and the trivial representation in \cite{GKM}.
This computation is a corollary of \cref{mainth} and is given in \cref{fonctiond}.

\bibliography{bibpapier.bib}
\bibliographystyle{alpha}

\end{document}